\documentclass[a4paper,11pt]{article}
\usepackage[
    a4paper,
    top=1.05in,
    bottom=1.05in,
    left=1.00in,
    right=1.00in
]{geometry}
\usepackage[T1]{fontenc}

\usepackage{amsmath,amsthm,mathtools}
\usepackage{amssymb,mathrsfs}
\usepackage{txfonts}
\InputIfFileExists{ursfs.fd}{}{}
\DeclareFontShape{U}{rsfs}{m}{n}{<-6.5>rsfs5<6.5-8>rsfs7<8->rsfs10}{}
\DeclareMathAlphabet{\mathsf}{OT1}{phv}{m}{n}
\SetMathAlphabet{\mathsf}{bold}{OT1}{phv}{b}{n}
\usepackage{microtype}
\microtypesetup{protrusion=true,expansion=true}
\allowdisplaybreaks[4]

\usepackage{enumitem}
\setlist[enumerate]{leftmargin=2.5em,itemsep=0.25em,topsep=0.35em}
\setlist[itemize]{leftmargin=2.2em,itemsep=0.25em,topsep=0.35em}
\usepackage{booktabs,array,longtable}

\usepackage{titlesec,needspace}
\titleformat{\section}{\normalfont\Large\bfseries}{\thesection.}{0.75em}{}
\titleformat{\subsection}{\normalfont\large\bfseries}{\thesubsection}{0.75em}{}
\titlespacing*{\section}{0pt}{2.6ex plus 0.8ex minus 0.2ex}{1.2ex plus 0.2ex}
\titlespacing*{\subsection}{0pt}{2.0ex plus 0.6ex minus 0.2ex}{0.9ex plus 0.2ex}

\usepackage{fancyhdr}
\fancypagestyle{plain}{%
  \fancyhf{}%
  \fancyfoot[C]{\thepage}%
}

\usepackage[dvipsnames]{xcolor}
\definecolor{linkblue}{RGB}{20,67,128}
\definecolor{citegreen}{RGB}{25,105,80}
\definecolor{urlpurple}{RGB}{100,55,125}
\usepackage[colorlinks=true,
            linkcolor=linkblue,
            citecolor=citegreen,
            urlcolor=urlpurple,
            plainpages=false]{hyperref}
\usepackage{bookmark}
\usepackage{aliascnt}
\usepackage[nameinlink,noabbrev,capitalise]{cleveref}
\usepackage{tikz-cd}
\usepackage{multicol}
\let\templatebibliography\thebibliography
\let\endtemplatebibliography\endthebibliography
\renewenvironment{thebibliography}[1]{%
  \templatebibliography{#1}%
  \setlength{\itemsep}{2pt}%
  \setlength{\parsep}{0pt}%
}{\endtemplatebibliography}

\numberwithin{equation}{section}
\makeatletter
\renewcommand{\theequation}{\thesection.\ifnum\value{equation}<10 0\fi\arabic{equation}}
\makeatother

\theoremstyle{plain}
\newtheorem{Theorem}{Theorem}[section]

\theoremstyle{plain}
\newaliascnt{Proposition}{Theorem}
\newtheorem{Proposition}[Proposition]{Proposition}
\aliascntresetthe{Proposition}
\crefname{Proposition}{Proposition}{Propositions}
\Crefname{Proposition}{Proposition}{Propositions}
\newaliascnt{Lemma}{Theorem}
\newtheorem{Lemma}[Lemma]{Lemma}
\aliascntresetthe{Lemma}
\newaliascnt{Corollary}{Theorem}
\newtheorem{Corollary}[Corollary]{Corollary}
\aliascntresetthe{Corollary}
\newaliascnt{Claim}{Theorem}

\aliascntresetthe{Claim}
\newaliascnt{Question}{Theorem}

\aliascntresetthe{Question}
\newaliascnt{Conjecture}{Theorem}

\aliascntresetthe{Conjecture}

\theoremstyle{definition}
\newaliascnt{Definition}{Theorem}
\newtheorem{Definition}[Definition]{Definition}
\aliascntresetthe{Definition}
\newaliascnt{Example}{Theorem}
\newtheorem{Example}[Example]{Example}
\aliascntresetthe{Example}

\theoremstyle{remark}
\newaliascnt{Remark}{Theorem}
\newtheorem{Remark}[Remark]{Remark}
\aliascntresetthe{Remark}

\crefformat{equation}{(#2#1#3)}
\Crefformat{equation}{(#2#1#3)}
\crefrangeformat{equation}{(#3#1#4)--(#5#2#6)}
\Crefrangeformat{equation}{(#3#1#4)--(#5#2#6)}
\crefmultiformat{equation}{(#2#1#3)}{ and~(#2#1#3)}{, (#2#1#3)}{ and~(#2#1#3)}
\Crefmultiformat{equation}{(#2#1#3)}{ and~(#2#1#3)}{, (#2#1#3)}{ and~(#2#1#3)}

\crefformat{enumi}{(#2#1#3)}
\Crefformat{enumi}{(#2#1#3)}
\crefformat{enumii}{(#2#1#3)}
\Crefformat{enumii}{(#2#1#3)}
\crefformat{enumiii}{(#2#1#3)}
\Crefformat{enumiii}{(#2#1#3)}
\crefformat{enumiv}{(#2#1#3)}
\Crefformat{enumiv}{(#2#1#3)}

\DeclareMathOperator{\tr}{tr}

\DeclareMathOperator{\inj}{inj}

\DeclareMathOperator{\End}{End}

\DeclareMathOperator{\Ker}{Ker}
\DeclareMathOperator{\im}{Im}

\DeclareMathOperator{\Supp}{Supp}

\newcommand{\R}{{\mathbb R}}
\newcommand{\Z}{{\mathbb Z}}

\newcommand*\norm[1]{\lVert#1\rVert}

\newcommand*\abs[1]{\lvert#1\rvert}

\newcommand*\inn[2]{\langle#1,#2\rangle}
\renewcommand{\geq}{\geqslant}
\renewcommand{\leq}{\leqslant}

\newcommand{\Rm}{\mathrm{Rm}}
\newcommand{\Ric}{\mathrm{Ric}}

\providecommand{\hg}{\hat g}
\providecommand{\hf}{\hat f}

\hypersetup{
  pdftitle={paper I},
  pdfauthor={Yu Li, Longhuan Yan, and Wenjia Zhang},
  pdfsubject={Weighted analysis on AC Ricci shrinkers},
  pdfkeywords={Ricci shrinkers, AC geometry, weighted analysis, Ricci flow}
}
\title{On the geometry of asymptotically conical \\Ricci shrinkers (I): weighted analysis}
\author{Yu Li \and Longhuan Yan \and Wenjia Zhang}
\date{}

\begin{document}
\maketitle
\begin{abstract}
We develop the weighted Sobolev and H\"older analysis needed to study
perturbations of complete asymptotically conical (AC) Ricci shrinkers.
Weighted elliptic estimates yield a global weighted divergence-free slice
in spaces that retain the asymptotic data at infinity. As a geometric
application, we prove pointed Gromov--Hausdorff rigidity of nontrivial AC
shrinkers with vanishing infinitesimal solitonic deformation space.
\end{abstract}

\tableofcontents

\section{Introduction}\label{sec:introduction}

Ricci shrinkers are self-similar solutions of the Ricci flow and arise as
models of its singularities. A Ricci shrinker is a triple $(M^n,g,f)$
satisfying
\[
    \Ric_g+\nabla_g^2f=\frac12g.
\]
The study of nearby shrinkers leads to questions about their moduli,
rigidity, and the choice of coordinates in which to compare them.

On a compact manifold, these questions fit into the classical framework
of geometric deformation theory. Ebin's slice theorem
\cite{ebin1970manifold} gives a local description of the space of metrics
modulo diffeomorphisms. Koiso \cite{koiso1983einstein} used this framework
to study Einstein deformations. For compact Ricci shrinkers,
Podest\`a--Spiro \cite{podesta2015moduli} constructed a weighted slice
whose tangent space is the kernel of the weighted divergence. In this
slice, the local solitonic pre-moduli space is an analytic subset of a
finite-dimensional real-analytic manifold whose tangent space at the
background shrinker is the space of essential infinitesimal solitonic
deformations. The remaining action
of the background isometry group gives the local moduli space.

This description also relates linearized analysis to rigidity. Vanishing
of the infinitesimal deformation space implies local rigidity, whereas
a nonzero infinitesimal deformation may be obstructed at higher order.
Kr\"oncke \cite{kroencke2016rigidity} showed, for example, that complex
projective spaces of even complex dimension are solitonically rigid
despite admitting infinitesimal solitonic deformations. Sun--Zhu
\cite{sun2025rigidity} proved rigidity of $S^2\times S^2$, and
Li--Zhang \cite{lizhang2023projective} established rigidity of the
odd-dimensional complex projective spaces $\mathbb{CP}^{2m+1}$.
On a compact
manifold, Gaussian weighted Sobolev norms are equivalent to the usual
Sobolev norms. Once the gauge is fixed, elliptic estimates
control the deformation on the whole manifold, and there are no
additional asymptotic coefficients to prescribe or estimate.

Singularity analysis of Ricci flow also requires noncompact shrinkers,
even when the original flow is defined on a closed manifold. Among
these, asymptotically conical (AC) shrinkers form a natural class with
controlled geometry at infinity. An AC shrinker approaches a regular
metric cone $g_C=\mathrm dr^2+r^2g_\Sigma$ outside a compact set, with
convergence in the smooth rescaled geometry of the end. The
Feldman--Ilmanen--Knopf (FIK) shrinkers \cite{feldman2003rotationally} provide
basic non-Gaussian examples. Kotschwar--Wang
\cite{kotschwar2015rigidity} proved that two shrinkers asymptotic to the
same regular cone are isometric near infinity on the corresponding
ends. Perturbations for which the cone varies require control of both
the asymptotic data and the geometry on the whole manifold.

Related questions have a well-developed theory in AC Ricci-flat and
Calabi--Yau geometry. Earlier work of Bando--Kasue--Nakajima
\cite{bando1989coordinates} constructed coordinates at infinity under
fast curvature decay and maximal volume growth. Cheeger--Tian
\cite{cheeger1994cone} studied the cone structure at infinity of
Ricci-flat manifolds with Euclidean volume growth and quadratic
curvature decay. Polynomial weighted spaces and the Fredholm theory
of elliptic operators on noncompact manifolds
\cite{lockhart1985elliptic} describe the allowable rates at infinity.
Conlon--Hein \cite{conlon2013asymptotically} established existence and
uniqueness results for AC Calabi--Yau metrics with prescribed asymptotic
geometry and K\"ahler class. For general Ricci-flat conifolds,
Kr\"oncke--Szab\'o \cite{kroencke2022optimal} computed the indicial roots
of the Lichnerowicz Laplacian and constructed Bianchi gauge slices to
obtain optimal asymptotic coordinates. In these settings, the spectrum
of the link determines the exceptional weights of the cone operators.
Non-Gaussian AC shrinkers have exactly quadratic curvature decay: the
AC upper bound and the scalar-curvature lower bound of
Chow--Lu--Yang \cite{chow2011lower} give
$c r^{-2}\leq|\mathrm{Rm}|\leq C r^{-2}$ for large $r$.
Thus the faster curvature decay used in the asymptotically flat setting
is unavailable for these shrinkers. The radial drift also changes the
asymptotic analysis, as we explain below.

AC Ricci expanders provide a closer comparison: their associated flows
smooth a cone forward in time. Deruelle \cite{deruelle2016smoothing}
developed weighted H\"older and tame Fr\'echet methods for deformations
of AC expanders with positive curvature operator, using Nash--Moser to
control smooth asymptotic data. Bamler--Chen \cite{bamler2023degree} developed a
deformation and degree theory for four-dimensional AC expanding
solitons, using DeTurck gauge and implicit-function arguments in
weighted Banach spaces. For shrinkers, Colding--Minicozzi
\cite{colding2021singularities} studied the noncompact gauge problem and
obtained quantitative control of the resulting diffeomorphisms. Their
analysis of the weighted gauge operator is an important input here.

The purpose of this paper is to develop finite-order weighted Sobolev
and H\"older analysis for AC Ricci shrinkers. Chodosh--Schulze
\cite{chodosh2021uniqueness} established related weighted Sobolev and
H\"older estimates on AC mean curvature self-shrinkers and proved
uniqueness of multiplicity-one AC tangent flows using a
\L{}ojasiewicz--Simon inequality. Some of our arguments follow similar
ideas; in particular, their use of the self-similar flow to obtain
weighted Schauder estimates informs the parabolic argument in
\cref{sec:Schauder-estimates}.

In their setting, nearby hypersurfaces are described as normal graphs
over a fixed self-shrinker. For Ricci shrinkers, comparing nearby metrics
also requires a global choice of gauge. The diffeomorphisms used to fix
the gauge must preserve the prescribed asymptotic structure. We use
Gaussian solvability to define an operator graph and prove the finite
coordinate estimates needed by its nonlinear pullback action. This
permits a direct application of the Banach implicit function
theorem to the forward gauge equation.

The principal nonlinear result is a global weighted divergence free
slice theorem. It places
nearby metrics in a common gauge while preserving the asymptotic
structure required by the weighted equations. As an application, we
prove pointed Gromov--Hausdorff rigidity of AC shrinkers with
vanishing infinitesimal solitonic deformation space, including the
four-dimensional FIK shrinker.

Fix a complete AC shrinker $(M^n,\hat g,\hat f)$, with the potential
normalized by
\[
    (4\pi)^{-n/2}\int_M e^{-\hat f}\,\mathrm dV_{\hat g}=1.
\]
We use the background weighted divergence $\hat\delta_{\hat f}$ and
its formal adjoint $\hat\delta^*=-\operatorname{Sym}\hat\nabla$.
Thus $\mathcal L_{Z^{\sharp_{\hat g}}}\hat g=-2\hat\delta^*Z$.
Let $\mathcal K$ be the finite-dimensional space of Gaussian
$L^2_{\hat f}$ Killing one-forms, and take orthogonal complements with
respect to this Gaussian inner product.

The metric spaces $E_m$ of \cref{def:slice-radial-filtration}
retain tangential coefficients of degrees zero and minus one, a mixed
coefficient of degree minus one, and an ordinary remainder of weight
minus two. In particular, the leading link metric may vary. The gauge
space $\mathscr X_k$ in \cref{ic:graph-definition} is the Gaussian
inverse image of a finite one-form target, with an additional scalar
divergence condition. Its elements have a tangential degree-one term,
a degree-zero term, and a remainder of weight minus one. The coefficients
have the different finite regularities supplied by the equations.

\begin{Theorem}[Weighted divergence free slice]\label{thm:intro-slice}
Let $(M^n,\hat g,\hat f)$ be a complete nontrivial AC Ricci shrinker.
Fix $k\ge6$ and $0<\alpha<1$. For sufficiently small
$h\in E_{k+4}$ there is a unique small
$Z(h)\in\mathscr X_k\subset\mathcal K^\perp$ such that
\begin{equation}\label{eq:intro-slice}
 \hat\delta_{\hat f}\bigl(F_{Z(h)}^*(\hat g+h)-\hat g\bigr)=0,
 \qquad F_Z=\operatorname{Exp}^{\mathrm a}(Z^{\sharp_{\hat g}}).
\end{equation}
The map $h\mapsto Z(h)$ is $C^1$ and
\[
 \|Z(h)\|_{\mathscr X_k}
 +\|F_{Z(h)}^*(\hat g+h)-\hat g\|_{E_{k-1}}
 \le C_k\|h\|_{E_{k+4}}.
\]
Each $F_{Z(h)}$ is a global $C^{k+2,\alpha}_{\mathrm{loc}}$
diffeomorphism isotopic to the identity. Uniqueness is in a fixed small
ball of $\mathscr X_k$, and the constants and ball may depend on $k$.
\end{Theorem}

This is proved in \cref{ic:finite-slice}. Its infinitesimal equation
is governed by
\[
    \hat{\mathcal P}:=\hat\delta_{\hat f}\hat\delta^*.
\]
Writing $P(h,Z)=\hat\delta_{\hat f}
(F_Z^*(\hat g+h)-\hat g)$, we have
$\mathrm D_ZP(0,0)=-2\hat{\mathcal P}$.
The Gaussian Sobolev theory gives solvability on $\mathcal K^\perp$.
On a noncompact manifold, this provides an infinitesimal correction,
but further estimates are needed to obtain a global nonlinear gauge.
A small Gaussian Sobolev norm alone does not control the growth of a
vector field at infinity. The solution must have the asymptotic behavior
needed to define a global diffeomorphism, and pullback by that
diffeomorphism must preserve the metric space.

The weighted Schauder estimates supply this control. To compare them
with the usual AC elliptic estimates, consider a homogeneous function
$r^\gamma a(\theta)$ on the cone. Then
\[
    \Delta_{g_C}(r^\gamma a)
    =r^{\gamma-2}\bigl(\bar\Delta a+\gamma(\gamma+n-2)a\bigr),
\]
where $\bar\Delta$ is the intrinsic Laplacian of $g_\Sigma$.
Thus the indicial degrees of the cone Laplacian are determined by the
link spectrum. In the cylindrical variable $t=\log r$, the rescaled
operator $r^2\Delta_{g_C}$ is translation-invariant and elliptic. The
shrinker model instead takes the form
\begin{equation}\label{eq:intro-cone-drift}
    \Delta_{g_C}-\frac r2\partial_r
    =r^{-2}\bigl((r\partial_r)^2+(n-2)r\partial_r+\bar\Delta\bigr)
     -\frac r2\partial_r.
\end{equation}
On fixed compact regions, the Laplacian gives the usual local
two-derivative gain. On large annuli rescaled to unit size, its
coefficient is $r^{-2}$, while the radial drift remains of order one.
Equivalently, in the variable $t$, the second-order coefficient tends
to zero and the drift is $-\frac12\partial_t$. The asymptotic equation
is therefore governed by this first-order drift.
For $\Delta_f-\lambda$, the leading multiplier on a homogeneous
term of degree $\gamma$ is $-\gamma/2-\lambda$. The characteristic
degree is therefore $-2\lambda$, and the second-order terms produce
successive lower degrees. Our estimates keep track of these coefficients
and the remainder separately. Local elliptic regularity does not in
general give a two-derivative gain for the asymptotic trace; the
counterexample in \cref{rmk:loss-of-boundary-regularity} makes this
failure explicit.

For comparison, write an expander using a positive potential
$f_{\mathrm e}\sim r^2/4$ satisfying
$\nabla^2f_{\mathrm e}=\Ric+\frac12g$. Its drift Laplacian is
$\Delta+\nabla_{\nabla f_{\mathrm e}}$, with the opposite radial sign
and symmetric measure $e^{f_{\mathrm e}}\,\mathrm dV$. General
polynomial perturbations do not belong to this growing weighted
$L^2$ space, which motivates the H\"older approach in
\cite{deruelle2016smoothing}. For shrinkers, the decaying Gaussian
measure accommodates polynomial growth and supplies the Sobolev
inverse. Polynomial weighted H\"older estimates then control the
asymptotic coefficients and the global gauge.

The scalar sum space
\[
 \mathscr S_k=C^{k+1,\alpha}_0+C^{k-1,\alpha}_{-2}
\]
records two kinds of terms produced by the radial equation. It contains
$C^{k,\alpha}_{-1}$ by a finite interpolation estimate. The target
$\mathscr Y_k$ consists of one-forms with a tangential degree-zero
coefficient and a weight-minus-one remainder whose weighted divergence
belongs to $\mathscr S_k$. Gaussian inversion of
$\hat{\mathcal P}$ on $\mathcal K^\perp$ defines the domain
$\mathscr X_k$. Scalar terminal estimates and a positive resolvent lift
then control the asymptotic coefficients and the whole covariant
Hessian of a gauge parameter.

The radial gauge equation gives further estimates for the normalized
radial derivatives. In the forward pullback these occur multiplied by
coefficients controlled in $C^{k+1,\alpha}_0$. This proves closure in
the entire target norm, including its scalar sum component. A finite
reserve of four derivatives on the metric input gives a locally
Lipschitz first differential of the gauge map. Its derivative in $Z$
at zero is $-2\hat{\mathcal P}$, so the implicit function
theorem gives the slice.

We next apply the slice to rigidity. For a complete nontrivial AC
shrinker $(M^n,\hat g,\hat f)$, set
$\hat L=\frac12\hat\Delta_{\hat f}+\Rm_{\hat g}$ and define
\begin{equation}\label{eq:intro-ISD}
    \operatorname{ISD}(\hat g,\hat f)
    :=\{h\in H^2_{\hat f}(S^2T^*M):
        \hat Lh=0,\ \hat\delta_{\hat f}h=0\}.
\end{equation}
This is the infinitesimal solitonic deformation space of
Li--Zhang \cite[Definition~3.18]{li2023rigidity}.

\begin{Theorem}[Pointed Gromov--Hausdorff rigidity]\label{thm:intro-rigidity}
Let $(M^n,\hat g,\hat f)$ be a complete nontrivial AC Ricci shrinker
with $\operatorname{ISD}(\hat g,\hat f)=\{0\}$, and choose
$\hat p\in\arg\min\hat f$. There is a pointed Gromov--Hausdorff
neighborhood of $(M,\hat g,\hat p)$ such that every complete
$n$-dimensional Ricci shrinker $(N,g,f,p)$ in this neighborhood, based
at a minimum point of $f$, has $(N,g)$ isometric to $(M,\hat g)$.
\end{Theorem}

The neighborhood is taken among all complete shrinkers of the given
dimension. In particular, the comparison shrinker is not required
beforehand to be AC or to have the same asymptotic cone. This differs
from uniqueness with a prescribed cone at infinity. The theorem is
proved in \cref{thm:AC-rigidity} by passing from pointed convergence
to a global comparison in the spaces used by the slice.

The compactness theory of Li--Li--Wang
\cite{liliwang2021structure} first gives smooth convergence on compact
sets near the smooth limit. Shrinker pseudolocality
\cite{li2020heat,liwang2024heatII} propagates curvature control from
a large level set along the soliton flow. It follows that sufficiently
nearby shrinkers are AC, and canonical end coordinates extend the
compact comparison to a global diffeomorphism. The asymptotic
expansions then give convergence in every fixed $E_m$; this is the content of
\cref{fik:lem:comparison}.

After applying \cref{thm:intro-slice}, the metric difference is
weighted divergence free. When $\operatorname{ISD}=0$, the
linearized soliton equation has a coercive estimate in this gauge.
For the soliton defect
$\Phi(g,f)=\frac12g-\Ric_g-\nabla_g^2f$, its nonlinear form is
\begin{equation}\label{eq:intro-rigidity-estimate}
    \|h\|_{H^2_{\hat f}}+\|\mathrm du\|_{H^1_{\hat f}}
    \leq C\|\Phi(\hat g+h,\hat f+u)\|_{L^2_{\hat f}},
\end{equation}
provided that $h$ and $\mathrm du$ belong to the indicated Sobolev
spaces, $\hat\delta_{\hat f}h=0$, and $h$ is sufficiently small in
$C^2$; see \cref{fik:prop:rigidity}.
This adapts the rigidity estimate of
Li--Zhang \cite[Theorem~4.11]{li2023rigidity} to the global AC slice.
Only the metric perturbation must be small in $C^2$. This is useful
because the bounded radial displacement of the slice can produce a
degree-one term in the potential difference. The Gaussian norms in
\cref{eq:intro-rigidity-estimate} accommodate this growth. The
nonlinear remainder is absorbed by metric smallness, and an exact
shrinker in the slice has $h=0$ and $\mathrm du=0$.

The four-dimensional FIK shrinker on
$\mathcal O(-1)\to\mathbb{CP}^1$ satisfies the hypothesis
$\operatorname{ISD}=0$. Indeed, Naff--Ozuch
\cite[Theorem~1]{naff2025linear} prove strict negativity of the
entropy Hessian on the subspace orthogonal to the Ricci tensor and
infinitesimal diffeomorphisms. An infinitesimal solitonic deformation
lies in this subspace, so their result gives the following consequence.

\begin{Corollary}[Rigidity of the FIK shrinker]\label{cor:intro-FIK}
Let $(M^4,\hat g,\hat f)$ be the four-dimensional FIK shrinker and
$\hat p\in\arg\min\hat f$. Every complete four-dimensional Ricci
shrinker sufficiently close to $(M,\hat g,\hat p)$ in the pointed
Gromov--Hausdorff topology, with its basepoint chosen at a minimum of
its potential, is isometric to $(M,\hat g)$.
\end{Corollary}

The weighted estimates and the fixed-order slice also provide analytic
inputs for the entropy inequalities and strong uniqueness results in the
sequel \cite{LYZ-II}. The scalar entropy spaces, tensor normal
estimates, and the corresponding metric class are developed there.

The paper is organized as follows. \Cref{sec:preliminaries} recalls
the shrinker identities and canonical AC geometry.
\Cref{sec:weighted-Sobolev} develops weighted Sobolev and H\"older
spaces. \Cref{sec:Schauder-estimates} proves weighted Schauder
estimates and finite asymptotic expansions.
\Cref{sec:fixed-Banach} establishes the finite terminal estimates and
the gauge graph. \Cref{sec:slice-theorem} proves the weighted slice,
its geometric consequences, and the rigidity application.

\medskip
\noindent\textbf{Acknowledgements.}
Yu Li is supported by the National Key Research and Development Program
of China (Grant No.\ 2025YFA1018200), NSFC-12522105, and research funds
from the University of Science and Technology of China and the Chinese
Academy of Sciences.

\medskip
\noindent\textbf{AI declaration.}
This paper originated in a program begun by the authors in early 2025.
The authors developed the overall strategy and main ideas, together with
the corresponding arguments in this paper. Beginning in August 2026,
the authors used ChatGPT to assist with implementing some of the
computations, mathematical exposition, revision, and drafting, in particular
in Section 5. The authors
take full responsibility for the mathematical content and correctness.

\section{Preliminaries}\label{sec:preliminaries}
\subsection{Conventions}
\phantomsection\label{idx:conventions}
We consider a weighted manifold $(M,g,f)$, where $M$ is connected of dimension $n\geq3$, $g$ is a
Riemannian metric, and $f\in C^\infty(M)$.

\begin{Definition}\label{def:weighted-conventions}\label{idx:weighted-sobolev}
    Write $T^{p,q}M=(TM)^{\otimes p}\otimes(T^*M)^{\otimes q}$.
    For a tensor bundle $E$ with its induced metric and connection, set
    \begin{equation*}
        \mathrm{d}V_f=(4\pi)^{-n/2}e^{-f}\,\mathrm{d}V_g,
        \qquad
        \inn{u}{v}_f=\int_M\inn{u}{v}\,\mathrm{d}V_f.
    \end{equation*}
    For $k\geq0$, the weighted Sobolev space $H^k_f(M,E)$ consists of
    sections with weak covariant derivatives through order $k$ and finite norm
    \begin{equation*}
        \norm{u}_{H^k_f}^2
        =\sum_{j=0}^k\int_M|\nabla^j u|^2\,\mathrm{d}V_f.
    \end{equation*}
    Its inner product is defined by the corresponding sum of pairings.
    We set $L^2_f=H^0_f$ and also write $\inn{\cdot}{\cdot}_{L^2_f}=\inn{\cdot}{\cdot}_f$.
    For $q\geq1$, the weighted divergence is
    \begin{equation*}
        (\delta_fu)^{i_1\cdots i_p}_{j_2\cdots j_q}
        =\nabla^a u^{i_1\cdots i_p}_{a j_2\cdots j_q}
        -(\nabla^af)u^{i_1\cdots i_p}_{a j_2\cdots j_q}.
    \end{equation*}
    When $f$ is constant, we omit the subscript and write $\delta$. We set
    $\Delta_f=\Delta-\nabla_{\nabla f}$, where $\Delta=\tr_g\nabla^2$.
    With respect to $\inn{\cdot}{\cdot}_f$, the formal adjoints of
    $\delta_f$ on covariant tensors, differential forms, and symmetric tensors
    are respectively $-\nabla$, $-\mathrm{d}$, and
    $\delta^*:=-\mathrm{Sym}\circ\nabla$.
\end{Definition}
In particular, $\delta_f\mathrm{d}u=\Delta_fu$ for functions, and
\begin{equation*}
    (\delta^*\omega)_{ij}=-\frac12(\nabla_i\omega_j+\nabla_j\omega_i),
    \qquad \mathcal L_{\omega^\sharp}g=-2\delta^*\omega.
\end{equation*}
These conventions agree with \cite[Definition~2.2]{li2023rigidity}.

We use $S*T$ schematically for a linear combination of contractions of
$S\otimes T$ with the metric. For a vector field $X$ and a $(p,q)$-tensor $u$,
\begin{equation}\label{eq:difference-of-Lie-and-covariant-derivative}
    (\mathcal L_X-\nabla_X)u=\nabla X*u.
\end{equation}
In this identity, an endomorphism acts by its negative on each contravariant
slot and by its dual on each covariant slot. On a general weighted manifold, this gives
\begin{equation*}
        (\mathcal L_{\nabla f}-\nabla_{\nabla f})u =\frac{q-p}{2}u-\Ric(g)*u
        -\left(\frac12g-\Ric(g)-\nabla^2f\right)*u.
\end{equation*}
Here symmetric two-tensors act after raising an index with $g$.
In particular, on a Ricci shrinker,
\begin{equation}\label{eq:soliton-Lie-contraction}
    (\mathcal L_{\nabla f}-\nabla_{\nabla f})u
    =\frac{q-p}{2}u-\Ric(g)*u,
\end{equation}
where the Ricci endomorphism acts on the slots in the same way.

\subsection{Ricci shrinkers and the associated Ricci flows}\label{sec:Ricci-flow}
\begin{Definition}
    A Ricci shrinker $(M^n,g,f)$ is a smooth Riemannian manifold $(M,g)$
    with $f\in C^\infty(M)$ satisfying
    \begin{equation}\label{eq:Ricci-shrinker}
        \Ric+\nabla^2f=\frac12g.
    \end{equation}
    Throughout the paper, the potential is normalized by
    \begin{equation}\label{eq:normalization-of-potential}
        \int_M \,\mathrm{d}V_f=1.
    \end{equation}
    With this convention,
    \begin{equation}\label{eq:normalization-of-potential-second}
        \mathrm{R}+|\nabla f|^2=f-\boldsymbol{\mu}(g),
    \end{equation}
    where $\boldsymbol{\mu}(g)$ is the shrinker entropy; see
    \cite[Proposition~3 and \S\,5]{li2020heat} for its properties.
\end{Definition}
Let $(M^n,g,f)$ be a complete Ricci shrinker.\phantomsection\label{idx:associated-Ricci-flow}
Then $\frac{1}{-t}\nabla f$ generates a family of diffeomorphisms $\{\Psi_t\}_{-\infty<t<0}$ with $\Psi_{-1}=\mathrm{id}$.
That is,
\begin{equation*}
	\frac{\partial}{\partial t}\Psi_t(x)
	=
	\frac{1}{-t}\nabla f(\Psi_t(x)).
\end{equation*}
Then $g(t):=(-t)\Psi_t^*g$ is a Ricci flow with $g(-1)=g$.
Set $f(t):=\Psi_t^*f$.
Then
\begin{equation*}
	\Ric(g(t))+\nabla^2_{g(t)}f(t)=\frac{g(t)}{-2t}.
\end{equation*}

Every complete Ricci shrinker has $\mathrm R\geq0$ by
\cite[Corollary~2.5]{chen2009strong}.  Moreover, the strong maximum
principle applied to $\Delta_f\mathrm R=\mathrm R-2|\Ric|^2$ shows that
$\mathrm R>0$ unless $\Ric=0$ and $\nabla^2f=g/2$, in which case the
complete shrinker is Gaussian.  In particular,
$|\nabla f|^2\leq f-\boldsymbol\mu(g)$ throughout.

Recall the potential $f$ of a complete Ricci shrinker always attains its infimum at some point $p_0\in M$ and satisfies the quadratic growth estimate of Cao--Zhou
\cite[Theorem~1.1]{cao2010complete} and Haslhofer--M\"uller
\cite[Lemma~2.1]{haslhofer2011compactness}
\begin{equation}\label{ineq:quadratic-growth-of-potential}
    \frac{1}{4}(d_g(x,p_0)-5n)_+^2
    \leq f(x)-\boldsymbol{\mu}(g)
    \leq \frac{1}{4}(d_g(x,p_0)+\sqrt{2n})^2.
\end{equation}

\subsection{AC Ricci shrinkers}\label{sec:AC-shrinkers}
\phantomsection\label{idx:cone-metric}
Let $(\Sigma^{n-1},g_{\Sigma})$ be a compact, possibly disconnected manifold and $(\mathscr{C},g_C)$ be the Riemannian cone with link $\Sigma$, i.e., $\mathscr{C}=(0,\infty)\times\Sigma$ and $g_C=\mathrm{d}r^2+r^2g_\Sigma$.
For $\lambda>0$, define $D_\lambda:\mathscr{C}\to\mathscr{C}$ as $D_\lambda(r,\theta)=(\lambda r,\theta)$.
Denote the subspace $(\rho,\infty)\times\Sigma$ of $\mathscr{C}$ by $\mathscr{C}_{>\rho}$.

\begin{Definition}
    We say a Riemannian manifold $(M^n,g)$ is \textbf{asymptotic} to the Riemannian cone $(\mathscr{C},g_C)$ if there exists a diffeomorphism $\Phi:\mathscr{C}_{>\rho}\to M\setminus K$ such that
    \begin{equation*}
        \lambda^{-2}D^*_{\lambda}\Phi^*g\to g_C
        \qquad\text{in }C^2_{\mathrm{loc}}(\mathscr{C})
    \end{equation*}
    as $\lambda\to\infty$.
    Here $K$ is a compact subset of $M$ and $\rho>0$ is a constant.
\end{Definition}
The finitely many connected components of $\Sigma$ parametrize the ends of $M$.
For a complete noncompact Ricci shrinker, being AC in this sense is
equivalent to quadratic curvature decay,
\begin{equation*}
    A_0(g):=\sup_{x\in M}|\Rm(g)|_g(x)d_g(p_0,x)^2<\infty,
\end{equation*}
by \cite[\S\,2.2 and Appendix~A]{kotschwar2015rigidity}.
When $A_0(g)<\infty$, the potential is proper and has no critical points
outside a compact set, since $|\nabla f|^2=f-\boldsymbol\mu(g)-\mathrm R$
and the scalar curvature is bounded. A sufficiently high regular level set
therefore has finitely many components, and the cited construction applies
on each corresponding end.
Applying Shi's local derivative estimates \cite{shi1989deforming} to the
self-similar Ricci flow gives
\begin{equation}\label{ineq:quadratic-curvature-decay}
    A_k(g):=\sup_{x\in M}|\nabla^k\Rm(g)|_g(x)d_g(p_0,x)^{2+k}<\infty
\end{equation}
for every $k\geq0$; see \cite[\S\,2.2.3]{kotschwar2015rigidity}.
Note also there is a constant $a(g)>0$ such that $\mathrm{inj}_g(x)\geq ad_g(p_0,x)$.

We call an AC Ricci shrinker \emph{nontrivial}\phantomsection\label{idx:nontrivial-AC} if it is not the Gaussian
soliton. Note that a nontrivial AC shrinker does not split off a line.

On the cone, radial parallel transport defines the extension
\begin{equation}\label{eq:parallel-transport-along-radial-direction}
    \boldsymbol\tau:C^\infty(\Sigma,T^{p,q}\mathscr C|_\Sigma)
    \longrightarrow C^\infty(\mathscr C,T^{p,q}\mathscr C).
\end{equation}
It includes radial as well as tangential components. For example,
$\boldsymbol\tau(a\,\mathrm{d}r+b)=a\,\mathrm{d}r+r\pi^*b$, where
$\pi:\mathscr C\to\Sigma$ is the projection,
$a\in C^\infty(\Sigma)$, and
$b\in C^\infty(\Sigma,T^*\Sigma)$.
\Needspace{4\baselineskip}
\begin{Definition}\label{def:degree}
    A tensor $u=r^\gamma\boldsymbol\tau(\bar u)$ is homogeneous of degree
    $\gamma\in\mathbb R$. For a finite set $\Lambda\subset\mathbb R$, a
    tensor homogeneous of degrees in $\Lambda$ is a sum
    $u=\sum_{\gamma\in\Lambda}u_\gamma$ of such tensors.
\end{Definition}
Covariant differentiation on the cone lowers each homogeneous degree by one.

We have the following gauge at infinity from
\cite[Theorem~3.8 and \S\,3.3]{conlon2024classification}.

\begin{Theorem}[Gauge at infinity]\label{thm:canonical-asymptotic-chart}
    Let $(M^n,g,f)$ be a complete AC Ricci shrinker with asymptotic cone $(\mathscr{C}^n,g_C)$, whose link we denote by $\Sigma^{n-1}$.
    Then there exist a constant $\underline{r}(n,A_0,a)>0$ and a smooth embedding $\iota:(\underline{r},\infty)\times\Sigma\to M$ such that the following holds.
    \begin{enumerate}[label=\textnormal{(\roman{*})},ref=\roman{*}]
        \item $M\setminus\iota((r,\infty)\times\Sigma)$ is compact for all $r\geq\underline{r}$.
        \item \label{thm:canonical-asymptotic-chart-gradient-f} $\iota^*(\nabla^g f)=\frac{1}{2}r\partial_r$.
        \item For all $k\geq 0$ and $r\geq\underline{r}$, we have
        \begin{equation}\label{ineq:asymptotic-rate}
    		\abs{\nabla^{C,k}(\iota^*g-g_C-2\Ric(g_C))}_{g_C}
    		\leq C_k(n,A_0,a)r^{-4-k}.
    	\end{equation}
    \end{enumerate}
\end{Theorem}
Under the above coordinates at infinity, the family of diffeomorphisms $\{\Psi_t\}_{-\infty<t<0}$ generated by $\frac{1}{-t}\nabla^g f$ with $\Psi_{-1}=\mathrm{id}$ takes the form $(r,\theta)\mapsto(\frac{r}{\sqrt{-t}},\theta)$, i.e.,
\begin{equation}\label{eq:diffeomorphisms-under-asymptotic-coordinates}
    \Psi_t(\iota(r,\theta))=\iota(\frac{r}{\sqrt{-t}},\theta).
\end{equation}
This identity holds whenever $r>\underline r$ and
$r/\sqrt{-t}>\underline r$; in particular, it holds on the whole end
for $-1\leq t<0$.
After pulling back by $\iota$, the canonical Ricci flow is
\begin{equation*}
    g(t)=(-t)D_{(-t)^{-1/2}}^*g,\qquad -1\leq t<0.
\end{equation*}
It converges smoothly on compact subsets of the end to $g_C$ as $t\uparrow0$.
Thus the asymptotic cone metric is the time-zero limiting metric of the
self-similar Ricci flow in this gauge.

One may choose $\underline{r}$ so that $\iota^*g$ and $g_C$ are uniformly equivalent, say $\frac{9}{10}g_C\leq\iota^*g\leq\frac{11}{10}g_C$.
Thus when dealing with estimates on tensors, we simply write $|\cdot|$, which can be either $|\cdot|_{\iota^*g}$ or $|\cdot|_{g_C}$.
On the canonical end, write
\begin{equation*}
    \mathbf{X}:=\nabla^g f=\frac r2\partial_r.
\end{equation*}
\phantomsection\label{idx:canonical-radial-field}
We record some consequences of \cref{eq:diffeomorphisms-under-asymptotic-coordinates} as follows.
\begin{Lemma}\label{lem:asymptotic-rate-of-connection}
    Let $T$ be the $(1,2)$-tensor defined by $T_{ij}^{k}:=\Gamma_{ij}^{k}(g)-\Gamma_{ij}^{k}(g_C)$.
    Then for all $k\geq 0$,
    \begin{align*}
        |\nabla^{C,k}T|=O(r^{-3-k}),  \\
        |\nabla^{C,k}(\nabla^g f\lrcorner T+\Ric(g_C))|=O(r^{-4-k}),  \\
        |\nabla^{C,k}(\Rm(g)-\Rm(g_C))|=O(r^{-4-k}),  \\
        |\nabla^{C,k}(f-\boldsymbol{\mu}(g)-\frac{1}{4}r^2)|
		=O(r^{-2-k}).
    \end{align*}
    Here all the implicit constants depend only on $n,A_0,a$ and $k$.
    In the second inequality, $\nabla^gf\lrcorner T$ denotes the $(1,1)$-tensor $T(\nabla^gf,\cdot)$ and $\Ric(g_C)$ stands for the Ricci transformation.
\end{Lemma}
\begin{proof}
    Write $g$ for $\iota^*g$ and use $\nabla^C$ for the cone connection.
    The connection difference satisfies
    \begin{equation*}
        T_{ij}^k=\frac12g^{k\ell}
        (\nabla^C_i g_{j\ell}+\nabla^C_j g_{i\ell}-\nabla^C_\ell g_{ij}).
    \end{equation*}
    The first estimate follows by differentiating this identity and using
    \cref{ineq:asymptotic-rate}.

    We next prove the third estimate. For the covariant curvature tensor,
    the Ricci flow equation gives schematically
    \begin{equation*}
        \partial_t\Rm(g(t))=\Delta^{g(t)}\Rm(g(t))+\Rm(g(t))*\Rm(g(t)).
    \end{equation*}
    Put $Q=\Delta^g\Rm(g)+\Rm(g)*\Rm(g)$, with the contractions in
    the evolution equation. By \cref{ineq:quadratic-curvature-decay}, the
    first estimate, and conversion between $\nabla^g$ and $\nabla^C$,
    \begin{equation*}
        |\nabla^{C,k} Q|_{g_C}\leq C_k r^{-4-k}.
    \end{equation*}
    Constant metric scaling and naturality under pullback give
    $\partial_t\Rm(g(t))=D_\lambda^*Q$, where $\lambda=(-t)^{-1/2}$.
    The cone connection is invariant under $D_\lambda$, and a covariant
    $(4+k)$-tensor acquires the norm factor $\lambda^{4+k}$. Hence
    \begin{equation*}
        |\nabla^{C,k}\partial_t\Rm(g(t))|_{g_C}(r,\theta)
        \leq C_k\lambda^{4+k}(\lambda r)^{-4-k}
        =C_k r^{-4-k}
    \end{equation*}
    uniformly for $-1\leq t<0$. Integrating with the fixed connection
    $\nabla^C$ and using smooth convergence at $t=0$ yields
    \begin{equation*}
        \nabla^{C,k}(\Rm(g_C)-\Rm(g))
        =\int_{-1}^0\nabla^{C,k}\partial_t\Rm(g(t))\,\mathrm{d}t,
    \end{equation*}
    which proves the third estimate for every $k$.

    For the second estimate, use $\nabla^C\mathbf{X}=\operatorname{Id}/2$.
    The soliton equation gives the exact identity
    \begin{equation*}
        T(\mathbf{X},\cdot)=\nabla^g_{(\cdot)}\mathbf{X}-\nabla^C_{(\cdot)}\mathbf{X}
        =-\Ric(g)^\sharp.
    \end{equation*}
    Here the sharp is taken with the corresponding metric.
    Contracting the curvature estimate and using
    $g^{-1}-g_C^{-1}=O(r^{-2})$, with its differentiated versions, gives
    \begin{equation*}
        |\nabla^{C,k}(\Ric(g)^\sharp-\Ric(g_C)^\sharp)|_{g_C}
        \leq C_k r^{-4-k}.
    \end{equation*}
    This proves the second estimate.

    Finally, let $h=g-g_C-2\Ric(g_C)$. Since
    $\Ric(g_C)(\mathbf{X},\cdot)=0$, the soliton identity implies
    \begin{equation*}
        f-\boldsymbol{\mu}(g)-\frac{r^2}{4}
        =\mathrm R_g+h(\mathbf{X},\mathbf{X}).
    \end{equation*}
    We have $|\nabla^{C,j} h|=O(r^{-4-j})$ and
    $|\nabla^{C,j}\mathrm R_g|=O(r^{-2-j})$ for all $j$.
    Moreover, $|\mathbf{X}|_{g_C}=r/2$, $\nabla^C\mathbf{X}=\mathrm{Id}/2$, and
    $\nabla^{C,j}\mathbf{X}=0$ for $j\geq2$.
    In the Leibniz expansion of $\nabla^{C,k}(h(\mathbf{X},\mathbf{X}))$, the terms with
    zero, one, or two derivatives on the $\mathbf{X}$ factors are respectively bounded by
    \begin{equation*}
        C_k r^{-4-k}r^2,\qquad
        C_k r^{-3-k}r,\qquad
        C_k r^{-2-k}.
    \end{equation*}
    All terms are $O(r^{-2-k})$, proving the fourth estimate.
\end{proof}

The following proposition gives the full asymptotic expansion of the metric.
\begin{Proposition}\label{prop:asymptotic-expansion-of-shrinker-metric}
    In the coordinates of \cref{thm:canonical-asymptotic-chart}, there exist homogeneous symmetric $2$-tensors $g_{-2m}$ of degree $-2m$, $m\in\mathbb Z_{>0}$, such that
    \begin{equation}\label{ineq:asymptotic-expansion-of-shrinker-metric}
        \bigg|  \nabla^{C,k}\bigg(  \iota^*g-g_C-\sum_{m=1}^{\ell}g_{-2m}  \bigg)  \bigg|_{g_C}
        \leq C_k(n,\ell,A_0,a)r^{-2-2\ell-k}
    \end{equation}
    for all $\ell,k\geq 0$ and $r\geq\underline{r}(n,A_0,a)$.
    The degree-zero coefficient is $g_C$, and $g_{-2}=2\Ric(g_C)$.
\end{Proposition}
\begin{proof}
    Use the cone dilations $D_s$ and write $g$ for $\iota^*g$ on the end.
    For $t\in[-1,0)$, set
    \begin{equation*}
        \hat g(t):=(-t)D_{(-t)^{-1/2}}^*g.
    \end{equation*}
    By \cref{eq:diffeomorphisms-under-asymptotic-coordinates} and the shrinker equation, $\hat g(t)$ solves
    \begin{equation}\label{eq:rescaled-shrinker-Ricci-flow}
        \partial_t\hat g(t)=-2\Ric(\hat g(t)).
    \end{equation}
    The estimate \cref{ineq:asymptotic-rate}, after applying $D_{(-t)^{-1/2}}$, gives on every compact subannulus of the cone end
    \begin{equation}\label{eq:first-time-expansion-of-rescaled-metric}
        \hat g(t)=g_C+2(-t)\Ric(g_C)+O_{C^k}(t^2)
    \end{equation}
    for every $k\geq0$.
    In particular, $\hat g(t)\to g_C$ in $C^\infty_{\mathrm{loc}}$ as $t\uparrow0$.

    Equation \cref{eq:rescaled-shrinker-Ricci-flow} upgrades this convergence to smooth one-sided convergence in spacetime.
    Indeed, for $t<0$, every $\partial_t^m\hat g(t)$ is a universal expression in $\hat g(t)^{-1}$ and finitely many spatial derivatives of $\hat g(t)$, obtained by differentiating \cref{eq:rescaled-shrinker-Ricci-flow} repeatedly.
    The $C^\infty_{\mathrm{loc}}$ convergence therefore shows inductively that each $\partial_t^m\hat g(t)$ has a $C^\infty_{\mathrm{loc}}$ limit as $t\uparrow0$.
    Thus $\hat g(t)$ extends smoothly to $t=0$ with $\hat g(0)=g_C$.

    Define
    \begin{equation*}
        g_{-2m}:=\frac{(-1)^m}{m!}\partial_t^m\hat g(0),
        \qquad m\geq1.
    \end{equation*}
    The definition of $\hat g(t)$ gives the self-similarity identity
    \begin{equation*}
        D_s^*\hat g(t)=s^2\hat g(t/s^2).
    \end{equation*}
    Differentiating it $m$ times at $t=0$ yields
    \begin{equation*}
        D_s^*g_{-2m}=s^{2-2m}g_{-2m},
    \end{equation*}
    meaning that $g_{-2m}$ is homogeneous of degree $-2m$.

    Taylor's theorem on a fixed compact annulus gives
    \begin{equation*}
        \hat g(-s^{-2})
        =g_C+\sum_{m=1}^{\ell}s^{-2m}g_{-2m}
        +O_{C^k}(s^{-2\ell-2})
        \qquad\text{as }s\to\infty.
    \end{equation*}
    On the other hand, $\hat g(-s^{-2})=s^{-2}D_s^*g$ and
    $s^{-2}D_s^*g_{-2m}=s^{-2m}g_{-2m}$.
    Scaling the last estimate from the fixed annulus to $D_s$ of that annulus gives \cref{ineq:asymptotic-expansion-of-shrinker-metric} as required.
\end{proof}

\section{Weighted Sobolev spaces and H\"older spaces}\label{sec:weighted-Sobolev}
We first collect the weighted Sobolev theory of the operators
$\Delta_f$ and $\delta_f\delta^*$ on Ricci shrinkers.  We follow
\cite[\S\,3]{li2023rigidity} and
\cite[\S\,5.2]{colding2021singularities}, imposing bounded curvature for the
higher-order estimates. We then develop weighted H\"older spaces and their
cone expansions in \cref{sec:weighted-spaces,sec:cone-Holder-spaces}, under
the AC assumption.

\subsection{Weighted Sobolev spaces}
Throughout the first two subsections, $(M^n,g,f)$ is a complete connected Ricci shrinker
with the normalization of \cref{eq:normalization-of-potential}, and $E$ is a
tensor bundle with its induced metric and connection.  We use the spaces
of \cref{def:weighted-conventions} and write $F=f-\boldsymbol\mu(g)$.

Fix a smooth nonincreasing function $\eta:\mathbb R\to[0,1]$ which equals
$1$ on $(-\infty,1]$ and $0$ on $[2,\infty)$, and put
$\eta_R=\eta(F/R)$ for $R\geq1$.  We use these exhaustion cutoffs in the Sobolev arguments;
see \cite[\S\,3, before Lemma~3.3]{li2023rigidity}.  They satisfy
$|\nabla\eta_R|\leq CR^{-1/2}$ and $|\Delta_f\eta_R|\leq C$.
If the curvature is bounded, its covariant derivatives are bounded by
Shi's estimates, and the soliton equation also gives
$|\nabla^j\eta_R|\leq C_jR^{-1}$ for $j\geq2$.
In particular, $C_c^\infty(M,E)$ is dense in $H^1_f(M,E)$, and in every
$H^j_f(M,E)$ under bounded curvature; see
\cite[Remark~3.4]{li2023rigidity}.

The following Poincar\'{e} type inequality is the case $k=1$ of
\cite[Lemma~3.6]{li2023rigidity}.  It is the Ricci flow analogue of
\cite[Proposition~3.9]{chodosh2021uniqueness}; see also
\cite[Lemma~2.10]{colding2021singularities} and
\cite[Lemma~5.41]{bamler2023degree}.
\begin{Lemma}\label{lem:Poincare-inequality}
    In the above setting, every $u\in H^1_f(M,E)$ satisfies
    \begin{equation}\label{ineq:Poincare}
        \int_M F|u|^2\,\mathrm{d}V_f
        \leq \int_M\bigl(n|u|^2+4|\nabla u|^2\bigr)\,\mathrm{d}V_f.
    \end{equation}
    For $q\geq1$, the operator
    $\delta_f:H^1_f(M,T^{p,q}M)\to L^2_f(M,T^{p,q-1}M)$ is continuous.
    If the curvature is bounded, then
    $\delta_f:H^{k+1}_f(M,T^{p,q}M)\to H^k_f(M,T^{p,q-1}M)$ is continuous
    for every integer $k\geq0$.
\end{Lemma}
\begin{proof}
    First let $u\in C_c^\infty(M,E)$.  The soliton identities and
    $\mathrm R\geq0$ give $\Delta_fF=n/2-F$ and $|\nabla f|^2\leq F$.
    Integration by parts and Young's inequality yield
    \begin{align*}
        \int_M(\Delta_fF)|u|^2\,\mathrm{d}V_f
        &=-\int_M\inn{\nabla f}{\nabla|u|^2}\,\mathrm{d}V_f\\
        &\geq-\frac12\int_M F|u|^2\,\mathrm{d}V_f
               -2\int_M|\nabla u|^2\,\mathrm{d}V_f.
    \end{align*}
    Substituting $\Delta_fF=n/2-F$ proves \cref{ineq:Poincare}.
    For general $u\in H^1_f$, choose $u_j\in C_c^\infty(M,E)$ converging to
    $u$ in $H^1_f$.  Applying the inequality to $u_j-u_\ell$ shows that
    $F^{1/2}u_j$ is Cauchy in $L^2_f$; its limit is $F^{1/2}u$, as follows
    by taking an almost-everywhere convergent subsequence.  Passing to the
    limit proves the assertion.

    The estimate for $\delta_f:H^1_f\to L^2_f$ follows from
    $|\delta_fu|\leq C|\nabla u|+|\nabla f||u|$ and
    \cref{ineq:Poincare}.  To prove the higher-order assertion, the Leibniz
    rule gives
    \begin{equation*}
        |\nabla^k(\delta_fu)|
        \leq C_k\left(|\nabla^{k+1}u|+|\nabla f||\nabla^ku|
        +\sum_{a=0}^{k-1}|\nabla^{k+1-a}f||\nabla^au|\right).
    \end{equation*}
    Bounded curvature and the soliton equation bound all the derivatives
    of $f$ in the sum.  Applying \cref{ineq:Poincare} to $\nabla^ku$ gives
    \begin{equation*}
        \bigl\||\nabla f||\nabla^ku|\bigr\|_{L^2_f}
        \leq C\|\nabla^ku\|_{H^1_f}.
    \end{equation*}
    Summing over derivative orders and using density proves
    $\|\delta_fu\|_{H^k_f}\leq C_k\|u\|_{H^{k+1}_f}$.
\end{proof}

The preceding estimate supplies the confinement which replaces compactness
of the underlying manifold.
\begin{Lemma}[Gaussian Rellich compactness]\label{lem:Gaussian-Rellich}
    In the above setting, $H^{j+1}_f(M,E)\Subset H^j_f(M,E)$ for every
    integer $j\geq0$.
\end{Lemma}
\begin{proof}
    For $R>1$, \cref{lem:Poincare-inequality} gives
    \begin{equation}\label{eq:Gaussian-tail-compactness}
        \int_{\{F\geq R\}}|u|^2\,\mathrm{d}V_f
        \leq \frac1R\int_M F|u|^2\,\mathrm{d}V_f
        \leq \frac{C(n)}R\|u\|_{H^1_f}^2.
    \end{equation}
    The sublevel sets of $F$ are compact by
    \cref{ineq:quadratic-growth-of-potential}.  Local Rellich compactness
    and a diagonal subsequence give convergence in $L^2_f$ on each compact
    sublevel set for a bounded sequence in $H^1_f$.  The displayed estimate
    makes the complementary tails uniformly small, proving the result
    for $j=0$.  For general $j$, apply the same estimate to
    $\nabla^\ell u$, $0\leq\ell\leq j$, and combine these tail bounds with
    the local compact inclusion $H^{j+1}_{\mathrm{loc}}\Subset
    H^j_{\mathrm{loc}}$.
\end{proof}

We use the word ``solution'' to mean a distributional solution.  In
particular, if $u,v\in L^2_f(M,E)$ and $\Delta_fu=v$, then
$u\in H^1_f(M,E)$ and weighted integration by parts is valid by
\cite[Lemma~3.3]{li2023rigidity}.

The next proposition collects
\cite[Propositions~3.7 and~3.8]{li2023rigidity} and the resulting spectral
solvability criterion.
\begin{Proposition}\label{prop:elliptic-theory-for-weighted-Laplacian}
    In the above setting, the following statements hold.
    \begin{enumerate}[label=\textnormal{(\roman{*})},ref=\roman{*}]
        \item The nonnegative self-adjoint operator $-\Delta_f$ on
        $L^2_f(M,E)$ has discrete spectrum
        $0\leq\mu_0<\mu_1<\cdots$, with $\mu_i\to\infty$ and
        finite-dimensional eigenspaces.  Its $L^2_f$ kernel consists of
        parallel tensors.
        \item For $v\in L^2_f(M,E)$, the equation $\Delta_fu=v$ has an
        $L^2_f$ solution if and only if $v$ is $L^2_f$-orthogonal to the
        kernel of $\Delta_f$.
        \item Suppose additionally that the curvature is bounded.  If
        $u\in L^2_f(M,E)$ and $\Delta_fu\in H^k_f(M,E)$, where $k\geq0$
        is an integer, then $u\in H^{k+2}_f(M,E)$ and
        \begin{equation}\label{eq:Delta-f-Sobolev-graph-estimate}
            \|u\|_{H^{k+2}_f}
            \leq C_k\bigl(\|\Delta_fu\|_{H^k_f}+\|u\|_{L^2_f}\bigr).
        \end{equation}
        The constants depend on the dimension, the tensor type, the
        derivative order, and the curvature bound.  In particular, the
        operator domain of $-\Delta_f$ is $H^2_f(M,E)$.
    \end{enumerate}
\end{Proposition}

A compact perturbation gives the nonsymmetric Fredholm theorem needed later.
\begin{Proposition}[Gaussian Fredholm realization]\label{prop:Gaussian-Fredholm-realization}
    In the above setting, assume that the curvature is bounded and let
    \begin{equation*}
        L=\Delta_f-\lambda+\nabla_X+\phi,
    \end{equation*}
    where $\lambda\in\mathbb R$, and $X$ and $\phi$ are bounded $C^1$
    sections of $TM$ and $\End(E)$, respectively.  The maximal
    distributional realization is
    $L:H^2_f(M,E)\to L^2_f(M,E)$ and is Fredholm of index zero.
    If $L^*$ is its Hilbert-space adjoint, then
    \begin{equation}\label{eq:Gaussian-Fredholm-alternative}
        \operatorname{Ran}L=(\Ker L^*)^\perp,
        \qquad
        \|u\|_{H^2_f}\leq C\bigl(\|Lu\|_{L^2_f}+\|u\|_{L^2_f}\bigr).
    \end{equation}
\end{Proposition}
\begin{proof}
    By \cref{lem:Gaussian-Rellich}, the operator
    $\nabla_X+\phi:H^2_f\to L^2_f$ factors through the compact inclusion
    $H^2_f\Subset H^1_f$.  Thus $L$ is a compact perturbation of
    the index-zero Fredholm operator $\Delta_f-\lambda$ supplied by
    \cref{prop:elliptic-theory-for-weighted-Laplacian}.  This proves the
    Fredholm and range assertions.  The estimate follows from
    \cref{eq:Delta-f-Sobolev-graph-estimate} by interpolation and absorption
    of the first-order term.

    To identify the maximal domain, suppose $u,Lu\in L^2_f$
    distributionally.  Local elliptic regularity for $C^1$ lower-order
    coefficients permits testing against $\eta_R^2u$.
    Integration by parts and absorption give
    \begin{equation*}
        \|\eta_R\nabla u\|_{L^2_f}^2
        \leq C\bigl(\|\eta_RLu\|_{L^2_f}^2
        +\|\eta_Ru\|_{L^2_f}^2+\|u\nabla\eta_R\|_{L^2_f}^2\bigr).
    \end{equation*}
    Letting $R\to\infty$ gives $u\in H^1_f$.  Consequently
    $\Delta_fu=Lu+\lambda u-\nabla_Xu-\phi u\in L^2_f$, and
    \cref{prop:elliptic-theory-for-weighted-Laplacian} gives $u\in H^2_f$.
    Conversely, \cref{lem:Poincare-inequality}, applied to $\nabla u$,
    shows that $\Delta_f:H^2_f\to L^2_f$ is bounded, so every $H^2_f$
    section belongs to the maximal domain.
\end{proof}

\begin{Corollary}\label{prop:Sobolev-Helmholtz-decomposition}
    In the above setting, assume that the curvature is bounded.  For every
    integer $k\geq1$, there is an $L^2_f$-orthogonal decomposition
    \begin{equation}\label{eq:Sobolev-Helmholtz-decomposition}
        H^k_f(M,T^*M)
        =\im\bigl(\mathrm{d}\restriction H^{k+1}_f(M)\bigr)
        \oplus\Ker\bigl(\delta_f\restriction H^k_f(M,T^*M)\bigr).
    \end{equation}
    Thus $\omega=\mathrm{d}v+\zeta$, where $v\in H^{k+1}_f(M)$ and
    $\zeta\in H^k_f(M,T^*M)$ satisfies $\delta_f\zeta=0$.
    The projection onto the first summand is
    $\mathrm{d}\Delta_f^{-1}\delta_f$, where $\Delta_f^{-1}$ is the inverse
    on functions orthogonal to constants.
\end{Corollary}
\begin{proof}
    By \cref{lem:Poincare-inequality}, $\delta_f\omega\in H^{k-1}_f(M)$.
    The kernel of $\Delta_f$ on functions consists of constants, and
    weighted integration by parts gives
    $\int_M\delta_f\omega\,\mathrm{d}V_f=0$.
    Hence \cref{prop:elliptic-theory-for-weighted-Laplacian} gives a
    mean-zero $v\in H^{k+1}_f(M)$ with $\Delta_fv=\delta_f\omega$.
    Then $\zeta=\omega-\mathrm{d}v$ has $\delta_f\zeta=0$.
    Weighted adjointness makes the two summands orthogonal and proves the
    projection formula.
\end{proof}

\subsection{Gauge operator}
Set $\mathcal P:=\delta_f\delta^*$.  On smooth one-forms, the soliton
identities give
\begin{equation}\label{eq:gauge-operator}
    \mathcal P=-\frac12\Delta_f-\frac12\mathrm{d}\delta_f-\frac14.
\end{equation}
Moreover, $\Delta_f\mathrm{d}u=\mathrm{d}(\Delta_f+\frac12)u$ for every
smooth function $u$, and
$\delta_f\Delta_f\omega=(\Delta_f+\frac12)\delta_f\omega$ for every
smooth one-form $\omega$; see
\cite[Lemma~3.1 and Proposition~3.3]{colding2021singularities}.
When the curvature is bounded, $\mathcal P$ preserves the decomposition
\cref{eq:Sobolev-Helmholtz-decomposition} and, for $k\geq1$, acts as
\begin{equation}\label{eq:decomposition-of-gauge-operator}
    \begin{split}
        \bigl(-\Delta_f,-\tfrac12(\Delta_f+\tfrac12)\bigr):
        \im\bigl(\mathrm{d}\restriction H^{k+3}_f(M)\bigr)
        &\oplus\Ker\bigl(\delta_f\restriction H^{k+2}_f(M,T^*M)\bigr)\\
        \longrightarrow\quad
        \im\bigl(\mathrm{d}\restriction H^{k+1}_f(M)\bigr)
        &\oplus\Ker\bigl(\delta_f\restriction H^k_f(M,T^*M)\bigr).
    \end{split}
\end{equation}

Define
\begin{equation}\label{eq:Gaussian-Killing-space}
    \mathcal K:=\{\omega\in L^2_f(M,T^*M):
                    \delta^*\omega=0\text{ distributionally}\},
\end{equation}
and let $\mathcal K^\perp$ be its $L^2_f$-orthogonal complement.
The following collects the spectral and first-order results of
\cite[Lemma~3.5, Theorem~5.5, and Lemma~5.7]{colding2021singularities}
and the higher-order estimates of
\cite[Propositions~3.7 and~3.10]{li2023rigidity}.
\begin{Proposition}
\label{prop:elliptic-theory-for-gauge-operator}
    In the above setting, $\mathcal P$ is nonnegative and self-adjoint on
    $L^2_f(M,T^*M)$, has compact resolvent, and has kernel $\mathcal K$.
    The space $\mathcal K$ is finite-dimensional and consists of smooth
    Killing forms in $H^1_f$.
    For each $v\in L^2_f\cap\mathcal K^\perp$, there is a unique
    $L^2_f$ solution $\omega\in\mathcal K^\perp$ of $\mathcal P\omega=v$;
    it satisfies
    \begin{equation*}
        \|\omega\|_{H^1_f}\leq C\|v\|_{L^2_f}.
    \end{equation*}

    If, in addition, the curvature is bounded, then the operator domain
    is $H^2_f(M,T^*M)$, $\mathcal K\subset H^j_f$ for every $j\geq0$, and
    for every integer $k\geq0$ the map
    \begin{equation}\label{eq:Gaussian-gauge-complement-isomorphism}
        \mathcal P:H^{k+2}_f(M,T^*M)\cap\mathcal K^\perp
        \longrightarrow H^k_f(M,T^*M)\cap\mathcal K^\perp
    \end{equation}
    is a bounded isomorphism.  In particular,
    \begin{equation}\label{ineq:Gaussian-gauge-complement-estimate}
        \|\omega\|_{H^{k+2}_f}\leq C_k\|\mathcal P\omega\|_{H^k_f},
        \qquad \omega\in H^{k+2}_f\cap\mathcal K^\perp.
    \end{equation}
\end{Proposition}

\begin{Proposition}[Linear growth of Killing fields]\label{prop:Killing-linear-growth}
    In the above setting, every Killing field $Y$ satisfies
    \begin{equation}\label{eq:Killing-linear-growth}
        |Y|^2\leq C\bigl(1+f-\boldsymbol\mu(g)\bigr),
    \end{equation}
    with $C$ depending on $Y$ and the
    shrinker.  In particular, $Y\in L^2_f(M,TM)$, and $\mathcal K$ is the
    space of all smooth Killing one-forms.  This conclusion holds without a curvature bound.
\end{Proposition}
\begin{proof}
    The compact case is immediate, so assume that $M$ is noncompact.
    We first separate the part of $Y$ that moves the potential.  Since a
    Killing field preserves the metric, the connection, and the Ricci
    tensor, taking the Lie derivative of the soliton equation gives
    $\nabla^2(Yf)=0$.  Thus $V:=\nabla(Yf)$ is parallel.
    Applying $Y$ to $\mathrm R+|\nabla f|^2=F$, and using $Y\mathrm R=0$,
    gives
    \begin{equation*}
        Yf=2\inn{V}{\nabla f},\qquad Y_0:=Y-2V,\qquad Y_0f=0.
    \end{equation*}
    The field $Y_0$ is Killing and $|V|$ is constant.  It therefore suffices
    to prove the estimate when $Yf=0$, which we now assume.

    Differentiating $\inn{Y}{\nabla f}=0$ and using the skew-symmetry of
    $\nabla Y$ yields $\nabla_{\nabla f}Y=(\nabla^2f)^\sharp Y$.
    Together with the Killing identity $\Delta Y=-\Ric^\sharp Y$, this gives
    \begin{equation*}
        \Delta_fY=-\frac12Y,\qquad
        \Delta_fv=2|\nabla Y|^2-v,\qquad v:=|Y|^2.
    \end{equation*}
    At a point, let $A(Z)=\nabla_ZY$.  This endomorphism is skew-adjoint,
    $\nabla v=-2AY$, and its orthogonal decomposition into rotation blocks
    gives $|AY|^2\leq\frac12|A|^2|Y|^2$.  Consequently,
    \begin{equation}\label{eq:Killing-improved-Kato}
        |\nabla v|^2\leq2v|\nabla Y|^2.
    \end{equation}

    We next establish the integrability needed for comparison at infinity.
    Since $1\in H^1_f(M)$, \cref{lem:Poincare-inequality} gives
    $\int_M(1+F)\,\mathrm{d}V_f<\infty$, and hence $\nabla f\in L^2_f$.
    The commutation identity
    $\Delta_f\nabla f=-\frac12\nabla f$, together with
    \cite[Lemma~3.3]{li2023rigidity}, yields
    \begin{equation*}
        \int_M|\nabla^2f|^2\,\mathrm{d}V_f
        =\frac12\int_M|\nabla f|^2\,\mathrm{d}V_f<\infty.
    \end{equation*}
    Put $h=\frac12\log(1+v)$.  The identity for
    $\nabla_{\nabla f}Y$ gives
    \begin{equation*}
        \bigl|\inn{\nabla f}{\nabla h}\bigr|
        =\frac{|\nabla^2f(Y,Y)|}{1+|Y|^2}\leq|\nabla^2f|.
    \end{equation*}
    For $T>0$, let $h_T=\min\{h,T\}$.  Integrating the weighted divergence
    of $\eta_Rh_T^2\nabla f$ and then letting $R\to\infty$ gives
    \begin{equation*}
        \int_M(F-n/2)h_T^2\,\mathrm{d}V_f
        =2\int_M h_T\inn{\nabla f}{\nabla h_T}\,\mathrm{d}V_f.
    \end{equation*}
    Indeed, for fixed $T$ the two integrands are integrable, while the
    cutoff error tends to zero because
    $|\inn{\nabla\eta_R}{\nabla f}|\leq C$ on $\{R\leq F\leq2R\}$ and
    this annulus has weighted measure tending to zero.  Young's inequality
    therefore implies
    \begin{equation*}
        \int_M(F-n/2-1)h_T^2\,\mathrm{d}V_f
        \leq\int_M|\nabla^2f|^2\,\mathrm{d}V_f.
    \end{equation*}
    On $\{F\geq n/2+2\}$ the coefficient on the left is at least $1$;
    on the complementary compact set, $h_T\leq h$ is uniformly bounded.
    Thus $\|h_T\|_{L^2_f}$ is bounded independently of $T$, and monotone
    convergence gives $h\in L^2_f$.

    The chain rule and \cref{eq:Killing-improved-Kato} now give
    \begin{equation*}
        \Delta_fh
        =\frac{|\nabla Y|^2}{1+v}-\frac{v}{2(1+v)}
          -\frac{|\nabla v|^2}{2(1+v)^2}
        \geq-\frac12.
    \end{equation*}
    On $\{F>n/2\}$, we also have
    \begin{equation*}
        \Delta_f\left(\frac12\log(F-n/2)\right)
        =-\frac12-\frac{|\nabla f|^2}{2(F-n/2)^2}\leq-\frac12.
    \end{equation*}
    Choose $B=1+\max_{\{F\leq n/2+3\}}h$ and, on $\Omega=\{F>n/2+2\}$, set
    $w=h-B-\frac12\log(F-n/2)$.  Then $\Delta_fw\geq0$, and $w<0$ on
    $\{n/2+2<F\leq n/2+3\}$.  Hence $w_+$ extends by zero to a locally
    Lipschitz function on $M$, with $0\leq w_+\leq h$ and $w_+\in L^2_f$.
    Testing the inequality against $\eta_R^2w_+$, which vanishes near the
    inner boundary of $\Omega$, gives
    \begin{equation*}
        \int_M\eta_R^2|\nabla w_+|^2\,\mathrm{d}V_f
        \leq4\int_M w_+^2|\nabla\eta_R|^2\,\mathrm{d}V_f
        \leq\frac CR\|h\|_{L^2_f}^2.
    \end{equation*}
    Fatou's lemma gives $\nabla w_+=0$.  Since $M$ is connected and $w_+$
    vanishes near a minimum point of $F$, whose minimum is at most $n/2$,
    we have $w_+=0$.  Exponentiating the resulting inequality yields
    $1+|Y|^2\leq e^{2B}(F-n/2)$ on $\Omega$.  Absorbing the compact
    complement into the constant proves $|Y|^2\leq C(1+F)$.

    For the original Killing field, use $Y=Y_0+2V$ and the constant norm
    of $V$ to obtain the same estimate.  Finally,
    $\int_M(1+F)\,\mathrm{d}V_f<\infty$ gives $Y\in L^2_f$.
\end{proof}

\subsection{Weighted H\"{o}lder spaces}\label{sec:weighted-spaces}\label{sec:Holder-spaces}
\begin{Definition}[Standard AC setting]\label{def:standard-AC-setting}
    The following fixed data will be called the \emph{standard AC setting}.
    \begin{enumerate}[label=\textnormal{(\roman{*})},ref=\roman{*}]
        \item $(M^n,g,f)$ is a complete AC Ricci shrinker, with asymptotic cone $(\mathscr C(\Sigma),g_C)$.
        \item The map $\iota:(\underline r,\infty)\times\Sigma\to M$ is the canonical end embedding from \cref{thm:canonical-asymptotic-chart}. We increase $\underline r$ if necessary so that the chart extends across the initial slice, and write
        \begin{equation}\label{eq:positive-potential-weight}
            Q_\rho:=\iota((\rho,\infty)\times\Sigma),
            \qquad
            \boldsymbol{\rho}:=2(1+f-\boldsymbol{\mu}(g))^{1/2}.
        \end{equation}
        \item The function $\chi:M\to[0,1]$ is a fixed smooth cut-off satisfying
        \begin{equation*}
            \Supp\chi\subset Q_{\underline r},
            \qquad
            \chi=1\text{ on }Q_{2\underline r},
            \qquad
            |\nabla^j\chi|\leq C_j\underline r^{-j}\qquad(j\geq0).
        \end{equation*}
    \end{enumerate}
    We also fix a compact domain $K$ containing $M\setminus Q_{2\underline r}$.
\end{Definition}
Throughout this and the next subsection, constants may depend on the displayed regularity and weight parameters, the tensor bundles and algebraic operations involved, and the fixed shrinker, end chart, compact core, and cutoff.

We define the weighted H\"older spaces using ordinary H\"older norms on a cylinder. This is equivalent to the usual definition on AC manifolds; see \cite[$\S 8.3$]{joyce2000compact}.
On the canonical end, put
\begin{equation}\label{eq:radial-compactifying-variable}
    t:=\log\frac r{\underline r},\qquad r=\underline r e^t.
\end{equation}
\phantomsection\label{idx:radial-parallel-identification}
For $r>0$, write $\Sigma_r=\{r\}\times\Sigma$.
For $E=T^{p,q}M$, let $E_\Sigma=T^{p,q}\mathscr C|_{\Sigma_1}$, with the metric and connection induced from the cone. We use the same notation for its pullback to the cylinder. Radial parallel transport for $g_C$ gives the fiberwise identifications $\boldsymbol\tau_r:E_\Sigma\to T^{p,q}\mathscr C|_{\Sigma_r}$.
After identifying the end with its image under $\iota$, define
\begin{equation}\label{eq:compactified-weighted-representative}
    U(t,\theta):=(\mathscr T_\beta u)(t,\theta)
    =r^{-\beta}\boldsymbol\tau_r^{-1}u(r,\theta).
\end{equation}
Thus all cylindrical representatives take values in the same fixed bundle.
Ordinary H\"older norms on the cylinder use the product metric $\mathrm dt^2+g_\Sigma$ and the fixed bundle connection. Set
\begin{equation}\label{eq:uniform-cylindrical-Holder-norm}
    \norm{U}_{C^{k,\alpha}_{\mathrm{cyl}}}
    :=\sup_{T\geq0}\norm{U(T+\cdot,\cdot)}_{C^{k,\alpha}([0,1]\times\Sigma,E_\Sigma)}.
\end{equation}

\begin{Definition}[Weighted H\"older spaces]\label{def:weighted-Holder-spaces}
    For $k\geq0$, $\alpha\in[0,1]$, and $\beta\in\mathbb R$, let $C^{k,\alpha}_\beta(M,E)$ be the space of locally $C^{k,\alpha}$ sections with finite norm
    \begin{equation}\label{eq:weighted-Holder-cylindrical-norm}
        \norm{u}_{C^{k,\alpha}_\beta(M,E)}
        :=\norm{u}_{C^{k,\alpha}(K,E)}
        +\norm{\mathscr T_\beta u}_{C^{k,\alpha}_{\mathrm{cyl}}}.
    \end{equation}
    At $\alpha=0$, the ordinary norms in this formula are $C^k$ norms; we omit $\alpha$ and write $C^k_\beta=C^{k,0}_\beta$. At $\alpha=1$, $C^{k,1}$ means that the derivatives of order $k$ are Lipschitz.
\end{Definition}
These are Banach spaces. Changing the compact core, the fixed connection on $E_\Sigma$, or the positive length of the cylinder slabs gives equivalent norms.
An equivalent description uses the shrinker metric on rescaled annuli. Put
\begin{equation*}
    R_T:=\underline r e^T,\qquad
    \kappa_T(s,\theta):=\iota(R_Te^s,\theta),\qquad
    g_T:=R_T^{-2}\kappa_T^*g.
\end{equation*}
For $E=T^{p,q}M$, one has
\begin{equation}\label{eq:weighted-Holder-rescaled-annulus-norm}
    \norm{u}_{C^{k,\alpha}_\beta}
    \asymp\norm{u}_{C^{k,\alpha}(K,E),g}
    +\sup_{T\geq0}\norm{R_T^{p-q-\beta}\kappa_T^*u}
      _{C^{k,\alpha}([0,1]\times\Sigma),g_T}.
\end{equation}
Indeed, the metrics $g_T$ are uniformly equivalent to the product cylinder metric and have uniformly bounded derivatives of every order. Identify the radial vector at $r=1$ with $\partial_s$ and retain the tangential vectors on $\Sigma$. Under the resulting identification of $E_\Sigma$ with the cylinder tensor bundle, $R_T^{p-q-\beta}\kappa_T^*u$ is $e^{(\beta+q-p)s}\mathscr T_\beta u(T+s,\theta)$. The factor and its inverse have uniformly bounded derivatives for $s\in[0,1]$, which proves the comparison, including both endpoint conventions.

On a slab where $r\asymp R$, covariant differentiation corresponds, after tensor rescaling, to $R^{-1}$ times a first-order cylindrical operator with uniformly bounded coefficients. Consequently,
\begin{equation*}
    \norm{u}_{C^k_\beta}\asymp\sum_{j=0}^k\sup_M\boldsymbol{\rho}^{j-\beta}|\nabla^j u|_g.
\end{equation*}
In particular, membership in $C^k_\beta$ means that $|u|=O(\boldsymbol{\rho}^\beta)$, $|\nabla u|=O(\boldsymbol{\rho}^{\beta-1})$, and so on through order $k$. Formula \cref{eq:weighted-Holder-rescaled-annulus-norm} gives the corresponding conically scaled H\"older control.
Here $r$ is the cone coordinate and $\boldsymbol{\rho}$ is the global weight. By \cref{lem:asymptotic-rate-of-connection},
\begin{equation*}
    \boldsymbol{\rho}=r+O(r^{-1}),\qquad
    \left|\nabla^{C,j}\left(\frac{\boldsymbol{\rho}}r-1\right)\right|
    =O(r^{-2-j})\qquad(j\geq0).
\end{equation*}

\begin{Lemma}\label{lem:Holder-to-Gaussian-Sobolev}
    In the standard AC setting, for every $\beta\in\mathbb R$, integer $k\geq0$, and $\alpha\in[0,1]$, the inclusion
    \begin{equation}\label{eq:Holder-to-Gaussian-Sobolev}
        C^{k,\alpha}_\beta(M,E)\hookrightarrow H^k_f(M,E)
    \end{equation}
    is continuous.
\end{Lemma}
\begin{proof}
    The preceding norm comparison gives
    $|\nabla^j u|\leq C\boldsymbol{\rho}^{\beta-j}\norm{u}_{C^{k,\alpha}_\beta}$ for $0\leq j\leq k$.
    Since $f=\boldsymbol{\rho}^2/4-1+\boldsymbol\mu(g)$ and AC volume growth is polynomial, each $\boldsymbol{\rho}^{2(\beta-j)}$ is integrable with respect to $\mathrm dV_f$.
    Integrating these bounds and summing over $j$ proves the assertion.
\end{proof}

We record basic properties that will be used repeatedly. They are the conically scaled analogues of \cite[Lemma 5.7]{bamler2023degree}.
\begin{Lemma}\label{lem:basic-properties}
    Assume the standard AC setting. Let $E,E'$ be tensor bundles, $k\in\mathbb Z_{\ge0}$, $\alpha\in[0,1]$, and $\beta,\beta'\in\mathbb R$.
    \begin{enumerate}[label=\textnormal{(\roman{*})},ref=\roman{*}]
        \item\label{lem:basic-properties-continuity-of-covariant-derivative}
        If $k\geq1$, then
        \begin{equation*}
            \norm{u}_{C^{k,\alpha}_\beta}
            \asymp\norm{u}_{C^0_\beta}+\norm{\nabla u}_{C^{k-1,\alpha}_{\beta-1}}.
        \end{equation*}
        In particular, covariant differentiation is bounded from $C^{k,\alpha}_\beta$ to $C^{k-1,\alpha}_{\beta-1}$.
        \item\label{lem:basic-properties-isomorphism}
        Multiplication by $\boldsymbol{\rho}^{\beta'-\beta}$ is a bounded isomorphism from $C^{k,\alpha}_\beta(M,E)$ to $C^{k,\alpha}_{\beta'}(M,E)$, with
        \begin{equation*}
            \norm{\boldsymbol{\rho}^{\beta'-\beta}u}_{C^{k,\alpha}_{\beta'}}\leq C\norm{u}_{C^{k,\alpha}_\beta}.
        \end{equation*}
        \item\label{lem:basic-properties-multiplication}
        For a fiberwise bilinear algebraic operation $*$ and sections $u\in C^{k,\alpha}_\beta(M,E)$, $u'\in C^{k,\alpha}_{\beta'}(M,E')$,
        \begin{equation*}
            \norm{u*u'}_{C^{k,\alpha}_{\beta+\beta'}}
            \leq C\norm{u}_{C^{k,\alpha}_\beta}\norm{u'}_{C^{k,\alpha}_{\beta'}}.
        \end{equation*}
        \item\label{lem:basic-properties-compact-embedding}
        If $k+\alpha>k'+\alpha'$ and $\beta<\beta'$, where $k'\in\mathbb Z_{\ge0}$ and $\alpha'\in[0,1]$, then the inclusion $C^{k,\alpha}_\beta(M,E)\hookrightarrow C^{k',\alpha'}_{\beta'}(M,E)$ is compact.
    \end{enumerate}
\end{Lemma}
\begin{proof}
    \cref{lem:basic-properties-continuity-of-covariant-derivative} On each cylinder slab, the rescaled covariant derivative is a first-order operator whose principal part is the full cylindrical derivative and whose coefficients are uniformly bounded with all derivatives. The usual local norm comparison, followed by the supremum over the slabs and the compact-core estimate, gives the asserted equivalence.

    \noindent\cref{lem:basic-properties-isomorphism} The AC estimates for $f$ imply that $\boldsymbol{\rho}/r$ and its reciprocal have uniformly bounded cylindrical derivatives of every order. Moreover,
    \begin{equation*}
        \mathscr T_{\beta'}(\boldsymbol{\rho}^{\beta'-\beta}u)
        =(\boldsymbol{\rho}/r)^{\beta'-\beta}\mathscr T_\beta u.
    \end{equation*}
    Ordinary H\"older multiplication on the slabs and the compact core proves the bound. Interchanging $\beta$ and $\beta'$ gives its inverse.

    \noindent\cref{lem:basic-properties-multiplication} In the fixed radial identification, tensor products have constant algebraic coefficients. Contractions with $g$ or $g^{-1}$ have uniformly bounded cylindrical coefficients. The ordinary H\"older product estimate therefore applies on every slab with the same constant; the compact-core estimate completes the proof.

    \noindent\cref{lem:basic-properties-compact-embedding} On every bounded cylinder segment and on the compact core, the ordinary compact embedding applies. On the tail $t\geq S$, the change from weight $\beta$ to weight $\beta'$ contributes $r^{\beta-\beta'}$, whose cylindrical derivatives are bounded by $Ce^{(\beta-\beta')S}$. Thus the norm of the inclusion on the tail tends to zero as $S\to\infty$. A diagonal subsequence and this uniform tail estimate give convergence in $C^{k',\alpha'}_{\beta'}$.
\end{proof}

The positive weight $\boldsymbol\rho$ also gives a useful conjugation
formula. For $L=\Delta_f-\lambda$, $\gamma\in\mathbb R$, and
$u\in C^2_{\mathrm{loc}}(M,E)$, the product rule gives
\begin{equation}\label{eq:weight-conjugation}
    \boldsymbol\rho^{-\gamma}L(\boldsymbol\rho^\gamma u)
    =\Delta_{f-2\gamma\log\boldsymbol\rho}u
    -(\lambda+\gamma/2)u+\phi_\gamma u,
\end{equation}
where
\begin{equation*}
    \phi_\gamma
    =4\gamma(\gamma-2)\boldsymbol\rho^{-4}|\nabla f|^2
    +\gamma(n+2)\boldsymbol\rho^{-2}.
\end{equation*}
Indeed, $\Delta_f(\boldsymbol\rho^2)=2n+4-\boldsymbol\rho^2$, so
$\boldsymbol\rho^{-\gamma}\Delta_f\boldsymbol\rho^\gamma
=-\gamma/2+\phi_\gamma$.
The soliton equation, the curvature derivative bounds, and the weighted
product estimates give
\begin{equation}\label{eq:weight-conjugation-coefficients}
    \norm{2\gamma\nabla\log\boldsymbol{\rho}}_{C^{k+1,\alpha}_{-1}}
    +\norm{\phi_\gamma}_{C^{k+1,\alpha}_{-2}}
    \leq C(|\gamma|+\gamma^2),
\end{equation}
where $C$ is independent of $\gamma$.
Indeed, $|\nabla^j\log\boldsymbol{\rho}|\leq C_j\boldsymbol{\rho}^{-j}$ for $j\geq1$, while
$|\nabla^j(\boldsymbol{\rho}^{-4}|\nabla f|^2)|\leq C_j\boldsymbol{\rho}^{-2-j}$; one further derivative controls the H\"older norm on each rescaled annulus.
Thus conjugation shifts $\lambda$ to $\lambda+\gamma/2$ and introduces first-order and zeroth-order coefficients of weights $-1$ and $-2$, respectively. Multiplication by $\boldsymbol{\rho}^{-\gamma}$ changes the weight $\beta$ to $\beta-\gamma$, so $2\lambda+\beta$ is preserved.
The Gaussian Fredholm realization and its finite weighted estimates
are given in \cref{prop:Gaussian-Fredholm-realization,prop:weighted-Schauder-estimates-BCD}.

For a compatible pair of Banach spaces $(X,Y)$, $s>0$, and
$0<\vartheta<1$, recall the $K$-functional and real interpolation norm
\begin{equation*}
    K(s,u;X,Y):=\inf_{u=x+y}
       \bigl(\norm{x}_X+s\norm{y}_Y\bigr),
    \qquad
    \norm{u}_{(X,Y)_{\vartheta,\infty}}
       :=\sup_{s>0}s^{-\vartheta}K(s,u;X,Y).
\end{equation*}
The space $(X,Y)_{\vartheta,\infty}$ consists of the elements of $X+Y$
for which this norm is finite.
\Needspace{8\baselineskip}
\begin{Lemma}[Weighted real interpolation]\label{lem:weighted-real-interpolation}
    Assume the standard AC setting, let $E$ be a tensor bundle, and set
    $\beta_\vartheta=(1-\vartheta)\beta_0+\vartheta\beta_1$ for
    $\beta_0,\beta_1\in\mathbb R$ and $0<\vartheta<1$.
    The following identities hold with equivalent norms.
    \begin{enumerate}[label=\textnormal{(\roman{*})},ref=\roman{*}]
        \item\label{lem:weighted-real-interpolation-regularity}
        For $k\in\mathbb Z_{\geq0}$, $0<\alpha<1$, and $\beta\in\mathbb R$,
        \begin{equation}\label{eq:weighted-real-interpolation}
            (C^k_\beta(M,E),C^{k+1}_\beta(M,E))_{\alpha,\infty}
            =C^{k,\alpha}_\beta(M,E).
        \end{equation}
        \item\label{lem:weighted-real-interpolation-weights}
        For $k\in\mathbb Z_{\geq0}$ and $0\leq\alpha<1$,
        \begin{equation}\label{eq:weighted-weight-interpolation}
            (C^{k,\alpha}_{\beta_0}(M,E),
             C^{k,\alpha}_{\beta_1}(M,E))_{\vartheta,\infty}
            =C^{k,\alpha}_{\beta_\vartheta}(M,E).
        \end{equation}
        \item\label{lem:weighted-real-interpolation-simultaneous}
        Let $s_i=k_i+\alpha_i$, where $k_i\geq0$ are integers and
        $0\leq\alpha_i<1$, and suppose $s_0<s_1$. If
        \[
            (1-\vartheta)s_0+\vartheta s_1=k_\vartheta+\alpha_\vartheta,
            \qquad k_\vartheta\geq0,\quad 0<\alpha_\vartheta<1,
        \]
        with $k_\vartheta$ an integer, then
        \begin{equation}\label{eq:weighted-simultaneous-interpolation}
            (C^{k_0,\alpha_0}_{\beta_0}(M,E),
             C^{k_1,\alpha_1}_{\beta_1}(M,E))_{\vartheta,\infty}
            =C^{k_\vartheta,\alpha_\vartheta}_{\beta_\vartheta}(M,E).
        \end{equation}
    \end{enumerate}
    The corresponding unweighted regularity identities hold on the
    compact link $\Sigma$, for sections of every fixed smooth bundle
    over $\Sigma$.
\end{Lemma}
\begin{proof}
    Cover the compact core by finitely many charts and the cylindrical
    end by integer translates of finitely many fixed charts. Enlarge
    the core to absorb the initial collar. Choose a subordinate
    partition of unity $\{\varphi_j\}$ and cutoffs $\psi_j$ equal to
    one on $\operatorname{supp}\varphi_j$. The cover has uniformly
    bounded overlap, and all cutoffs, coordinate changes, and bundle
    trivializations have uniformly bounded derivatives.

    Let $L_j$ be the local coordinate representation using the
    unweighted radial identification $\boldsymbol\tau_r^{-1}$ on the
    end and an ordinary trivialization on the core. Let $R_j$ be the
    radius of the end chart and set $R_j=1$ on the core. Define
    \[
        Ju=(\varphi_jL_ju)_j,\qquad
        R(v_j)=\sum_jL_j^{-1}(\psi_jv_j),
    \]
    with the localized functions extended by zero. These maps do not
    depend on the weight or regularity. Put
    $A_i=C_b^{k_i,\alpha_i}(\mathbb R^n,\mathbb R^N)$, where $N$ is
    the bundle rank and $C_b^{k_i,0}=C_b^{k_i}$, and define
    \[
        Y_i=\{(v_j):\ \|(v_j)\|_{Y_i}
                   :=\sup_jR_j^{-\beta_i}\|v_j\|_{A_i}<\infty\}.
    \]
    The weighted norm on each chart is equivalent to
    $R_j^{-\beta_i}$ times its unweighted norm. Uniform localization
    therefore makes $J:C^{k_i,\alpha_i}_{\beta_i}\to Y_i$ and
    $R:Y_i\to C^{k_i,\alpha_i}_{\beta_i}$ bounded at both endpoints,
    with $RJ=\operatorname{Id}$. The same is true at every intermediate
    regularity and weight under consideration.

    For $v=(v_j)$, restriction of decompositions to individual
    components and independent choices of component decompositions give
    \[
    \begin{split}
        \mathcal K_t(v)&\leq K(t,v;Y_0,Y_1)\leq2\mathcal K_t(v),\\
        \mathcal K_t(v)&:=\sup_jR_j^{-\beta_0}
          K\bigl(tR_j^{\beta_0-\beta_1},v_j;A_0,A_1\bigr).
    \end{split}
    \]
    For the second inequality, choose each component decomposition
    within an arbitrary common error of its weighted infimum, take
    the two suprema, and let the error tend to zero. Interchanging
    the suprema over $t$ and $j$ and putting
    $s=tR_j^{\beta_0-\beta_1}$ yield
    \[
        \|v\|_{(Y_0,Y_1)_{\vartheta,\infty}}
        \asymp\sup_jR_j^{-\beta_\vartheta}
                 \|v_j\|_{(A_0,A_1)_{\vartheta,\infty}}.
    \]
    If $A_0=A_1=A$, then
    $K(s,a;A,A)=\min\{1,s\}\|a\|_A$, which proves part~\textup{(ii)},
    including $\alpha=0$, after applying $J$ and $R$.

    For parts~\textup{(i)} and~\textup{(iii)}, ordinary full H\"older
    interpolation gives
    \[
        (A_0,A_1)_{\vartheta,\infty}
            =C_b^{k_\vartheta,\alpha_\vartheta}(\mathbb R^n,\mathbb R^N)
    \]
    whenever the intermediate total order is noninteger; integer
    endpoint orders are allowed. Regularization with sufficiently many vanishing
    moments at scale $\epsilon$ gives an error bounded by
    $C\epsilon^{k_\vartheta+\alpha_\vartheta-s_0}\|v\|_{C^{k_\vartheta,\alpha_\vartheta}}$
    in $A_0$ and a regularized part bounded by
    $C\epsilon^{k_\vartheta+\alpha_\vartheta-s_1}\|v\|_{C^{k_\vartheta,\alpha_\vartheta}}$
    in $A_1$. This proves one $K$-functional bound. Conversely,
    differences of integer order larger than $s_1$, applied to a
    decomposition at parameter $|h|^{s_1-s_0}$, are bounded by
    $C|h|^{s_0}K(|h|^{s_1-s_0},v;A_0,A_1)$.
    The finite-difference characterization at a noninteger order gives
    the reverse bound; see also \cite[Chapter~1]{lunardi2018interpolation}.
    Applying the preceding sequence-space identity and the common retraction proves
    part~\textup{(iii)}. Taking equal weights, integer consecutive
    endpoint orders, and $\vartheta=\alpha$ gives part~\textup{(i)}.
    A finite chart cover gives the same proof on the compact link.
\end{proof}

\begin{Remark}\label{rmk:conic-Holder-interpolation}
    Deruelle's spaces $C^{k,\alpha}_{\mathrm{con}}$ in \cite[$\S 2.1$]{deruelle2016smoothing} measure the H\"older seminorms of the weighted derivatives on balls of a fixed radius. In \cite[Remark~2.1]{deruelle2016smoothing}, he distinguishes these spaces from the real interpolation spaces between consecutive integer conical spaces: the latter use balls whose radius is proportional to the distance from the compact core, together with the corresponding scaling of the H\"older quotient. Definition~\ref{def:weighted-Holder-spaces} uses this conic scale through ordinary norms on fixed cylinder slabs. Thus, up to equivalent norms, $C^{k,\alpha}_0$ has the interpolation convention in that remark. General weights follow by multiplication by $\boldsymbol{\rho}^{-\beta}$, as in \cref{lem:basic-properties}\cref{lem:basic-properties-isomorphism}.
\end{Remark}

The drift Laplacian combines a second-order elliptic term with the first-order term $-\nabla_{\nabla f}$. On the end, $\nabla f=\frac12r\partial_r$, so the drift has the same asymptotic weight as the section itself. This motivates a domain that controls $u$, $\nabla^2u$, and $\nabla_{\nabla f}u$ at a common target weight; \cref{cor:equivalent-critical-norm} below makes this precise.
Related anisotropic spaces were introduced by Chodosh--Schulze \cite[Definitions 3.1--3.2]{chodosh2021uniqueness}. Related graph and drift domains appear in \cite[$\S 2.1$]{deruelle2016smoothing} and \cite[Definition 5.3]{bamler2023degree}, with different H\"older conventions. Here we use conically scaled H\"older norms at every derivative order.
\begin{Definition}\label{def:cone-holder-space-general-component}
    Let $E$ be a tensor bundle over $M$.
    For $k\geq2$, $\alpha\in[0,1]$, and $\beta\in\mathbb R$, the space $C^{k,\alpha}_{\beta,\nabla f}(M,E)$ consists of locally $C^{k,\alpha}$ sections satisfying
    \begin{equation*}
        u,\nabla_{\nabla f}u\in C^{k-2,\alpha}_\beta,
        \qquad \nabla^{k-1}u,\nabla^ku\in C^{0,\alpha}_{\beta+2-k},
    \end{equation*}
    with the appropriate tensor bundles understood. Its norm is
    \begin{equation*}
        \norm{u}_{C^{k,\alpha}_{\beta,\nabla f}}
        :=\norm{u}_{C^{k-2,\alpha}_\beta}
        +\norm{\nabla_{\nabla f}u}_{C^{k-2,\alpha}_\beta}
        +\norm{\nabla^{k-1}u}_{C^{0,\alpha}_{\beta+2-k}}
        +\norm{\nabla^ku}_{C^{0,\alpha}_{\beta+2-k}}.
    \end{equation*}
    In particular, $u\in C^{k+2,\alpha}_{\beta,\nabla f}$ implies $u\in C^{k+2,\alpha}_{\beta+2}$, $\nabla u\in C^{k+1,\alpha}_{\beta+1}$, and $\nabla^2u\in C^{k,\alpha}_\beta$.
\end{Definition}

The following interpolation estimate at the critical weight gives an equivalent norm.
\begin{Lemma}[Critical-weight interpolation]\label{lem:critical-weight-interpolation}
    Let $\gamma\in\mathbb R$, $\alpha\in[0,1]$, and let $v$ be a locally $C^{2,\alpha}$ section of a tensor bundle $E$.
    If $v\in C^{0,\alpha}_\gamma(M,E)$ and $\nabla^2v\in C^{0,\alpha}_\gamma(M,(T^*M)^{\otimes2}\otimes E)$, then $\nabla v\in C^{0,\alpha}_\gamma(M,T^*M\otimes E)$ and
    \begin{equation}\label{ineq:critical-weight-interpolation}
        \norm{\nabla v}_{C^{0,\alpha}_\gamma}
        \leq C\bigl(\norm{v}_{C^{0,\alpha}_\gamma}+\norm{\nabla^2v}_{C^{0,\alpha}_\gamma}\bigr).
    \end{equation}
\end{Lemma}
\begin{proof}
    Write $E=T^{p,q}M$ and put $N:=\norm{v}_{C^{0,\alpha}_\gamma}+\norm{\nabla^2v}_{C^{0,\alpha}_\gamma}$.
    On a translated cylinder slab with $r\asymp R$, use the rescaled metric $R^{-2}g$ and the tensor normalization $R^{p-q-\gamma}$ from \cref{eq:weighted-Holder-rescaled-annulus-norm}. Denote the resulting section by $\widetilde v$ and the rescaled connection by $\widetilde\nabla$. On a slightly larger slab,
    \begin{equation*}
        \norm{\widetilde v}_{C^{0,\alpha}}\leq CN,
        \qquad \norm{\widetilde\nabla^2\widetilde v}_{C^{0,\alpha}}\leq CR^2N.
    \end{equation*}
    Ordinary local interpolation on these uniformly controlled slabs gives, on the inner slab,
    \begin{equation*}
        \norm{\widetilde\nabla\widetilde v}_{C^{0,\alpha}}
        \leq C\bigl(\epsilon^{-1}\norm{\widetilde v}_{C^{0,\alpha}}
        +\epsilon\norm{\widetilde\nabla^2\widetilde v}_{C^{0,\alpha}}\bigr)
        \qquad(0<\epsilon\leq\epsilon_0).
    \end{equation*}
    Choose $\epsilon=c/R$, with a fixed sufficiently small $c>0$. The resulting bound is $\norm{\widetilde\nabla\widetilde v}_{C^{0,\alpha}}\leq CRN$.
    Scaling back yields a uniform $C^{0,\alpha}_\gamma$ bound for $\nabla v$ on every end slab. Ordinary interpolation on the compact core and the initial collar completes the proof.
\end{proof}

\begin{Corollary}\label{cor:equivalent-critical-norm}
    A tensor $u\in C^{k,\alpha}_{\mathrm{loc}}(M,E)$, with $k\geq2$, $\alpha\in[0,1]$, and $\beta\in\mathbb R$, lies in $C^{k,\alpha}_{\beta,\nabla f}(M,E)$ if and only if
    \begin{equation*}
        u,\nabla^2u,\nabla_{\nabla f}u\in C^{k-2,\alpha}_\beta,
    \end{equation*}
    with the appropriate tensor bundles understood. Moreover,
    \begin{equation}\label{ineq:equivalent-critical-norm}
        \norm{u}_{C^{k,\alpha}_{\beta,\nabla f}}
        \asymp\norm{u}_{C^{k-2,\alpha}_\beta}
        +\norm{\nabla^2u}_{C^{k-2,\alpha}_\beta}
        +\norm{\nabla_{\nabla f}u}_{C^{k-2,\alpha}_\beta}.
    \end{equation}
\end{Corollary}
\begin{proof}
    The defining norm controls the right-hand side: the lower derivatives of $\nabla^2u$ are controlled by $u\in C^{k-2,\alpha}_\beta$, and the remaining derivatives by the two critical terms in \cref{def:cone-holder-space-general-component}. The derivative norm comparison in \cref{lem:basic-properties}\cref{lem:basic-properties-continuity-of-covariant-derivative} and the inclusions between weights give the required bound.

    Conversely, set $v:=\nabla^{k-2}u$ and $\gamma:=\beta+2-k$. Finiteness of the three terms on the right gives $v,\nabla^2v\in C^{0,\alpha}_\gamma$. By \cref{lem:critical-weight-interpolation}, $\nabla^{k-1}u=\nabla v\in C^{0,\alpha}_\gamma$. The condition $\nabla^ku\in C^{0,\alpha}_\gamma$ follows directly from $\nabla^2u\in C^{k-2,\alpha}_\beta$. These are the remaining defining conditions, with the asserted norm estimate.
\end{proof}

The map $\Delta_f-\lambda:C^{k+2,\alpha}_{\beta,\nabla f}(M,E)\to C^{k,\alpha}_\beta(M,E)$ is therefore well-defined and continuous for every $\lambda\in\mathbb R$.
Multiplication by $\boldsymbol{\rho}^{\beta'-\beta}$ also gives a bounded isomorphism $C^{k,\alpha}_{\beta,\nabla f}\to C^{k,\alpha}_{\beta',\nabla f}$. This follows from the Leibniz rule, \cref{lem:basic-properties}, and the bound $\nabla_{\nabla f}\log\boldsymbol{\rho}\in C^{k-2,\alpha}_0$.

\subsection{Cone H\"{o}lder spaces}\label{sec:cone-Holder-spaces}
We retain the standard AC setting of \cref{def:standard-AC-setting}.
Given a tensor $u$ on $(\underline r,\infty)\times\Sigma$, we view $\chi u$ as a tensor on all of $M$.
Recall \cref{def:degree} for the notion ``degree''.
\Needspace{4\baselineskip}
\begin{Definition}[Cone H\"{o}lder spaces]\label{def:cone-holder-space}
    Given integers $p,q\geq0$, $k\geq2$, $\alpha\in[0,1]$, $\beta\in\R$ and a finite set $\Lambda\subset(\beta,\infty)$, the space $C^{k,\alpha}_{\Lambda;\beta,\nabla f}(M,T^{p,q}M)$ consists of those $u\in C^{k,\alpha}_{\mathrm{loc}}(M,T^{p,q}M)$ such that
    \begin{equation*}
        u=\chi\sum_{\gamma\in\Lambda}
        r^{\gamma}\boldsymbol{\tau}(u_{\gamma})
        +u^\star
    \end{equation*}
    for some $u_{\gamma}\in C^{k,\alpha}(\Sigma,T^{p,q}\mathscr{C}|_\Sigma)$ and $u^\star\in C^{k,\alpha}_{\beta,\nabla f}(M,T^{p,q}M)$.
    Here $\boldsymbol{\tau}$ denotes the radial parallel transport \cref{eq:parallel-transport-along-radial-direction}.
    We call $\chi r^{\gamma}\boldsymbol{\tau}(u_{\gamma})$ the \emph{homogeneous components} of $u$ with \emph{degree} $\gamma$; the above $u^\star$ is referred to as the \emph{general component} of $u$.
    The norm of $u\in C^{k,\alpha}_{\Lambda;\beta,\nabla f}(M,T^{p,q}M)$ is defined to be
	\begin{equation*}
		\norm{u}_{C^{k,\alpha}_{\Lambda;\beta,\nabla f}(M,T^{p,q}M)}
		:=
        \sum_{\gamma\in\Lambda}\norm{u_{\gamma}}_{C^{k,\alpha}(\Sigma,T^{p,q}\mathscr{C}|_\Sigma)}
        +\norm{u^\star}_{C^{k,\alpha}_{\beta,\nabla f}(M,T^{p,q}M)}.
	\end{equation*}
    We denote the degree-$\gamma$ component by $[u]_{\gamma}$ and its
    general component by $u^\star$ throughout. The bracket on a homogeneous
    component will be dropped when there is no chance of confusion.

    Replacing the space $C^{k,\alpha}_{\beta,\nabla f}(M,T^{p,q}M)$ in the above definition by $C^{k,\alpha}_{\beta}(M,T^{p,q}M)$, one can similarly define, for every integer $k\geq0$, the space $C^{k,\alpha}_{\Lambda;\beta}(M,T^{p,q}M)$.
    For any neighborhood $U$ of infinity, define $C^{k,\alpha}_{\beta}(U,T^{p,q}M)$ and $C^{k,\alpha}_{\beta,\nabla f}(U,T^{p,q}M)$ by restriction.
    
    For convenience, we write $C^{k,\alpha}_{\Lambda;-\infty}(M,T^{p,q}M)$ for the subspace of $C^{k,\alpha}_{\Lambda;\beta}(M,T^{p,q}M)$ consisting of those elements with vanishing general component.
    The subspace norm is independent of $\beta$ and we denote it by $\norm{\cdot}_{C^{k,\alpha}_{\Lambda;-\infty}(M,T^{p,q}M)}$.

    We usually omit the brace when $\Lambda$ is given explicitly, e.g., $C^{k,\alpha}_{0;-2,\nabla f}(M):=C^{k,\alpha}_{\{0\};-2,\nabla f}(M)$.
    The domain or the tensor bundle will also be omitted when there is no chance of confusion.
\end{Definition}

The coefficients $u_{\gamma}$ are uniquely determined by $u$.
Moreover, the map
\begin{equation*}
    \prod_{\gamma\in\Lambda}
    C^{k,\alpha}(\Sigma,T^{p,q}\mathscr{C}|_\Sigma)
    \times
    C^{k,\alpha}_{\beta,\nabla f}(M,T^{p,q}M)
    \to
    C^{k,\alpha}_{\Lambda;\beta,\nabla f}(M,T^{p,q}M)
\end{equation*}
sending $((u_{\gamma})_{\gamma\in\Lambda},u^\star)$ to
\begin{equation*}
    \chi\sum_{\gamma\in\Lambda}
    r^{\gamma}\boldsymbol{\tau}(u_{\gamma})
    +u^\star
\end{equation*}
is an isomorphism of vector spaces and $C^{k,\alpha}_{\Lambda;\beta,\nabla f}(M,T^{p,q}M)$ is therefore a Banach space.

\begin{Example}
    Let $k\geq2$ and $\alpha\in[0,1]$.
    Let $C^{k,\alpha}_{0;-2}(M,S^2T^*M)$ be the subspace of $C^{k,\alpha}_{0;-2}(M,T^{0,2}M)$ consisting of symmetric $(0,2)$-tensors.
    Consider an element $h$ of the form $h=\chi r^2c(\theta)+h^\star$, where $c\in C^{k,\alpha}(\Sigma,S^2T^*\Sigma)$ and $h^\star\in C^{k,\alpha}_{-2}(M,S^2T^*M)$.
    When $\norm{h}_{C^{k,\alpha}_{0;-2}(M,S^2T^*M)}$ is sufficiently small, $g+h$ is a Riemannian metric on $M$ asymptotic to the $C^{k,\alpha}$ Riemannian cone with link $(\Sigma,g_{\Sigma}+c)$.
\end{Example}

On the end, the shrinker operators are controlled decaying perturbations
of the corresponding cone operators. The following comparison lets us
use the cone model to organize weighted estimates and asymptotic terms;
global solvability is handled by the Gaussian Fredholm theory.
Denote the potential function of the asymptotic cone by $f_0(r,\theta):=\frac{1}{4}r^2$.
\Needspace{6\baselineskip}
\begin{Lemma}\label{lem:difference}
    Let $k\in\mathbb Z_{\ge0}$, $\alpha\in[0,1]$, $\beta\in\R$ and $\rho\geq\underline{r}$.
    \begin{enumerate}[label=\textnormal{(\roman{*})},ref=\roman{*}]
        \item \label{lem:difference-cut-off} The map $C^{k,\alpha}_{\beta}(Q_{\underline{r}},T^{p,q}M)\to C^{k,\alpha}_{\beta}(M,T^{p,q}M),u\mapsto\chi u$ is well-defined and
        \begin{equation*}
            \norm{\chi u}_{C^{k,\alpha}_{\beta}(M,T^{p,q}M)}
            \leq C\norm{u}_{C^{k,\alpha}_{\beta}(Q_{\underline{r}},T^{p,q}M)}.
        \end{equation*}

        \item \label{lem:difference-radial-derivative} For any $u\in C^{1}_{\mathrm{loc}}(Q_{\rho},T^{p,q}M)$, we have
        $(\nabla_{\nabla f}-\nabla^C_{\mathbf{X}})u=\Ric(g)* u$.
        In particular,
        \begin{itemize}
            \item If $u\in C^{k,\alpha}_{\beta}(Q_\rho,T^{p,q}M)$, then $(\nabla_{\nabla f}-\nabla^C_{\mathbf{X}})u\in C^{k,\alpha}_{\beta-2}(Q_{\rho},T^{p,q}M)$ and
            \begin{equation*}
                \norm{(\nabla_{\nabla f}-\nabla^C_{\mathbf{X}})u}_{C^{k,\alpha}_{\beta-2}(Q_{\rho},T^{p,q}M)}
                \leq C\norm{u}_{C^{k,\alpha}_{\beta}(Q_\rho,T^{p,q}M)}.
            \end{equation*}

            \item If $u\in C^{k,\alpha}_{\beta;-\infty}(Q_\rho,T^{p,q}M)$, then $(\nabla_{\nabla f}-\nabla^C_{\mathbf{X}})u\in C^{k,\alpha}_{\beta-2;\beta-4}(Q_{\rho},T^{p,q}M)$ and
            \begin{equation*}
                \norm{(\nabla_{\nabla f}-\nabla^C_{\mathbf{X}})u}_{C^{k,\alpha}_{\beta-2;\beta-4}(Q_{\rho},T^{p,q}M)}
                \leq C\norm{u}_{C^{k,\alpha}_{\beta;-\infty}(Q_\rho,T^{p,q}M)}.
            \end{equation*}
        \end{itemize}
        Note when $p=q=0$, we have $\nabla_{\nabla f}u=\nabla^C_{\mathbf{X}}u$ over $Q_{\underline{r}}$.

        \item \label{lem:difference-Laplacian} If $u\in C^{k+2,\alpha}_{\beta}(Q_\rho,T^{p,q}M)$, then $(\Delta-\Delta^C)u\in C^{k,\alpha}_{\beta-4}(Q_\rho,T^{p,q}M)$ and
        \begin{equation*}
            \norm{(\Delta-\Delta^C)u}_{C^{k,\alpha}_{\beta-4}(Q_\rho,T^{p,q}M)}
            \leq C
            \norm{u}_{C^{k+2,\alpha}_{\beta}(Q_\rho,T^{p,q}M)}.
        \end{equation*}
        If $u\in C^{k+2,\alpha}_{\beta;-\infty}(Q_\rho,T^{p,q}M)$, then $(\Delta-\Delta^C)u\in C^{k,\alpha}_{\beta-4;\beta-6}(Q_\rho,T^{p,q}M)$ and
        \begin{equation*}
            \norm{(\Delta-\Delta^C)u}_{C^{k,\alpha}_{\beta-4;\beta-6}(Q_\rho,T^{p,q}M)}
            \leq C
            \norm{u}_{C^{k+2,\alpha}_{\beta;-\infty}(Q_\rho,T^{p,q}M)}.
        \end{equation*}
    \end{enumerate}
\end{Lemma}
\begin{proof}
    \cref{lem:difference-cut-off} The cutoff has uniformly bounded cylindrical derivatives and is supported away from the initial end slice. Multiplication by $\chi$ and extension by zero therefore give a bounded map, by the ordinary H\"older product estimate on the slabs and the compact transition region.

    \noindent \cref{lem:difference-radial-derivative} Recall $\nabla f=\mathbf{X}$ on $Q_{\underline{r}}$ (\cref{thm:canonical-asymptotic-chart}\cref{thm:canonical-asymptotic-chart-gradient-f}).
    By \cref{eq:difference-of-Lie-and-covariant-derivative} and the soliton equation, we have
    \begin{equation*}
        (\nabla_{\nabla f}-\nabla^C_{\mathbf{X}})u
        =(\nabla_{\nabla f}-\mathcal{L}_{\nabla f})u
        -(\nabla^C_{\mathbf{X}}-\mathcal{L}_{\mathbf{X}})u
        =\Ric(g)* u.
    \end{equation*}
    The first inequality is then a consequence of the quadratic curvature decay.
    For the homogeneous case, use $\Ric(g)=\Ric(g_C)+(\Ric(g)-\Ric(g_C))$ and \cref{lem:asymptotic-rate-of-connection}.

    \noindent \cref{lem:difference-Laplacian} Denote by $T$ the $(1,2)$-tensor defined by $T_{ij}^{k}:=\Gamma_{ij}^{k}(g)-\Gamma_{ij}^{k}(g_C)$.
    A direct computation gives
    \begin{align*}
        (\Delta-\Delta^C)u
        =(g^{-1}-g_C^{-1})*\nabla^{C,2}u
        +g^{-1}*(T*\nabla^Cu  +\nabla^CT*u  +T*T*u).
    \end{align*}
    The first estimate follows from \cref{ineq:asymptotic-rate} and \cref{lem:asymptotic-rate-of-connection}.
    To identify the leading term for homogeneous $u$, put $h_{-2}:=2\Ric(g_C)$ and write
    \begin{equation*}
        g=g_C+h_{-2}+O(r^{-4}),\qquad
        g^{-1}=g_C^{-1}+q_{-2}+O(r^{-4}),\qquad
        q_{-2}:=-g_C^{-1}h_{-2}g_C^{-1}.
    \end{equation*}
    All remainders here and below have the corresponding differentiated estimates. The connection-difference formula gives
    \begin{equation*}
        T=T_{-3}+O(r^{-5}),\qquad
        (T_{-3})_{ij}^{k}
        :=\frac12g_C^{k\ell}
        \bigl(\nabla_i^Ch_{-2,j\ell}+\nabla_j^Ch_{-2,i\ell}
        -\nabla_\ell^Ch_{-2,ij}\bigr).
    \end{equation*}
    The tensors $q_{-2}$ and $T_{-3}$ are homogeneous of degrees $-2$ and $-3$, respectively. Keeping the terms linear in these tensors in the displayed Laplacian difference gives
    \begin{equation*}
        q_{-2}*\nabla^{C,2}u
        +g_C^{-1}*\bigl(T_{-3}*\nabla^Cu+\nabla^CT_{-3}*u\bigr),
    \end{equation*}
    where the contractions are those of the exact difference formula and use $g_C$. If $u$ has degree $\beta$, this expression has degree $\beta-4$. The inverse-metric remainder times $\nabla^{C,2}u$, the connection remainder times $\nabla^Cu$, its derivative times $u$, and the quadratic terms in $T$ all have weight $\beta-6$. The weighted product estimates bound their $C^{k,\alpha}_{\beta-6}$ norms by the $C^{k+2,\alpha}$ norm of the link coefficient of $u$. This proves the homogeneous assertion.
\end{proof}
In particular, for $u\in C^{k+2,\alpha}_{\beta}(Q_\rho,T^{p,q}M)$, we have
\begin{equation*}
    (\Delta^g_f-\Delta^C_{f_0})u\in
    \begin{cases}
        C^{k,\alpha}_{\beta-4}(Q_{\rho},T^{p,q}M), & \text{if }p=q=0;  \\
        C^{k,\alpha}_{\beta-2}(Q_{\rho},T^{p,q}M), & \text{otherwise.}
    \end{cases}
\end{equation*}

In \cref{prop:characterization-of-cone-Holder-space} below, we characterize the critical cone H\"older spaces by their ordinary weighted norm and their characteristic radial derivative.
For a tensor $u$ on the end, write
\begin{equation*}
    u_r:=\boldsymbol\tau_r^{-1}(u|_{\Sigma_r}),\qquad
    u|_{\Sigma_r}=\boldsymbol\tau_r u_r,
\end{equation*}
so that $u_r$ is a section of the fixed link bundle $E_\Sigma$.
\begin{Lemma}\label{lem:characterization-of-classic-weighted-Holder-space}
    Let $k\geq0$ be an integer, $\alpha\in[0,1]$, and $\beta\in\mathbb R$.
    For any $u\in C^{k,\alpha}_{\beta}(Q_{\underline r},T^{p,q}M)$, we have
    \begin{equation}\label{ineq:norm-after-radial-transport}
        \sup_{r>\underline r}r^{-\beta}
        \norm{u_r}_{C^{k,\alpha}(\Sigma,E_\Sigma)}
        \leq C\norm{u}_{C^{k,\alpha}_{\beta}(Q_{\underline r},T^{p,q}M)}.
    \end{equation}
    If $\beta<0$, define
    \begin{equation*}
        v_r:=\int_r^\infty\frac{u_s}{s}\,\mathrm ds,
        \qquad v(r,\theta):=\boldsymbol\tau_r\bigl(v_r(\theta)\bigr).
    \end{equation*}
    Then $v\in C^{k,\alpha}_{\beta}(Q_{\underline r},T^{p,q}M)$ and
    \begin{equation}\label{ineq:integral-of-tensor}
        \norm{v}_{C^{k,\alpha}_{\beta}(Q_{\underline r},T^{p,q}M)}
        \leq C\norm{u}_{C^{k,\alpha}_{\beta}(Q_{\underline r},T^{p,q}M)}.
    \end{equation}
    In particular, suppose $v\in C^{k,\alpha}_{\mathrm{loc}}(Q_{\underline r},T^{p,q}M)$ has a classical radial derivative and
    \begin{equation*}
        \nabla^C_{\mathbf X}v\in C^{k,\alpha}_{\beta}(Q_{\underline r},T^{p,q}M),
        \qquad \beta<0.
    \end{equation*}
    Then there is a radially parallel trace $v_\infty\in C^{k,\alpha}(\Sigma,E_\Sigma)$ such that
    $v=\boldsymbol\tau(v_\infty)+v^\star$ on the end, with $v^\star\in C^{k,\alpha}_\beta$, and
    \begin{equation}\label{ineq:radial-trace-remainder}
        \norm{v^\star}_{C^{k,\alpha}_\beta(Q_{\underline r})}
        \leq C\norm{\nabla^C_{\mathbf X}v}_{C^{k,\alpha}_\beta(Q_{\underline r})}.
    \end{equation}
\end{Lemma}
\begin{proof}
    The first assertion follows by restricting the cylindrical representative in \cref{def:weighted-Holder-spaces} to each slice.

    For the second assertion, use $t=\log(r/\underline r)$ and write $U(t):=u_{\underline r e^t}$ and $V(t):=v_{\underline r e^t}$. The definition gives
    \begin{equation}\label{eq:log-cylinder-weighted-norm}
        \norm{u}_{C^{k,\alpha}_\beta(Q_{\underline r})}
        \asymp\sup_{T\geq0}e^{-\beta T}\norm{U}_{C^{k,\alpha}([T,T+1]\times\Sigma)}.
    \end{equation}
    The cylindrical norms use the fixed bundle $E_\Sigma$ and its connection.
    In these variables,
    \begin{equation*}
        V(t)=\int_t^\infty U(\tau)\,\mathrm d\tau.
    \end{equation*}
    Tangential derivatives commute with this integral, while every positive $t$-derivative of $V$ is a derivative of $-U$.
    Splitting the integral into unit intervals and using $\beta<0$ gives
    \begin{equation*}
        e^{-\beta T}\norm{V}_{C^{k,\alpha}([T,T+1]\times\Sigma)}
        \leq C\sum_{j=0}^\infty e^{\beta j}\norm{u}_{C^{k,\alpha}_\beta(Q_{\underline r})}
        \leq C(\beta)\norm{u}_{C^{k,\alpha}_\beta(Q_{\underline r})}.
    \end{equation*}
    The same estimate applies to differences of the top derivatives, including Lipschitz differences when $\alpha=1$; at $\alpha=0$ it is the ordinary $C^k$ estimate. Formula \cref{eq:log-cylinder-weighted-norm} proves \cref{ineq:integral-of-tensor}.

    Finally, $\mathbf X=\frac12r\partial_r$ and the definition of $\boldsymbol\tau_r$ give
    \begin{equation*}
        \partial_r v_r=\frac2r(\nabla^C_{\mathbf X}v)_r.
    \end{equation*}
    Consequently,
    \begin{equation*}
        v_\infty:=v_{r_0}+2\int_{r_0}^\infty
        \frac{(\nabla^C_{\mathbf X}v)_s}{s}\,\mathrm ds
    \end{equation*}
    is independent of $r_0>\underline r$, and
    \begin{equation*}
        v_r-v_\infty=-2\int_r^\infty
        \frac{(\nabla^C_{\mathbf X}v)_s}{s}\,\mathrm ds.
    \end{equation*}
    The integral estimate proves \cref{ineq:radial-trace-remainder}.
\end{proof}

We use the positive potential weight $\boldsymbol{\rho}$ from \cref{eq:positive-potential-weight} for radial conjugation.
\begin{Lemma}\label{lem:radial-conjugation}
    Let $k\geq2$, $\alpha\in[0,1]$, and $\lambda\in\mathbb R$.
    In addition to the ordinary weighted isomorphisms in
    \cref{lem:basic-properties}\cref{lem:basic-properties-isomorphism},
    multiplication by $\boldsymbol{\rho}^{2\lambda}$ is a bounded isomorphism
    \begin{equation}\label{eq:critical-radial-conjugation}
        C^{k,\alpha}_{-2\lambda;-2\lambda-2,\nabla f}(M,T^{p,q}M)
        \longrightarrow C^{k,\alpha}_{0;-2,\nabla f}(M,T^{p,q}M).
    \end{equation}
    The inverse map is multiplication by $\boldsymbol{\rho}^{-2\lambda}$, and the operator norms depend on the displayed parameters and the fixed standard AC data. The leading link coefficient in \cref{eq:critical-radial-conjugation} is unchanged.
    Moreover, the smooth function
    \begin{equation}\label{eq:compact-conjugation-coefficient}
        b_\lambda:=2\lambda\boldsymbol{\rho}^{-1}
        \nabla_{\nabla f}\boldsymbol{\rho}-\lambda
        =-4\lambda\frac{1+\mathrm R}{\boldsymbol{\rho}^2}
    \end{equation}
    belongs to $C^\infty_{-2}$, and
    \begin{equation}\label{eq:radial-conjugation-identity}
        \nabla_{\nabla f}(\boldsymbol{\rho}^{2\lambda}u)
        =\boldsymbol{\rho}^{2\lambda}(\nabla_{\nabla f}u+\lambda u)
        -4\lambda(1+\mathrm R)\boldsymbol{\rho}^{2\lambda-2}u.
    \end{equation}
\end{Lemma}
\begin{proof}
    Weight multiplication is an isomorphism on the drift spaces, as observed after \cref{cor:equivalent-critical-norm}; in particular, it maps $C^{k,\alpha}_{-2\lambda-2,\nabla f}$ to $C^{k,\alpha}_{-2,\nabla f}$.

    By \cref{lem:asymptotic-rate-of-connection}, $\boldsymbol{\rho}/r=1+O(r^{-2})$ with the corresponding bounds on every derivative. Thus
    \begin{equation*}
        a_\lambda:=\left(\frac{\boldsymbol{\rho}}r\right)^{2\lambda}-1
        \in C^\infty_{-2}(Q_{\underline r}),
        \qquad \nabla_{\nabla f}a_\lambda\in C^\infty_{-2}(Q_{\underline r}).
    \end{equation*}
    For a homogeneous term with link coefficient $a\in C^{k,\alpha}(\Sigma,E_\Sigma)$,
    \begin{equation*}
        \boldsymbol{\rho}^{2\lambda}\chi r^{-2\lambda}\boldsymbol\tau(a)
        =\chi\boldsymbol\tau(a)
        +\chi a_\lambda\boldsymbol\tau(a).
    \end{equation*}
    The second term belongs to $C^{k,\alpha}_{-2,\nabla f}$ by the weighted product estimates and \cref{lem:difference}\cref{lem:difference-radial-derivative}. Together with the bound on the general component, this proves boundedness in \cref{eq:critical-radial-conjugation}. The analogous expansion for $(\boldsymbol{\rho}/r)^{-2\lambda}$ proves that multiplication by $\boldsymbol{\rho}^{-2\lambda}$ is a bounded inverse.

    Since $\boldsymbol{\rho}^2=4(1+f-\boldsymbol\mu(g))$ and $|\nabla f|^2=\boldsymbol{\rho}^2/4-1-\mathrm R$, one has
    \begin{equation*}
        \nabla_{\nabla f}\boldsymbol{\rho}
        =\frac{2|\nabla f|^2}{\boldsymbol{\rho}}
        =\frac{\boldsymbol{\rho}}2-\frac{2(1+\mathrm R)}{\boldsymbol{\rho}}.
    \end{equation*}
    This gives \cref{eq:compact-conjugation-coefficient}, and the curvature derivative bounds imply $b_\lambda\in C^\infty_{-2}$. The Leibniz rule now proves \cref{eq:radial-conjugation-identity}.
\end{proof}

\begin{Proposition}\label{prop:characterization-of-cone-Holder-space}
    Assume the standard AC setting of \cref{def:standard-AC-setting}.
    Let $k\geq2$, $\alpha\in(0,1]$, $\lambda\in\mathbb R$, and $u\in C^{k,\alpha}_{\mathrm{loc}}(M,T^{p,q}M)$.
    Then
    \begin{equation*}
        u\in C^{k,\alpha}_{-2\lambda;-2\lambda-2,\nabla f}(M,T^{p,q}M)
    \end{equation*}
    if and only if
    \begin{equation*}
        u\in C^{k,\alpha}_{-2\lambda}(M,T^{p,q}M),\qquad
        \nabla_{\nabla f}u+\lambda u\in C^{k-2,\alpha}_{-2\lambda-2}(M,T^{p,q}M).
    \end{equation*}
    Moreover, there are constants $c,C>0$ such that
    \begin{equation}\label{ineq:critical-trace-norm-equivalence}
        c\norm{u}_{C^{k,\alpha}_{-2\lambda;-2\lambda-2,\nabla f}}
        \leq\norm{u}_{C^{k,\alpha}_{-2\lambda}}
        +\norm{\nabla_{\nabla f}u+\lambda u}_{C^{k-2,\alpha}_{-2\lambda-2}}
        \leq C\norm{u}_{C^{k,\alpha}_{-2\lambda;-2\lambda-2,\nabla f}}.
    \end{equation}
\end{Proposition}
\begin{proof}
    We first prove the result for $\lambda=0$. Put
    \begin{equation*}
        \mathsf Q(u):=\norm{u}_{C^{k,\alpha}_0}
        +\norm{\nabla_{\nabla f}u}_{C^{k-2,\alpha}_{-2}}.
    \end{equation*}
    Suppose $u=\chi\boldsymbol\tau(u_\infty)+u^\star$ belongs to $C^{k,\alpha}_{0;-2,\nabla f}$. The cone norm bounds $\norm{u}_{C^{k,\alpha}_0}$ and $\norm{\nabla_{\nabla f}u^\star}_{C^{k-2,\alpha}_{-2}}$.
    Since $\nabla^C_{\mathbf X}\boldsymbol\tau(u_\infty)=0$, \cref{lem:difference}\cref{lem:difference-radial-derivative} gives
    \begin{equation*}
        \nabla_{\nabla f}(\chi\boldsymbol\tau(u_\infty))
        =(\nabla_{\nabla f}\chi)\boldsymbol\tau(u_\infty)
        +\chi\Ric(g)*\boldsymbol\tau(u_\infty).
    \end{equation*}
    The cutoff term is compactly supported and the curvature term has quadratic decay. Hence $\mathsf Q(u)\leq C\norm{u}_{C^{k,\alpha}_{0;-2,\nabla f}}$.

    Conversely, suppose $u\in C^{k,\alpha}_0$ and $\nabla_{\nabla f}u\in C^{k-2,\alpha}_{-2}$. Set $z:=\nabla^C_{\mathbf X}u$. By \cref{lem:difference}\cref{lem:difference-radial-derivative}, $z=\nabla_{\nabla f}u-\Ric(g)*u$, so
    \begin{equation}\label{ineq:critical-radial-source}
        \norm{z}_{C^{k-2,\alpha}_{-2}}\leq C\mathsf Q(u).
    \end{equation}
    The identity $\partial_ru_r=2r^{-1}z_r$ and \cref{lem:characterization-of-classic-weighted-Holder-space} show that
    \begin{equation}\label{eq:defining-homogeneous-component}
        u_\infty:=u_{r_0}+2\int_{r_0}^\infty\frac{z_s}{s}\,\mathrm ds
    \end{equation}
    is independent of $r_0>\underline r$, and $u_r\to u_\infty$ in $C^{k-2,\alpha}(\Sigma,E_\Sigma)$.
    Estimate \cref{ineq:norm-after-radial-transport} also bounds $u_r$ uniformly in $C^{k,\alpha}$. Since $\alpha>0$, compact embedding on the link gives convergence in $C^{k,\alpha'}$ for every $0\leq\alpha'<\alpha$. Lower semicontinuity of the top H\"older or Lipschitz seminorm gives
    \begin{equation}\label{ineq:critical-trace-link-bound}
        \norm{u_\infty}_{C^{k,\alpha}(\Sigma,E_\Sigma)}
        \leq C\norm{u}_{C^{k,\alpha}_0}.
    \end{equation}

    Define $u^\star:=u-\chi\boldsymbol\tau(u_\infty)$. Formula \cref{eq:defining-homogeneous-component} implies
    \begin{equation}\label{eq:defining-general-component}
        u^\star_r=(1-\chi(r,\cdot))u_r
        -2\chi(r,\cdot)\int_r^\infty\frac{z_s}{s}\,\mathrm ds.
    \end{equation}
    The radial integral estimate and the compact transition terms give $\norm{u^\star}_{C^{k-2,\alpha}_{-2}}\leq C\mathsf Q(u)$.
    The assumption $u\in C^{k,\alpha}_0$ and \cref{ineq:critical-trace-link-bound} give $\norm{\nabla^ku^\star}_{C^{0,\alpha}_{-k}}\leq C\mathsf Q(u)$.
    Apply \cref{lem:critical-weight-interpolation} to $\nabla^{k-2}u^\star$ to obtain $\norm{\nabla^{k-1}u^\star}_{C^{0,\alpha}_{-k}}\leq C\mathsf Q(u)$.
    Finally, the preceding estimate for $\nabla_{\nabla f}(\chi\boldsymbol\tau(u_\infty))$ gives $\norm{\nabla_{\nabla f}u^\star}_{C^{k-2,\alpha}_{-2}}\leq C\mathsf Q(u)$. These are the defining bounds for $u^\star\in C^{k,\alpha}_{-2,\nabla f}$, and they prove the converse norm estimate.

    For general $\lambda$, put $\widetilde u:=\boldsymbol{\rho}^{2\lambda}u$.
    By \cref{lem:radial-conjugation}, its $C^{k,\alpha}_{0;-2,\nabla f}$ norm is equivalent to the cone norm of $u$, and its $C^{k,\alpha}_0$ norm is equivalent to the ordinary weighted norm of $u$.
    Since $b_\lambda\in C^\infty_{-2}$, the weighted product estimate gives
    \begin{equation*}
        \norm{b_\lambda\widetilde u}_{C^{k-2,\alpha}_{-2}}
        \leq C\norm{\widetilde u}_{C^{k,\alpha}_0}.
    \end{equation*}
    Thus \cref{eq:radial-conjugation-identity} and the weight isomorphism yield
    \begin{equation*}
        \norm{\widetilde u}_{C^{k,\alpha}_0}
        +\norm{\nabla_{\nabla f}\widetilde u}_{C^{k-2,\alpha}_{-2}}
        \asymp\norm{u}_{C^{k,\alpha}_{-2\lambda}}
        +\norm{\nabla_{\nabla f}u+\lambda u}_{C^{k-2,\alpha}_{-2\lambda-2}}.
    \end{equation*}
    Applying the $\lambda=0$ result to $\widetilde u$ proves the assertion.
\end{proof}

\section{Schauder estimates for the weighted Laplacian}\label{sec:Schauder-estimates}
For a fixed tensor bundle, write $L_\lambda=\Delta_f-\lambda$.
\label{ic:Llambda}
Its characteristic homogeneous degree on the cone is $-2\lambda$.

\subsection{Maximum principles for the weighted Laplacian}
In this subsection, we study the asymptotic behavior of subsolutions to
weighted Laplacian equations on complete Ricci shrinkers.
Let $p_0$ be a minimum point of the potential and set
$\boldsymbol{\rho}=2(1+f-\boldsymbol{\mu}(g))^{1/2}$.
This positive weight is defined on every complete Ricci shrinker and satisfies
\[
 \Delta_f\boldsymbol{\rho}^2=2n+4-\boldsymbol{\rho}^2,
 \qquad |\nabla f|^2\leq\frac{\boldsymbol{\rho}^2}4-1,
 \qquad |\nabla\boldsymbol{\rho}|\leq1.
\]
By \cref{ineq:quadratic-growth-of-potential},
$\boldsymbol{\rho}$ is comparable to $1+d_g(p_0,\cdot)$, with constants
depending only on $n$. For a function $u$, write
$u^+=\max\{u,0\}$ and $u^-=\max\{-u,0\}$.
Similar maximum principles for self-shrinkers in mean curvature flow
appear in \cite[Lemma~3.15]{chodosh2021uniqueness}.

\begin{Lemma}[$L^2_f$ maximum principle]\label{lem:L2-maximum-principle}
    For any $n\geq3$ and $\lambda\in\R$, there is a constant
    $\underline r(n,\lambda)>0$ with the following property.
    Let $(M^n,g,f)$ be a complete Ricci shrinker and let $K$ be a compact
    subset containing $B_r(p_0)$ for some $r\geq\underline r$.
    Consider $L=\Delta_f-\lambda+\nabla_X+\phi$ on functions, where
    $X\in C^0_{\mathrm{loc}}(M,TM)$ and $\phi\in C^0_{\mathrm{loc}}(M)$ satisfy
    \begin{equation}\label{ineq:coefficients-of-L2-maximum-principle}
        |X|^2+\phi^+\leq\frac1{64}\boldsymbol{\rho}^2
        \qquad\text{on }M\setminus K.
    \end{equation}
    If $w\in H^1_f(M)$ satisfies $Lw\geq0$ on $M\setminus K$
    in the sense of distributions and $w\leq0$ almost everywhere in a
    neighborhood of $\partial K$, then $w\leq0$ almost everywhere on
    $M\setminus K$.
\end{Lemma}
\begin{proof}
    For every nonnegative $\varphi\in C_c^\infty(M\setminus K)$,
    \[
        \int_{M\setminus K}
        \bigl(-\langle\nabla w,\nabla\varphi\rangle-\lambda w\varphi
        +(\nabla_Xw)\varphi+\phi w\varphi\bigr)\,\mathrm dV_f\geq0.
    \]
    The extension by zero of $w^+|_{M\setminus K}$ belongs to $H^1_f(M)$.
    By \cref{lem:Poincare-inequality},
    $\boldsymbol{\rho}w^+\in L^2_f(M\setminus K)$.
    Test with compactly supported cutoff approximations to $w^+$,
    discard the nonpositive term $-\phi^-(w^+)^2$, and let the cutoffs
    tend to one. The cutoff errors vanish by the preceding integrability
    and $w^+\in H^1_f$. Thus
    \[
        \int_{M\setminus K}|\nabla w^+|^2\,\mathrm dV_f
        \leq\int_{M\setminus K}
        \bigl(-\lambda(w^+)^2+(\nabla_Xw^+)w^++\phi^+(w^+)^2\bigr)\,\mathrm dV_f.
    \]
    Young's inequality and \cref{ineq:coefficients-of-L2-maximum-principle} give
    \[
        \int_{M\setminus K}|\nabla w^+|^2\,\mathrm dV_f
        \leq\int_{M\setminus K}
        \left(\frac1{32}\boldsymbol{\rho}^2-2\lambda\right)(w^+)^2\,\mathrm dV_f.
    \]
    Applying \cref{lem:Poincare-inequality} once more, with
    $f-\boldsymbol\mu(g)=\boldsymbol{\rho}^2/4-1$, yields
    \begin{equation}\label{ineq:integral-estimate-on-the-positive-part}
        \int_{M\setminus K}
        \left(\frac18\boldsymbol{\rho}^2+8\lambda-n-1\right)(w^+)^2
        \,\mathrm dV_f\leq0.
    \end{equation}
    The potential growth estimate allows us to choose $\underline r(n,\lambda)$
    so that the coefficient in parentheses is at least one outside
    $B_{\underline r}(p_0)$. Hence $w^+=0$ on $M\setminus K$.
\end{proof}

The barrier argument below uses powers of the potential and lower-order
corrections, as in the proof of \cite[Theorem~1]{chow2011lower}.

\begin{Theorem}\label{thm:C0-estimate}
    For any $n\geq3$, $\lambda,\beta\in\R$, and $\delta,\Lambda>0$, there are
    constants $C>0$ and $\underline r>1$, depending only on these parameters,
    such that the following holds for every $r\geq\underline r$.
    Let $(M^n,g,f)$ be a complete Ricci shrinker and let
    $X\in C^0_{\mathrm{loc}}(M,TM)$ and $\phi\in C^0_{\mathrm{loc}}(M)$ satisfy
    \[
        |X|\leq\Lambda\boldsymbol{\rho}^{1-\delta},
        \qquad |\phi|\leq\Lambda\boldsymbol{\rho}^{-\delta}.
    \]
    Set $L=\Delta_f-\lambda+\nabla_X+\phi$.
    Suppose $u\in C^0_{\mathrm{loc}}(M)\cap H^1_f(M)$ satisfies $Lu\geq v$ in the sense
    of distributions, where $v$ is a locally integrable function with a fixed representative and
    \[
        V:=\sup_M\boldsymbol{\rho}^{-\beta}v^-<\infty.
    \]
    Write $\Omega_r=M\setminus B_r(p_0)$ and
    $U_r=\sup_{\partial B_r(p_0)}u^+$, with suprema of nonnegative functions
    over empty sets taken to be zero. Then
    \[
        \begin{dcases}
            \sup_{\Omega_r}\boldsymbol{\rho}^{2\lambda}u^+
            \leq C\bigl(V+r^{2\lambda}U_r\bigr),
            &\text{if }2\lambda+\beta<0;\\
            \sup_{\Omega_r}\frac{\boldsymbol{\rho}^{2\lambda}u^+}{1+\log\boldsymbol{\rho}}
            \leq C\left(V+\frac{r^{2\lambda}}{1+\log r}U_r\right),
            &\text{if }2\lambda+\beta=0;\\
            \sup_{\Omega_r}\boldsymbol{\rho}^{-\beta}u^+
            \leq C\bigl(V+r^{-\beta}U_r\bigr),
            &\text{if }2\lambda+\beta>0.
        \end{dcases}
    \]
\end{Theorem}
\begin{proof}
    The identities above give, for every $\gamma\in\R$,
    \begin{equation}\label{eq:L-acting-on-power-of-potential}
        L\boldsymbol{\rho}^{-2\gamma}
        =(\gamma-\lambda)\boldsymbol{\rho}^{-2\gamma}
        +\left(16\gamma(\gamma+1)\frac{|\nabla f|^2}{\boldsymbol{\rho}^2}
        -2\gamma(n+2)-4\gamma X(f)
        +\boldsymbol{\rho}^2\phi\right)\boldsymbol{\rho}^{-2\gamma-2}.
    \end{equation}
    Put $\delta_0=\min\{\delta,2\}$. Since
    $|X(f)|\leq\Lambda\boldsymbol{\rho}^{2-\delta}$, it follows that
    \begin{equation}\label{ineq:crucial-cancellation}
        \left|L\boldsymbol{\rho}^a
        +\left(\frac a2+\lambda\right)\boldsymbol{\rho}^a\right|
        \leq C(n,a,\delta,\Lambda)\boldsymbol{\rho}^{a-\delta_0}
        \qquad(a\in\R).
    \end{equation}
    Increase $\underline r$ so that
    \cref{ineq:coefficients-of-L2-maximum-principle} holds outside
    $B_{\underline r}(p_0)$. All powers of $\boldsymbol{\rho}$, as well as
    their products with $\log\boldsymbol{\rho}$, belong to $H^1_f(M)$.
    Indeed, normalization gives finite weighted volume, and successive
    applications of \cref{lem:Poincare-inequality} to cutoff powers of
    $\boldsymbol{\rho}$ give all polynomial moments; use also
    $|\nabla\boldsymbol{\rho}|\leq1$.

    First suppose $\theta:=2\lambda+\beta<0$. Choose
    $0<\epsilon<\min\{\delta_0,-\theta\}$ and set
    $h=\boldsymbol{\rho}^{-2\lambda}-\boldsymbol{\rho}^{-2\lambda-\epsilon}$.
    By \cref{ineq:crucial-cancellation}, after increasing $\underline r$,
    \[
        \frac12\boldsymbol{\rho}^{-2\lambda}\leq h\leq\boldsymbol{\rho}^{-2\lambda},
        \qquad Lh\leq-\frac\epsilon4\boldsymbol{\rho}^{-2\lambda-\epsilon}
        \quad\text{on }\Omega_{\underline r}.
    \]
    Since $\beta<-2\lambda-\epsilon$ and $\boldsymbol{\rho}\geq1$,
    the barrier $H=Ah$ satisfies $LH\leq-v^-$ whenever $A\geq4V/\epsilon$.
    For $\eta>0$, take $A=C(V+r^{2\lambda}U_r+\eta)$, with $C$ sufficiently
    large. The comparison $\boldsymbol{\rho}\asymp r$ on
    $\partial B_r(p_0)$ ensures $H>u^+$ there, hence in a neighborhood of
    the boundary. Extend $H$ smoothly through the compact core.
    Then $u-H\in H^1_f(M)$ and $L(u-H)\geq0$ on the exterior.
    Applying \cref{lem:L2-maximum-principle} gives $u\leq H$.
    Letting $\eta\downarrow0$ proves the first estimate.

    If $\theta=0$, differentiation of the power formula gives
    \[
        L(\boldsymbol{\rho}^{-2\lambda}\log\boldsymbol{\rho})
        =-\frac12\boldsymbol{\rho}^{-2\lambda}
        +O\bigl(\boldsymbol{\rho}^{-2\lambda-\delta_0}
        (1+\log\boldsymbol{\rho})\bigr)
        \leq-\frac14\boldsymbol{\rho}^{-2\lambda}
    \]
    outside a sufficiently large ball. Take
    $H=A\boldsymbol{\rho}^{-2\lambda}\log\boldsymbol{\rho}$, with
    $A=C(V+r^{2\lambda}(1+\log r)^{-1}U_r+\eta)$.
    The same comparison and the limit $\eta\downarrow0$ prove the critical estimate.

    Finally, if $\theta>0$, \cref{ineq:crucial-cancellation} gives
    $L\boldsymbol{\rho}^{\beta}\leq-\theta\boldsymbol{\rho}^{\beta}/4$
    sufficiently far out. Use $H=A\boldsymbol{\rho}^{\beta}$ with
    $A=C(V+r^{-\beta}U_r+\eta)$ and argue as above.
\end{proof}

\begin{Remark}\label{rmk:exterior-C0-estimate}
    The proof of \cref{thm:C0-estimate} uses the differential inequality
    and coefficient bounds only outside the comparison ball.
    Thus it also applies when these hypotheses hold on $M\setminus K_0$
    for a fixed compact set $K_0$: take $r\geq\underline r$ with
    $K_0\subset B_r(p_0)$, and replace $V$ by
    $\sup_{M\setminus K_0}\boldsymbol{\rho}^{-\beta}v^-$.
    The assumption $u\in H^1_f(M)$ and the boundary term $U_r$ are retained.
    The exterior barriers are extended smoothly through the compact core
    before applying \cref{lem:L2-maximum-principle}.
\end{Remark}

\begin{Proposition}\label{prop:C0-estimate}\label{cor:C0-estimate}
    For any $n\geq3$, $\lambda,\beta\in\R$, and $\delta,\Lambda>0$, there are
    constants $C,\underline r>0$, depending only on these parameters,
    with the following property. Let $(M^n,g,f)$ be a complete Ricci
    shrinker and let $E$ be a tensor bundle with its induced metric and connection. Write $\langle\cdot,\cdot\rangle_E$ for the pointwise inner product on $E$.
    Consider $L=\Delta_f-\lambda+\nabla_X+\phi$, where
    $X\in C^0_{\mathrm{loc}}(M,TM)$ and
    $\phi\in C^0_{\mathrm{loc}}(M,\End E)$ satisfy
    \[
        |X|\leq\Lambda\boldsymbol{\rho}^{1-\delta},
        \qquad |\phi|\leq\Lambda\boldsymbol{\rho}^{-\delta}.
    \]
    Suppose $u\in H^2_{\mathrm{loc}}(M,E)\cap H^1_f(M,E)$ and
    $v\in C^0_{\mathrm{loc}}(M,E)$ satisfy $\langle Lu,u\rangle_E\geq\langle v,u\rangle_E$
    almost everywhere and $V:=\sup_M\boldsymbol{\rho}^{-\beta}|v|<\infty$.
    Suprema involving a Sobolev section are understood as essential suprema;
    they agree with ordinary suprema for continuous sections. Then
    \[
        C\left(V+\sup_{B_{\underline r}(p_0)}|u|\right)
        \geq\begin{dcases}
            \sup_M\boldsymbol{\rho}^{2\lambda}|u|,&\text{if }2\lambda+\beta<0;\\
            \sup_M\frac{\boldsymbol{\rho}^{2\lambda}|u|}{1+\log\boldsymbol{\rho}},
            &\text{if }2\lambda+\beta=0;\\
            \sup_M\boldsymbol{\rho}^{-\beta}|u|,&\text{if }2\lambda+\beta>0.
        \end{dcases}
    \]
\end{Proposition}
\begin{proof}
    On functions, set $L'=\Delta_f-\lambda+\nabla_X+|\phi|$.
    For $\epsilon>0$, let $s_\epsilon=(\epsilon^2+|u|^2)^{1/2}$.
    The weak chain rule and Kato's inequality give, distributionally,
    \[
        L's_\epsilon
        \geq\frac{\langle Lu,u\rangle_E}{s_\epsilon}
        -\frac{\lambda\epsilon^2}{s_\epsilon}
        \geq-|v|-|\lambda|\epsilon.
    \]
    This follows first for smooth local approximations to $u$ and then
    by passage to the limit in $H^2_{\mathrm{loc}}$; the gradient-square
    terms are locally integrable and nonnegative.
    Letting $\epsilon\downarrow0$ gives $L'|u|\geq-|v|$ in distributions,
    with $|u|\in H^1_f$.
    The local boundedness estimate for weak subsolutions
    \cite[Theorem~8.17]{gilbarg2001elliptic} shows that $|u|$ is locally
    essentially bounded.

    Use the barriers in the proof of \cref{thm:C0-estimate}, choosing
    their coefficients to dominate $|u|$ almost everywhere on a fixed
    compact collar of the comparison sphere. After enlarging
    $\underline r$, this collar lies in $B_{\underline r}(p_0)$.
    The Sobolev maximum principle \cref{lem:L2-maximum-principle}
    applies directly to $|u|-H$ and gives the exterior estimate almost
    everywhere. The potential growth bound on the compact core completes
    the asserted global estimates.
\end{proof}

Related bounded-solution estimates for $\Delta_f-\lambda$, with a positive
zeroth-order gap, appear in \cite[Proposition~5.18]{bamler2023degree};
see \cite[Proposition~2.2 and Theorem~1]{lunardi1998schauder} and
\cite[Theorem~2.2]{deruelle2016smoothing} for antecedents.
On an AC shrinker, conjugation by $\boldsymbol{\rho}^{\beta}$ shifts the
spectral parameter to $\lambda+\beta/2$ and adds decaying lower-order terms,
as in \cref{eq:weight-conjugation}. Thus the range $2\lambda+\beta>0$ corresponds to
a positive zeroth-order gap after conjugation.

\begin{Corollary}\label{cor:AC-C0-estimate}
    In the standard AC setting of \cref{def:standard-AC-setting}, let
    $\lambda,\beta\in\R$ and $\delta,\Lambda>0$. Suppose
    $X\in C^0_{1-\delta}(M,TM)$ and
    $\phi\in C^0_{-\delta}(M,\End E)$ have norms at most $\Lambda$,
    and set $L=\Delta_f-\lambda+\nabla_X+\phi$.
    If $u\in H^2_{\mathrm{loc}}(M,E)\cap H^1_f(M,E)$ and
    $v\in C^0_\beta(M,E)$ satisfy $\langle Lu,u\rangle_E\geq\langle v,u\rangle_E$ almost everywhere, then
    \[
        C\left(\norm{v}_{C^0_\beta}+\norm{u}_{C^0(B_{\underline r}(p_0))}\right)
        \geq\begin{dcases}
            \norm{u}_{C^0_{-2\lambda}},&\text{if }2\lambda+\beta<0;\\
            \sup_M\frac{\boldsymbol{\rho}^{2\lambda}|u|}{1+\log\boldsymbol{\rho}},
            &\text{if }2\lambda+\beta=0;\\
            \norm{u}_{C^0_\beta},&\text{if }2\lambda+\beta>0.
        \end{dcases}
    \]
    Here $C,\underline r>0$ depend only on the standard AC data and the displayed parameters.
    For Sobolev sections, the zeroth-order norms and suprema in these
    estimates are essential bounds, as in \cref{prop:C0-estimate}.
    For continuous sections they are the usual weighted $C^0$ norms.

    For scalar subsolutions $u\in C^0_{\mathrm{loc}}(M)\cap H^1_f(M)$ with $Lu\geq v$
    distributionally, $v\in L^1_{\mathrm{loc}}(M)$, and
    $v^-\in C^0_\beta(M)$, the exterior estimates of \cref{thm:C0-estimate}
    hold with $V=\norm{v^-}_{C^0_\beta}$ and
    $U_r=\norm{u^+}_{C^0(\partial B_r(p_0))}$.
    In the first and third cases, the left-hand sides may be replaced by
    $\norm{u^+}_{C^0_{-2\lambda}(\Omega_r)}$ and
    $\norm{u^+}_{C^0_\beta(\Omega_r)}$, respectively.
\end{Corollary}
\begin{proof}
    The $C^0_\alpha$ norm is equivalent to
    $\sup_M\boldsymbol{\rho}^{-\alpha}|\cdot|$ by
    \cref{def:weighted-Holder-spaces}.
    Apply \cref{thm:C0-estimate,prop:C0-estimate} with the corresponding
    pointwise coefficient bounds.
\end{proof}

\begin{Remark}\label{rmk:Gaussian-integrability}
    On the Gaussian shrinker, radial solutions of $\Delta_fu=0$ on $\{r>1\}$ are
    \[
        u(r)=c_1+c_2\int_1^r s^{1-n}e^{s^2/4}\,\mathrm ds.
    \]
    When $c_2\neq0$, one has $u(r)\sim2c_2r^{-n}e^{r^2/4}$, so
    $u\notin L^2_f$ near infinity. This shows why the Gaussian
    integrability assumption is essential in the estimates above.

    The critical logarithmic factor is also necessary. A radial solution
    of $\Delta_fu=1$ on the Gaussian end is obtained by taking
    \[
        u'(r)=-r^{1-n}e^{r^2/4}\int_r^\infty s^{n-1}e^{-s^2/4}\,\mathrm ds
        =-\frac2r+O(r^{-3}).
    \]
    Thus $u(r)=-2\log r+O(1)$ and $u\in H^1_f$ near infinity.
    This gives the critical exterior behavior for $\lambda=\beta=0$.
\end{Remark}

The following direct consequence is related to the polynomial growth
estimates in \cite[Theorem~4.1]{colding2021singularities}.
\begin{Corollary}\label{cor:polynomial-growth}
    For every $n\geq3$ and $\lambda\in\R$, there exist $C,R>0$, depending only
    on $n,\lambda$, such that the following holds on any complete Ricci
    shrinker $(M^n,g,f)$.
    If $u\in H^2_{\mathrm{loc}}(M,E)\cap H^1_f(M,E)$ satisfies
    $\langle\Delta_fu,u\rangle\geq\lambda|u|^2$ almost everywhere, then
    \[
        \sup_M\boldsymbol{\rho}^{2\lambda}|u|
        \leq C\sup_{B_R(p_0)}|u|.
    \]
    Suprema of Sobolev sections are essential suprema, as in
    \cref{prop:C0-estimate}.
    In particular, if $\Delta_fu=-\nu u$ with $\nu\geq0$, then
    $|u(x)|\leq C\boldsymbol{\rho}(x)^{2\nu}\sup_{B_R(p_0)}|u|$.
\end{Corollary}
\begin{proof}
    Apply \cref{prop:C0-estimate} with $X=\phi=v=0$,
    $\beta=-2\lambda-1$, and, for example, $\delta=\Lambda=1$.
\end{proof}

\subsection{Schauder estimates I}\label{sec:parabolic-estimates-for-weighted-Laplacian}
Throughout this subsection, we consider AC Ricci shrinkers in the standard AC setting of \cref{def:standard-AC-setting}. Constants may depend on the fixed AC data and the displayed parameters and coefficient bounds.
\begin{Lemma}\label{lem:commuting-weighted-Laplacian-and-connection}
    Let $E$ be a tensor bundle with its induced connection, and write
    $\mathcal R^E(X,Y)=[\nabla_X,\nabla_Y]-\nabla_{[X,Y]}$ for its curvature.
    In a local orthonormal frame, every smooth section $u$ satisfies
    \begin{equation}\label{eq:exact-weighted-gradient-commutator}
        \bigl((\Delta_f-\frac12)\nabla u-\nabla\Delta_fu\bigr)_i
        =2\mathcal R^E_{ai}\nabla_a u,
    \end{equation}
    where repeated indices are summed and each Laplacian acts on the
    appropriate induced bundle. In particular, if
    $u\in C^{k,\alpha}_{\beta}(M,T^{p,q}M)$, $k\geq1$, and $\alpha\in[0,1]$, then
    \begin{equation}\label{ineq:commuting-weighted-Laplacian-and-connection}
        \norm{(\Delta_f-\frac12)\nabla u-\nabla\Delta_fu}_{C^{k-1,\alpha}_{\beta-3}(M,T^{p,q+1}M)}
        \leq C\norm{u}_{C^{k,\alpha}_{\beta}(M,T^{p,q}M)}.
    \end{equation}
\end{Lemma}
\begin{proof}
    The contracted second Bianchi identity and the soliton equation give
    the weighted divergence identity for curvature. Passing to the induced
    tensor representation yields
    \[
        \nabla^a\mathcal R^E_{ai}=(\nabla^a f)\mathcal R^E_{ai}.
    \]
    At the center of a normal orthonormal frame, ordinary commutation gives
    \[
        (\Delta\nabla u-\nabla\Delta u)_i
        =\Ric_i{}^a\nabla_a u+2\mathcal R^E_{ai}\nabla_a u
        +(\nabla^a\mathcal R^E_{ai})u,
    \]
    while
    \[
        (\nabla_{\nabla f}\nabla u-\nabla\nabla_{\nabla f}u)_i
        =(\nabla^a f)\mathcal R^E_{ai}u
        -(\nabla_i\nabla^a f)\nabla_a u.
    \]
    Subtract these identities and use $\Ric+\nabla^2f=g/2$.
    The zeroth-order curvature terms cancel, proving
    \cref{eq:exact-weighted-gradient-commutator}.
    The estimate follows from quadratic curvature decay and
    \cref{lem:basic-properties}\cref{lem:basic-properties-multiplication}.
\end{proof}

Next we establish the Ricci flow analogue of the weighted Schauder estimates in mean curvature flow \cite[Proposition 3.5]{chodosh2021uniqueness} with some slight generalizations (we deal with general tensors, higher order estimates and perturbed operators).
Their key observation is that one may interpret $\nabla_{\nabla f}$ as a time derivative when we compose solutions to $\Delta_fu=v$ with some one-parameter family of diffeomorphisms.
This transforms the elliptic equation (with unbounded coefficients) into a parabolic one, which admits parabolic Schauder estimates under appropriate conditions (making the coefficients all bounded).
We use Brandt's interior parabolic Schauder estimate \cite{brandt1969interior}, which requires only spatial H\"older bounds uniform in time. Chodosh--Schulze restate this estimate in \cite[Theorem~3.6]{chodosh2021uniqueness} and apply it in the shrinker setting. We record the form needed below.
\begin{Lemma}\label{lem:spatial-Holder-parabolic-estimate}
    Let $k\in\mathbb Z_{\ge0}$ and $\alpha\in(0,1)$. Let $\Omega_0\Subset\Omega_1$ be relatively compact coordinate domains and let $U:\Omega_1\times[-2,0)\to\R^N$ be a classical solution on compact subcylinders of $\Omega_1\times(-2,0)$ of
    \begin{equation}\label{eq:diagonal-parabolic-system}
        \partial_tU^A
        =a^{ij}(x,t)\partial_i\partial_jU^A
        +B^{iA}_{\phantom{iA}B}(x,t)\partial_iU^B
        +C^A_{\phantom{A}B}(x,t)U^B+F^A.
    \end{equation}
    Assume that $a^{ij}(x,t)\xi_i\xi_j\geq\kappa|\xi|^2$ for some $\kappa>0$ and that, uniformly for $t\in[-2,0)$, the spatial $C^{k,\alpha}(\Omega_1)$ norms of all the displayed coefficients are bounded by $\Lambda$.
    The coefficients and $F$ are measurable in $t$; no H\"older continuity in $t$ is assumed.
    Then, for every $T\in(-1,0]$ for which the right-hand side below is finite,
    \begin{equation}\label{ineq:spatial-Holder-parabolic-estimate}
        \sup_{-1\leq t<T}\norm{U(t)}_{C^{k+2,\alpha}(\Omega_0)}
        \leq C\sup_{-2\leq t<T}
        \left(\norm{U(t)}_{C^0(\Omega_1)}
        +\norm{F(t)}_{C^{k,\alpha}(\Omega_1)}\right),
    \end{equation}
    where $C$ depends only on $n,N,k,\alpha,\kappa,\Lambda$ and the two domains.
\end{Lemma}
\begin{proof}
    For $N=1$ and $k=0$, this is Brandt's estimate \cite{brandt1969interior}; see also \cite[Theorem~3.6]{chodosh2021uniqueness}.
    For a system with scalar principal symbol, apply that scalar estimate to each component, placing the mixed first- and zeroth-order terms on the right.
    On nested coordinate domains, the interpolation inequality
    \begin{equation*}
        \norm{DU}_{C^{0,\alpha}}
        +\norm{U}_{C^{0,\alpha}}
        \leq\epsilon\norm{U}_{C^{2,\alpha}}
        +C_\epsilon\norm{U}_{C^0}
    \end{equation*}
    absorbs these terms after choosing $\epsilon$ in terms of $\Lambda$.
    This proves the assertion for $k=0$.
    Spatial difference quotients in successively smaller domains satisfy systems of the same form; the new right-hand sides contain only spatial derivatives of the coefficients and derivatives of $U$ already controlled at the preceding stage.
    Induction on $k$, followed at each stage by the same interpolation and absorption, proves \cref{ineq:spatial-Holder-parabolic-estimate}.
\end{proof}

\begin{Proposition}\label{prop:parabolic-Schauder-estimate}
    Let $(M^n,g,f)$ be a complete AC Ricci shrinker.
    Given $k\in\mathbb Z_{\ge0}$, $\alpha\in(0,1)$ and $\lambda\in\R$, consider
    \begin{equation*}
        L=\Delta_f-\lambda+\nabla_X+\phi,
    \end{equation*}
    where
    \begin{equation*}
        \phi\in C^{k,\alpha}_{-2}(M,\End(T^{p,q}M)),
        \qquad
        X\in C^{k,\alpha}_{-1}(M,TM).
    \end{equation*}
    Suppose $u\in L^2_f(M,T^{p,q}M)$ is a weak solution of
    \begin{equation*}
        Lu=v\in C^{k,\alpha}_{-2\lambda-2}(M,T^{p,q}M).
    \end{equation*}
    For the possibly nonsmooth coefficient $X$, weak solution means that
    $u\in H^1_{\mathrm{loc}}$ and the equation holds distributionally.
    Then
    \begin{equation*}
        u\in
        C^{k+2,\alpha}_{-2\lambda;-2\lambda-2,\nabla f}
        (M,T^{p,q}M),
    \end{equation*}
    and
    \begin{equation}\label{ineq:parabolic-Schauder-estimate}
        \norm{u}_{C^{k+2,\alpha}_{-2\lambda;-2\lambda-2,\nabla f}(M,T^{p,q}M)}
        \leq
        C\big(  \norm{u}_{L^2_f(M,T^{p,q}M)}
        +\norm{Lu}_{C^{k,\alpha}_{-2\lambda-2}(M,T^{p,q}M)}  \big).
    \end{equation}
    Here $C$ depends on the fixed standard AC data, $n,k,\alpha,p,q,\lambda$, and the coefficient norms $\norm{\phi}_{C^{k,\alpha}_{-2}}$ and $\norm{X}_{C^{k,\alpha}_{-1}}$.
\end{Proposition}
\begin{proof}
    The source $v$ lies in $L^2_f$ by
    \cref{eq:Holder-to-Gaussian-Sobolev}, and $X,\phi$ are bounded.
    Testing the weak equation against $\eta_R^2u$ and absorbing the
    first-order term gives
    \[
        \|\eta_R\nabla u\|_{L^2_f}^2
        \leq C\bigl(\|v\|_{L^2_f}^2+\|u\|_{L^2_f}^2
        +\|u\nabla\eta_R\|_{L^2_f}^2\bigr).
    \]
    Letting $R\to\infty$ yields $u\in H^1_f$.
    Since $\Delta_fu=v+\lambda u-\nabla_Xu-\phi u\in L^2_f$,
    \cref{prop:elliptic-theory-for-weighted-Laplacian} gives
    $u\in H^2_f$, with its norm bounded by
    $C(\|u\|_{L^2_f}+\|v\|_{L^2_f})$.
    Local elliptic regularity then gives $u\in C^{k+2,\alpha}_{\mathrm{loc}}$.
    On a fixed compact core, the local $L^2$-to-$C^0$ estimate bounds
    $\|u\|_{C^0}$ by $C(\|u\|_{L^2_f}+\|v\|_{C^0_{-2\lambda-2}})$.
    Applying \cref{cor:AC-C0-estimate} with $\beta=-2\lambda-2$ and
    $\delta=2$ therefore gives
    \begin{equation}\label{ineq:L2-to-critical-C0}
        \|u\|_{C^0_{-2\lambda}}
        \leq C\bigl(\|u\|_{L^2_f}
        +\|v\|_{C^{k,\alpha}_{-2\lambda-2}}\bigr).
    \end{equation}

    Denote by $A_{\rho_1,\rho_2}$ the annulus $\iota((\rho_1,\rho_2)\times\Sigma)$ and recall $Q_\rho:=A_{\rho,\infty}$.
    Choose $R>0$ sufficiently large that $R/\sqrt2>\underline r$, and introduce the nested fixed annuli
    \begin{equation*}
        \Omega_0:=A_{2R,3R},
        \qquad
        \Omega_1:=A_{3R/2,7R/2},
        \qquad
        \Omega_2:=A_{R,4R}.
    \end{equation*}
    Let $\{\Psi_t\}_{-\infty<t<0}$ be the family of diffeomorphisms generated by $\frac{1}{-t}\nabla f$ and $\tilde{g}(t)=-t\Psi_t^*g$ as in \cref{sec:Ricci-flow}.
    Define
    \begin{align*}
        &\tilde{u}(x,t):=(-t)^{\frac{q-p}{2}-\lambda}\Psi_t^*u(x),
        &&\tilde{\phi}(x,t):=(-t)^{-1}\Psi_t^*\phi(x),  \\
        &\tilde{v}(x,t):=(-t)^{\frac{q-p}{2}-1-\lambda}\Psi_t^*v(x),
        &&\tilde X(x,t):=(-t)^{-1}\Psi_t^*X(x).
    \end{align*}
    Naturality of tensor norms and the constant rescaling of the metric give
    \begin{equation}\label{eq:dilated-spatial-derivatives}
        |(\nabla^{\tilde g})^j\tilde u|_{\tilde g}(x,t)
        =(-t)^{-\lambda-j/2}|\nabla^ju|_g(\Psi_t(x))
        \qquad(j\geq0).
    \end{equation}
    Since $r(\Psi_t(x))=r(x)/\sqrt{-t}$ in the canonical end chart, the
    dilation acts as translation in the logarithmic cylindrical variable. The definition in \cref{def:weighted-Holder-spaces} and the equivalent rescaled-annulus norms in \cref{eq:weighted-Holder-rescaled-annulus-norm} therefore give
    \begin{align}
        \sup_{-2\leq t<0}\norm{\tilde u(t)}_{C^0(\Omega_2,\tilde g(t))}
        &\leq C\norm{u}_{C^0_{-2\lambda}},\label{ineq:dilated-u-bound}\\
        \sup_{-2\leq t<0}\norm{\tilde v(t)}_{C^{k,\alpha}(\Omega_2,\tilde g(t))}
        &\leq C\norm{v}_{C^{k,\alpha}_{-2\lambda-2}},\label{ineq:dilated-v-bound}\\
        \sup_{-2\leq t<0}\left(
        \norm{\tilde X(t)}_{C^{k,\alpha}(\Omega_2,\tilde g(t))}
        +\norm{\tilde\phi(t)}_{C^{k,\alpha}(\Omega_2,\tilde g(t))}\right)
        &\leq C\left(\norm{X}_{C^{k,\alpha}_{-1}}+\norm{\phi}_{C^{k,\alpha}_{-2}}\right).
        \label{ineq:dilated-coefficient-bound}
    \end{align}
    For example,
    \begin{equation*}
        |\tilde X|_{\tilde g}=(-t)^{-1/2}|X|_g\circ\Psi_t,
        \qquad
        |\tilde\phi|=(-t)^{-1}|\phi|\circ\Psi_t,
    \end{equation*}
    and the factors on the right are exactly cancelled by the weights $-1$ and $-2$; the derivative and H\"{o}lder bounds follow from the same scaling calculation.

    A direct computation gives
    \begin{align*}
        \partial_t\tilde{u}
        &=
        (-t)^{\frac{q-p}{2}-1-\lambda}
        \Psi_t^*\Big(  \lambda u-\frac{q-p}{2}u+\mathcal{L}_{\nabla f}u  \Big)  \\
        &=
        (-t)^{\frac{q-p}{2}-1-\lambda}
        \Psi_t^*\Big(  \Delta u-\Delta_fu+\lambda u+\mathcal{L}_{\nabla f}u-\nabla_{\nabla f}u-\frac{q-p}{2}u  \Big)  \\
        &=
        (-t)^{\frac{q-p}{2}-1-\lambda}
        \Psi_t^*\Big(  \Delta u+\nabla_Xu+\phi u-v+\mathcal{L}_{\nabla f}u-\nabla_{\nabla f}u-\frac{q-p}{2}u  \Big).
    \end{align*}
    Using \cref{eq:difference-of-Lie-and-covariant-derivative} and the soliton equation \cref{eq:Ricci-shrinker}, we obtain
    \begin{equation*}
        \mathcal{L}_{\nabla f}u-\nabla_{\nabla f}u-\frac{q-p}{2}u
        =-\Ric(g)* u,
    \end{equation*}
    where we write $\Ric(g)$ for the Ricci transformation (it is an endomorphism of the tangent bundle).
    Note also the fact $\Delta^{\tilde{g}}\tilde{u}=(-t)^{\frac{q-p}{2}-1-\lambda}\Psi_t^*(\Delta u)$.
    Consequently,
    \begin{equation*}
        \partial_t\tilde{u}
        =\Delta^{\tilde{g}}\tilde{u}
        +\nabla^{\tilde g}_{\tilde X}\tilde u
        +\tilde{\phi}\tilde{u}
        -\Ric(\tilde{g})*\tilde{u}
        -\tilde{v}.
    \end{equation*}

    \cref{prop:asymptotic-expansion-of-shrinker-metric} shows that $\tilde g(t)$ extends smoothly to $t=0$ on $\Omega_2$, with limit $g_C$.
    In fixed coordinate charts and bundle frames, the last displayed equation is therefore a system of the form \cref{eq:diagonal-parabolic-system}, with scalar principal symbol and with the uniform spatial coefficient bounds required by \cref{lem:spatial-Holder-parabolic-estimate}.
    A finite cover of $\overline{\Omega_0}$ and \cref{ineq:dilated-u-bound}--\cref{ineq:dilated-coefficient-bound} give
    \begin{equation}\label{ineq:fixed-annulus-dilated-Schauder}
        \sup_{-1\leq t<0}\norm{\tilde u(t)}_{C^{k+2,\alpha}(\Omega_0,\tilde g(t))}
        \leq C\left(\norm{u}_{C^0_{-2\lambda}}
        +\norm{v}_{C^{k,\alpha}_{-2\lambda-2}}\right).
    \end{equation}

    We now convert this fixed-annulus estimate back to the weighted norm.
    Fix $r_*\in(2R,3R)$.
    If $y=\iota(r,\theta)$ with $r\geq r_*$, set
    \begin{equation*}
        t:=-\left(\frac{r_*}{r}\right)^2,
        \qquad
        x:=\iota(r_*,\theta).
    \end{equation*}
    Then $t\in[-1,0)$ and $y=\Psi_t(x)$.
    Formula \cref{eq:dilated-spatial-derivatives} converts every derivative term in \cref{ineq:fixed-annulus-dilated-Schauder} into the corresponding term of the $C^{k+2,\alpha}_{-2\lambda}$ norm at $y$.
    The same scaling identifies the ordinary H\"older norm on each fixed annulus with that of the corresponding translated cylinder segment. Segments of any fixed positive length give an equivalent uniform cylindrical norm. Covering the end by the translated images of $\Omega_0$ therefore shows that \cref{ineq:fixed-annulus-dilated-Schauder} controls the full exterior $C^{k+2,\alpha}_{-2\lambda}$ norm.
    Interior elliptic Schauder estimates
    \cite[Theorem 6.2]{gilbarg2001elliptic} on the fixed compact complement give
    \begin{equation}\label{ineq:ordinary-critical-Schauder}
        \norm{u}_{C^{k+2,\alpha}_{-2\lambda}}
        \leq C\left(\norm{u}_{C^0_{-2\lambda}}
        +\norm{v}_{C^{k,\alpha}_{-2\lambda-2}}\right).
    \end{equation}

    Finally, the equation gives
    \begin{equation*}
        (\nabla_{\nabla f}+\lambda)u
        =\Delta u+\nabla_Xu+\phi u-v
        \in C^{k,\alpha}_{-2\lambda-2},
    \end{equation*}
    with the same bound.
    \cref{prop:characterization-of-cone-Holder-space} and its norm equivalence,
    followed by \cref{ineq:L2-to-critical-C0}, give
    \cref{ineq:parabolic-Schauder-estimate}.
\end{proof}

\Needspace{4\baselineskip}
\begin{Lemma}
    \label{lem:homogeneous-drift-operator-expansion}
    Put $E_\Sigma=T^{p,q}\mathscr C|_{\Sigma_1}$ and $\bar g=g_\Sigma$,
    and write $\bar\nabla$ for the Levi-Civita connection of $\bar g$.
    For a vector field $Y$ on $\Sigma$, let $\widetilde Y$ be its
    $r$-independent lift to $\mathscr C$. Define the induced connection
    $\mathcal D$ and its connection Laplacian on $E_\Sigma$ by
    \[
        \mathcal D_Ya
        :=\left.\nabla^C_{\widetilde Y}\boldsymbol\tau(a)\right|_{\Sigma_1},
        \qquad
        \mathcal A_\Sigma a
        :=\sum_{j=1}^{n-1}\left(
            \mathcal D_{\bar e_j}\mathcal D_{\bar e_j}a
            -\mathcal D_{\bar\nabla_{\bar e_j}\bar e_j}a\right),
    \]
    where $\{\bar e_j\}$ is a local $\bar g$-orthonormal frame.
    The connection $\mathcal D$ acts on all radial and tangential cone
    tensor components. For functions,
    $\mathcal A_\Sigma=\bar\Delta:=\tr_{\bar g}\bar\nabla^2$.

    Let $\gamma,\lambda\in\R$, $N,k\in\mathbb Z_{\ge0}$, $\alpha\in[0,1]$, and
    $a\in C^{k+2,\alpha}(\Sigma,E_\Sigma)$.
    Then the cone Laplacian satisfies
    \[
        \Delta_C\bigl(r^\gamma\boldsymbol\tau(a)\bigr)
        =r^{\gamma-2}\boldsymbol\tau
        \bigl(\mathcal A_\Sigma a+\gamma(\gamma+n-2)a\bigr).
    \]
    If
    \begin{equation*}
        h:=\chi r^\gamma\boldsymbol\tau(a),
    \end{equation*}
    then there are differential operators
    $\mathcal B_{\gamma,\lambda,\ell}$ on the link, of order at most two,
    such that
    \begin{equation}\label{eq:homogeneous-drift-operator-expansion}
        (\Delta_f-\lambda)h
        =\chi\sum_{\ell=0}^{N}
        r^{\gamma-2\ell}
        \boldsymbol\tau(\mathcal B_{\gamma,\lambda,\ell}a)
        +R_N,
        \qquad
        R_N\in C^{k,\alpha}_{\gamma-2N-2}.
    \end{equation}
    The leading operator is
    \begin{equation}\label{eq:leading-homogeneous-drift-action}
        \mathcal B_{\gamma,\lambda,0}
        =\left(-\frac\gamma2-\lambda\right)\operatorname{Id}.
    \end{equation}
    The link coefficients and the remainder satisfy the corresponding norm
    bounds in terms of $a$, the displayed parameters, and the fixed standard AC data.
\end{Lemma}
\begin{proof}
    The standard cone connection formulas are
    \[
        \nabla^C_{\widetilde Y}\partial_r=r^{-1}\widetilde Y,
        \qquad
        \nabla^C_{\widetilde Y}(r^{-1}\widetilde Z)
        =r^{-1}\widetilde{\bar\nabla_YZ}-\bar g(Y,Z)\partial_r.
    \]
    They imply
    $\left.\nabla^C_{\widetilde Y}\boldsymbol\tau(a)\right|_{\Sigma_r}
       =\boldsymbol\tau_r(\mathcal D_Ya)$.
    Tracing the second covariant derivative in the cone orthonormal frame
    $\partial_r,r^{-1}\widetilde{\bar e_j}$ gives the radial operator
    $\partial_r^2+(n-1)r^{-1}\partial_r$ and angular operator
    $r^{-2}\mathcal A_\Sigma$ under radial parallel transport.
    Applying this to $r^\gamma a$ proves the homogeneous cone formula.

    Work first where $\chi=1$ and put $h_\gamma=r^\gamma\boldsymbol\tau(a)$.
    The canonical radial identity in
    \cref{lem:difference}\cref{lem:difference-radial-derivative} is
    $\nabla^g_{\nabla f}=\nabla^C_{\mathbf X}+\Ric(g)*$,
    where $\mathbf X=r\partial_r/2$. Since $\boldsymbol\tau(a)$ is radially
    parallel, it follows that
    \[
    (\Delta_f-\lambda)h_\gamma
        =-\left(\frac\gamma2+\lambda\right)h_\gamma
        +r^{\gamma-2}\boldsymbol\tau
        \bigl(\mathcal A_\Sigma a+\gamma(\gamma+n-2)a\bigr)
        -\Ric(g)*h_\gamma+(\Delta_g-\Delta_C)h_\gamma.
    \]

    Let $T=\nabla^g-\nabla^C$. The difference formula in
    \cref{lem:difference} has the form
    \[
    (\Delta_g-\Delta_C)h_\gamma
        =(g^{-1}-g_C^{-1})*\nabla_C^2h_\gamma
        +g^{-1}*(T*\nabla_Ch_\gamma+\nabla_CT*h_\gamma+T*T*h_\gamma).
    \]
    Since $g^{-1}-g_C^{-1}$ has weight $-2$ and $T$ has weight $-3$,
    this expression starts at degree $\gamma-4$.
    If $\mathcal R_\Sigma$ denotes the action induced by $\Ric(g_C)$ on
    $T^{p,q}\mathscr C|_\Sigma$, using the $\Ric*$ convention of
    \cref{lem:difference}\cref{lem:difference-radial-derivative}, then
    \[
        \mathcal B_{\gamma,\lambda,0}
        =-\left(\frac\gamma2+\lambda\right)\operatorname{Id},
        \qquad
        \mathcal B_{\gamma,\lambda,1}
        =\mathcal A_\Sigma+\gamma(\gamma+n-2)\operatorname{Id}-\mathcal R_\Sigma.
    \]
    The Ricci action vanishes for scalar functions.

    More generally, \cref{prop:asymptotic-expansion-of-shrinker-metric}
    expands $g$ in even degree drops with differentiated remainders.
    A finite Neumann expansion of $g^{-1}$, followed by the connection
    difference formula, gives corresponding expansions of $T$, $\nabla_CT$,
    and $\Ric(g)*$. Substitute these into the preceding identities and
    collect the terms of degree $\gamma-2\ell$. Each coefficient differentiates
    $a$ at most twice and therefore defines an operator
    $\mathcal B_{\gamma,\lambda,\ell}$ of order at most two on the link.
    The differentiated remainders and the weighted product estimates give
    \[
        \|R_N\|_{C^{k,\alpha}_{\gamma-2N-2}}
        \leq C\|a\|_{C^{k+2,\alpha}(\Sigma)},
        \qquad N\in\mathbb Z_{\ge0}.
    \]
    Restoring $\chi$ adds only terms supported in a fixed compact annulus,
    which satisfy the same bound. This proves the expansion and its norm estimates.
\end{proof}

\begin{Lemma}\label{lem:noncharacteristic-radial-extraction}
    Let $k\in\mathbb Z_{\ge0}$, $\alpha\in(0,1)$, $\lambda\in\mathbb R$, and
    $\gamma=-2\lambda-2j$ with $j\geq1$ an integer. Put
    $E_\Sigma=T^{p,q}\mathscr C|_\Sigma$.
    Suppose $w\in C^{k+2,\alpha}_\gamma(M,T^{p,q}M)$ and
    $b\in C^{k+2,\alpha}(\Sigma,E_\Sigma)$ satisfy
    \[
        (\Delta_f-\lambda)w=\chi r^\gamma\boldsymbol\tau(b)+F,
        \qquad F\in C^{k,\alpha}_{\gamma-2}.
    \]
    Then
    \[
        z:=w-\frac1j\chi r^\gamma\boldsymbol\tau(b)
        \in C^{k+2,\alpha}_{\gamma-2,\nabla f},
        \qquad
        \|z\|_{C^{k+2,\alpha}_{\gamma-2,\nabla f}}
        \leq C\bigl(\|w\|_{C^{k+2,\alpha}_\gamma}
        +\|b\|_{C^{k+2,\alpha}(\Sigma)}
        +\|F\|_{C^{k,\alpha}_{\gamma-2}}\bigr).
    \]
\end{Lemma}
\begin{proof}
    On the region where $\chi=1$, set
    $y(r)=r^{-\gamma}\boldsymbol\tau_r^{-1}(w|_{\Sigma_r})$.
    The identity $\nabla^g_{\nabla f}=\nabla^C_{\mathbf X}+\Ric(g)*$
    used in \cref{lem:homogeneous-drift-operator-expansion} gives
    \[
        \frac r2\partial_r y-jy=-b+e(r),
        \qquad
        e(r)=r^{-\gamma}\boldsymbol\tau_r^{-1}
        \bigl((\Delta w-\Ric(g)*w-F)|_{\Sigma_r}\bigr).
    \]
    The ordinary weighted hypothesis on $w$ places $\Delta w$ in
    $C^{k,\alpha}_{\gamma-2}$. Thus $e$ has weight $-2$ with
    $k$ cylindrical derivatives and the corresponding H\"older bound.
    Since $y$ is bounded and $j>0$, integration from infinity yields
    \[
        y(r)-\frac bj
        =-2r^{2j}\int_r^\infty t^{-2j-1}e(t)\,\mathrm dt.
    \]
    The radial integration estimates of
    \cref{lem:characterization-of-classic-weighted-Holder-space},
    or the same calculation after $t=re^\ell$, give
    $z\in C^{k,\alpha}_{\gamma-2}$ with the asserted bound at this order.
    The fixed cutoff collar is controlled by the stated norms.

    By \cref{lem:homogeneous-drift-operator-expansion},
    $(\Delta_f-\lambda)z\in C^{k,\alpha}_{\gamma-2}$.
    Consequently
    $(\Delta_f-(\lambda+j))z\in C^{k,\alpha}_{\gamma-2}$.
    Apply \cref{prop:parabolic-Schauder-estimate} with parameter
    $\lambda+j=-\gamma/2$; polynomial weighted sections lie in $L^2_f$
    by \cref{eq:Holder-to-Gaussian-Sobolev}.
    Its degree-$\gamma$ coefficient vanishes because
    $z\in C^0_{\gamma-2}$. The remaining component is in
    $C^{k+2,\alpha}_{\gamma-2,\nabla f}$, with the required estimate.
\end{proof}

\begin{Remark}\label{rmk:cone-model}
    Let $E=T^{p,q}\mathscr C$, put $E_\Sigma=E|_{\Sigma_1}$, and let
    $\mathcal A_\Sigma$ be the full angular operator defined in
    \cref{lem:homogeneous-drift-operator-expansion}.
    It includes the radial--tangential coupling terms, whereas
    $\bar\Delta$ denotes only the intrinsic link Laplacian.
    On the cone with potential $f=r^2/4$, separation of variables gives
    \[
        \Delta_f\bigl(\varphi\boldsymbol\tau(\bar u)\bigr)
        =\left(\varphi''+\left(\frac{n-1}{r}-\frac r2\right)\varphi'
        -\frac\nu{r^2}\varphi\right)\boldsymbol\tau(\bar u),
        \qquad-\mathcal A_\Sigma\bar u=\nu\bar u.
    \]
    For a homogeneous tensor $u_\gamma=r^\gamma\boldsymbol\tau(a)$,
    \[
        (\Delta_f-\lambda)u_\gamma
        =-\left(\frac\gamma2+\lambda\right)u_\gamma
        +r^{\gamma-2}\boldsymbol\tau
        \bigl(\mathcal A_\Sigma a+\gamma(\gamma+n-2)a\bigr).
    \]
    The drift preserves degree, whereas the cone Laplacian lowers it by two.
    At the characteristic degree $\gamma=-2\lambda$, the first term vanishes.
    This explains the homogeneous coefficients and two-degree steps in
    \cref{prop:finite-expansion}, which establishes the expansion and
    estimates on an AC shrinker.
\end{Remark}

\subsection{Schauder estimates II}
Let $(M^n,g,f)$ be a complete AC Ricci shrinker and $E$ be a tensor bundle over $M$.
Consider the operator $L=\Delta_f-\lambda+\nabla_X+\phi$ acting on sections of $E$, where $X$ is a vector field and $\phi$ is an endomorphism of $E$.
In what follows, $\lambda,\beta\in\R$, $k\in\mathbb Z_{\ge0}$, $\alpha\in(0,1)$ and $\Lambda\geq 0$ are constants.
As before, constants may depend on the fixed standard AC data, the displayed parameters, and the coefficient bounds.

\begin{Lemma}\label{lem:weak-estimate-by-induction}
    Assume $2\lambda+\beta>0$.
    Suppose
    $\norm{X}_{C^{k+1}_{-1}(M,TM)},
    \norm{\phi}_{C^{k+1}_{-2}(M,\End E)}\leq\Lambda$.
    If $u\in L^2_f(M,E)$ and $Lu\in C^{k,\alpha}_{\beta}(M,E)$, then
    \begin{equation}\label{ineq:weak-estimate}
        \norm{u}_{C^{k,\alpha}_{\beta}(M,E)}
        \leq C_L
        \left(\norm{u}_{L^2_f(M,E)}
        +\norm{Lu}_{C^{k,\alpha}_{\beta}(M,E)}\right).
    \end{equation}
    If $\Ker(L:H^2_f\to L^2_f)=\{0\}$, the $L^2_f$ term can be omitted.
\end{Lemma}
\begin{proof}
    We first prove an integer-order operator estimate.  For every
    $0\leq j\leq k+1$ and every $L^2_f$-solution whose right-hand side lies
    in $C^j_\beta$, there exists $\underline r$ depending on the displayed
    data such that
    \begin{equation}\label{ineq:weak-result-for-induction}
        \norm{u}_{C^{j}_{\beta}(M,E)}
        \leq C\left(\norm{u}_{C^{j}(B_{\underline r}(p_0))}
        +\norm{Lu}_{C^{j}_{\beta}(M,E)}\right).
    \end{equation}
    \cref{prop:Gaussian-Fredholm-realization} gives $u\in H^2_f$,
    and local elliptic regularity gives $u\in W^{2,p}_{\mathrm{loc}}$
    for every $p>n$. Thus the weak-solution version of
    \cref{cor:AC-C0-estimate} proves the case $j=0$. Assume the result
    through order $j-1$.
    Set $v:=Lu$.
    Put $V=\nabla u$. By \cref{eq:exact-weighted-gradient-commutator},
    \[
        (\Delta_f-\lambda-\frac12)V+\nabla_XV+\Psi V=v',
    \]
    where, in a normal orthonormal frame,
    \[
    \begin{aligned}
        (\Psi V)_i&=\phi V_i+(\nabla_iX^a)V_a-2\mathcal R^E_{ai}V_a,\\
        v'_i&=\nabla_i v+(X^a\mathcal R^E_{ai}-\nabla_i\phi)u.
    \end{aligned}
    \]
    Thus $\Psi$ has weight $-2$, and the coefficient multiplying $u$ in
    $v'$ has weight $-3$, with the derivative bounds needed below.
    Write this equation as $L'\nabla u=v'$.
    By \cref{lem:asymptotic-rate-of-connection} and the induction hypothesis (applied to $u$),
    \begin{equation*}
        \norm{v'}_{C^{j-1}_{\beta-1}(M,T^*M\otimes E)}
        \leq\norm{\nabla v}_{C^{j-1}_{\beta-1}(M,T^*M\otimes E)}
        +C\norm{u}_{C^{j-1}_{\beta}(M,E)}
        \leq C\left(\norm{u}_{C^{j}(B_{\underline r}(p_0))}
        +\norm{v}_{C^{j}_{\beta}(M,E)}\right).
    \end{equation*}
    Since the operator $L'$ also satisfies the condition of the induction hypothesis, we have
    \begin{equation*}
        \norm{\nabla u}_{C^{j-1}_{\beta-1}(M,T^*M\otimes E)}
        \leq C\left(\norm{\nabla u}_{C^{j-1}(B_{\underline r}(p_0))}
        +\norm{v'}_{C^{j-1}_{\beta-1}(M,T^*M\otimes E)}\right)
        \leq C\left(\norm{u}_{C^{j}(B_{\underline r}(p_0))}
        +\norm{v}_{C^{j}_{\beta}(M,E)}\right).
    \end{equation*}
    This proves \cref{ineq:weak-result-for-induction}.  Interior elliptic
    estimates on the fixed ball and
    \cref{eq:Gaussian-Fredholm-alternative} now give, for $j\leq k+1$,
    \begin{equation}\label{eq:integer-weighted-estimate-with-slush}
        \norm{u}_{C^j_\beta}
        \leq C\left(\norm{u}_{L^2_f}+\norm{Lu}_{C^j_\beta}\right).
    \end{equation}

    Let $\mathcal N=\Ker(L:H^2_f\to L^2_f)$. We first record the
    estimate on its $L^2_f$-orthogonal complement, with a constant
    $C_L$ depending on the fixed operator:
    \begin{equation}\label{eq:Gaussian-kernel-complement-estimate}
        \|u\|_{H^2_f}\leq C_L\|Lu\|_{L^2_f},
        \qquad u\in H^2_f\cap\mathcal N^\perp.
    \end{equation}
    Indeed, failure of the corresponding $L^2_f$ bound would give
    $u_i\perp\mathcal N$ with $\|u_i\|_{L^2_f}=1$ and
    $Lu_i\to0$ in $L^2_f$. The graph estimate in
    \cref{eq:Gaussian-Fredholm-alternative} bounds $u_i$ in $H^2_f$.
    By \cref{lem:Gaussian-Rellich}, a subsequence converges weakly in
    $H^2_f$ and strongly in $L^2_f$ to $u_\infty$.
    Then $Lu_\infty=0$, $u_\infty\perp\mathcal N$, and
    $\|u_\infty\|_{L^2_f}=1$, a contradiction. The graph estimate now
    proves \cref{eq:Gaussian-kernel-complement-estimate}.
    Combining this with \cref{eq:integer-weighted-estimate-with-slush}
    and \cref{eq:Holder-to-Gaussian-Sobolev} gives
    \begin{equation}\label{eq:integer-weighted-kernel-complement-estimate}
        \|u\|_{C^j_\beta}\leq C_L\|Lu\|_{C^j_\beta},
        \qquad u\perp\mathcal N,\quad 0\leq j\leq k+1.
    \end{equation}

    We now pass to the H\"older exponent. Put $\mathcal N^*=\Ker L^*$.
    By \cref{prop:Gaussian-Fredholm-realization},
    $\dim\mathcal N=\dim\mathcal N^*<\infty$ and
    $\operatorname{Ran}L=(\mathcal N^*)^\perp$.
    Applying \cref{eq:integer-weighted-estimate-with-slush} to kernel
    elements for $0\leq j\leq k+1$, then interpolating, gives
    $\mathcal N\subset C^{k,\alpha}_\beta$ with norm bounded by the
    $L^2_f$ norm. Thus the $L^2_f$-orthogonal projection satisfies
    \begin{equation}\label{eq:kernel-projection-weighted-bound}
        \|\Pi_{\mathcal N}u\|_{C^{k,\alpha}_\beta}
        \leq C\|u\|_{L^2_f}.
    \end{equation}

    Choose an $L^2_f$-orthonormal basis $z_1,\ldots,z_N$ of
    $\mathcal N^*$ and compactly supported smooth sections
    $e_1,\ldots,e_N$ such that
    $\inn{e_a}{z_b}_{L^2_f}=\delta_{ab}$. Such sections exist because
    the pairing map $C_c^\infty(M,E)\to\mathbb R^N$ is surjective;
    otherwise a nonzero linear combination of the $z_b$ would vanish
    as a distribution. Set $\mathcal E=\operatorname{span}\{e_1,\ldots,e_N\}$
    and define the projection onto $\operatorname{Ran}L$ along $\mathcal E$ by
    \[
        \Pi_{\operatorname{Ran}L}^{\mathcal E}v
        =v-\sum_{a=1}^N\inn{v}{z_a}_{L^2_f}e_a.
    \]
    This projection is the identity on $\operatorname{Ran}L$.
    It is bounded on $C^k_\beta$ and $C^{k+1}_\beta$, since polynomial
    weighted sections belong to $L^2_f$ and the $e_a$ have compact support.

    Let $G:\operatorname{Ran}L\to\mathcal N^\perp\cap H^2_f$ be the inverse
    furnished by \cref{eq:Gaussian-kernel-complement-estimate}.
    The estimate \cref{eq:integer-weighted-kernel-complement-estimate} gives
    \[
        G\Pi_{\operatorname{Ran}L}^{\mathcal E}:C^j_\beta\longrightarrow C^j_\beta
        \quad\text{boundedly for }j=k,k+1.
    \]
    Interpolating this same operator by \cref{lem:weighted-real-interpolation}
    makes it bounded on $C^{k,\alpha}_\beta$. Since
    $u-\Pi_{\mathcal N}u=G\Pi_{\operatorname{Ran}L}^{\mathcal E}(Lu)$,
    we obtain
    \[
        \|u-\Pi_{\mathcal N}u\|_{C^{k,\alpha}_\beta}
        \leq C_L\|Lu\|_{C^{k,\alpha}_\beta}.
    \]
    Together with \cref{eq:kernel-projection-weighted-bound}, this proves
    \cref{ineq:weak-estimate}. If $\mathcal N=0$, index zero also gives
    $\mathcal N^*=0$, so $\Pi_{\mathcal N}=0$ and
    $\Pi_{\operatorname{Ran}L}^{\mathcal E}=\operatorname{Id}$. The estimate
    then contains no $L^2_f$ term.
\end{proof}

The following result parallels \cite[Proposition 5.33]{bamler2023degree}.
\begin{Proposition}\label{prop:weighted-Schauder-estimates-BCD}
    Assume $2\lambda+\beta>0$ and $\norm{X}_{C^{k+1}_{-1}(M,TM)},\norm{\phi}_{C^{k+1}_{-2}(M,\End E)}\leq\Lambda$.
    For any $u\in L^2_f(M,E)$ satisfying $Lu\in C^{k,\alpha}_{\beta}(M,E)$, we have
    \begin{equation*}
        \norm{u}_{C^{k+2,\alpha}_{\beta,\nabla f}(M,E)}
        \leq C_L
        \Big(  \norm{u}_{L^2_f(M,E)}  +\norm{Lu}_{C^{k,\alpha}_{\beta}(M,E)}  \Big).
    \end{equation*}
    Moreover,
    \begin{equation*}
        L:C^{k+2,\alpha}_{\beta,\nabla f}(M,E)
        \longrightarrow C^{k,\alpha}_{\beta}(M,E)
    \end{equation*}
    is a bounded Fredholm operator of index zero.  Its range is
    \begin{equation}\label{eq:weighted-Holder-Fredholm-range}
        C^{k,\alpha}_\beta(M,E)\cap(\Ker L^*)^\perp,
    \end{equation}
    where $L^*$ is the adjoint of the Gaussian Sobolev realization in
    \cref{prop:Gaussian-Fredholm-realization}.  In particular, if
    $\Ker(L:H^2_f\to L^2_f)=\{0\}$, this weighted H\"older operator is an
    isomorphism and
    \begin{equation*}
        \norm{u}_{C^{k+2,\alpha}_{\beta,\nabla f}}
        \leq C_L\norm{Lu}_{C^{k,\alpha}_\beta}.
    \end{equation*}
\end{Proposition}
\begin{proof}
    \cref{lem:weak-estimate-by-induction} gives
    $u\in C^{k,\alpha}_\beta$.  Put $\lambda':=-1-\beta/2$.  Then
    $(L+\lambda-\lambda')u\in C^{k,\alpha}_{\beta}$ and
    $u\in C^0_\beta\subset C^0_{\beta+2}=C^0_{-2\lambda'}$.
    Local elliptic Schauder regularity gives
    $u\in C^{k+2,\alpha}_{\mathrm{loc}}$.
    It follows from \cref{prop:parabolic-Schauder-estimate} that $u\in C^{k+2,\alpha}_{\beta+2;\beta,\nabla f}(M,E)$.
    Since $u\in C^{k,\alpha}_{\beta}(M,E)$, the degree $\beta+2$ homogeneous component necessarily vanishes, meaning that $u\in C^{k+2,\alpha}_{\beta,\nabla f}(M,E)$.

    Let $L:H^2_f\to L^2_f$ be the realization of \cref{prop:Gaussian-Fredholm-realization}.  By
    \cref{eq:Holder-to-Gaussian-Sobolev}, the weighted H\"older operator is
    a restriction of $L$.  Conversely, the estimate just proved
    shows that every Sobolev solution of $Lu=v$ with
    $v\in C^{k,\alpha}_\beta$ belongs to the displayed H\"older domain.
    The Sobolev Fredholm alternative therefore gives exactly
    \cref{eq:weighted-Holder-Fredholm-range}.  The functionals defined by
    $\Ker L^*$ are continuous on $C^{k,\alpha}_\beta$ by
    \cref{eq:Holder-to-Gaussian-Sobolev}; hence the range is closed and has
    codimension $\dim\Ker L^*$.  Its kernel is the Sobolev kernel, so its index
    is
    \begin{equation*}
        \dim\Ker L-\dim\Ker L^*=0.
    \end{equation*}
    The final assertion follows from index zero and the estimate in \cref{lem:weak-estimate-by-induction}.
\end{proof}

For the ordinary Laplacian, finite expansions at infinity are organized by homogeneous harmonic terms; see \cite[Theorem~1.17]{bartnik1986mass} for the asymptotically flat case. On a general cone, the scalar indicial equation is
\[
    \Delta_C(r^\gamma a)
    =r^{\gamma-2}\bigl(\bar\Delta a+\gamma(\gamma+n-2)a\bigr).
\]
Thus the indicial roots are determined by the spectrum of $-\bar\Delta$ on the link. For $\Delta_f-\lambda$, the leading homogeneous multiplier is instead $-\gamma/2-\lambda$, by \cref{eq:leading-homogeneous-drift-action}. The characteristic exponent is $-2\lambda$, and the lower-order terms generate successive drops by two. The finite set $\Lambda_m$ below is the resulting asymptotic index set, playing the role analogous to the index sets in the classical expansion. Its elements beyond $-2\lambda$ are not additional roots of the leading multiplier; their coefficients are determined recursively.

\begin{Proposition}[Finite expansion]\label{prop:finite-expansion}
    Let $(M^n,g,f)$ be a complete AC Ricci shrinker with fixed standard AC data. Given $k\in\mathbb Z_{\ge0}$, $\alpha\in(0,1)$, $\lambda\in\R$, and $m\in\mathbb Z_{\ge0}$, set $\Lambda_m:=\{-2\lambda-2j:0\leq j\leq m,j\in\Z\}$.
    Suppose $u\in L^2_f(M,T^{p,q}M)$ and, distributionally,
    \begin{equation*}
        v:=(\Delta_f-\lambda)u
        \in C^{k,\alpha}_{-2\lambda-2m-2}(M,T^{p,q}M).
    \end{equation*}
    Then $u\in C^{k+2,\alpha}_{\Lambda_{m};-2\lambda-2m-2,\nabla f}(M,T^{p,q}M)$.
    More precisely, its homogeneous part of degree $-2\lambda-2j$ belongs to
    \begin{equation*}
        C^{k+2+2(m-j),\alpha}_{-2\lambda-2j;-\infty}
        (M,T^{p,q}M)
        \qquad(0\leq j\leq m),
    \end{equation*}
    and
    \begin{equation*}
        \sum_{j=0}^{m}\norm{u_{-2\lambda-2j}}_{C^{k+2+2(m-j),\alpha}_{-2\lambda-2j;-\infty}(M,T^{p,q}M)}
        +\norm{u^\star}_{C^{k+2,\alpha}_{-2\lambda-2m-2,\nabla f}(M,T^{p,q}M)} \leq C\Big(
        \norm{u}_{L^2_f(M,T^{p,q}M)}
        +\norm{v}_{C^{k,\alpha}_{-2\lambda-2m-2}(M,T^{p,q}M)} \Big).
    \end{equation*}
\end{Proposition}
\begin{proof}
    Put $E_\Sigma=T^{p,q}\mathscr C|_\Sigma$ and write
    $\gamma_j=-2\lambda-2j$,
    $\beta=\gamma_{m+1}$, and
    $N=\|u\|_{L^2_f}+\|v\|_{C^{k,\alpha}_\beta}$.
    For $m=0$, this is \cref{prop:parabolic-Schauder-estimate}.
    We henceforth assume $m\geq1$.

    We first supply the spatial regularity needed for the radial induction.
    We use \cref{prop:weighted-Schauder-estimates-BCD} as follows.
    Choose $\mu>\max\{0,-\beta/2\}$ and let
    $R=(\Delta_f-\mu)^{-1}$ on $L^2_f$.
    This inverse exists by \cref{prop:elliptic-theory-for-weighted-Laplacian}.
    Repeated application of the cited estimate, together with the inclusion
    $C^{a+2,\alpha}_{\delta,\nabla f}\subset
    C^{a+2,\alpha}_{\delta+2}$, gives
    \begin{equation}\label{ineq:finite-expansion-resolvent-regularity}
        \|R^h v\|_{C^{k+2,\alpha}_{\beta,\nabla f}}
        +\|R^h v\|_{C^{k+2\ell,\alpha}_{\beta+2\ell}}
        \leq C\|v\|_{C^{k,\alpha}_\beta}
        \qquad(1\leq h\leq m,\ 0\leq\ell\leq h).
    \end{equation}
    Indeed, each application preserves $C^{k,\alpha}_\beta$ and gains
    two ordinary weighted derivatives at the cost of increasing the weight
    by two. The inequality $2\mu+\beta+2\ell>0$ holds at every step.

    Put $c=\lambda-\mu$ and define
    \[
        U=\sum_{h=1}^m c^{h-1}R^h v,
        \qquad u^\sharp=u-U,
        \qquad G=c^mR^m v.
    \]
    Since $L_\lambda R=I-cR$, these definitions give
    \begin{equation}\label{eq:finite-expansion-regularized-equation}
        L_\lambda u^\sharp=G,
        \qquad U\in C^{k+2,\alpha}_{\beta,\nabla f},
        \qquad G\in C^{k+2\ell,\alpha}_{\beta+2\ell}
        \quad(0\leq\ell\leq m),
    \end{equation}
    with all the indicated norms, and $\|u^\sharp\|_{L^2_f}$, bounded by $CN$.
    At $\ell=m$, the weight of $G$ is $\gamma_1$.
    Applying \cref{prop:parabolic-Schauder-estimate} to $u^\sharp$ at order
    $k+2m$ therefore gives
    \[
        u^\sharp=\chi r^{\gamma_0}\boldsymbol\tau(a_0)+w_1,
        \quad a_0\in C^{k+2m+2,\alpha}(\Sigma,E_\Sigma),
        \quad w_1\in C^{k+2m+2,\alpha}_{\gamma_1,\nabla f},
    \]
    again with norm bounded by $CN$. The leading exponent $\gamma_0=-2\lambda$
    is characteristic, so $a_0$ is supplied by
    \cref{prop:parabolic-Schauder-estimate};
    \cref{lem:noncharacteristic-radial-extraction} applies only to the
    subsequent exponents $\gamma_j$, $j\geq1$.

    We now extract one term at a time. Suppose $1\leq j\leq m$ and
    \[
        u^\sharp=\chi\sum_{i=0}^{j-1}r^{\gamma_i}\boldsymbol\tau(a_i)+w_j,
        \qquad w_j\in C^{k+2+2(m-j),\alpha}_{\gamma_j},
        \qquad a_i\in C^{k+2+2(m-i),\alpha}(\Sigma,E_\Sigma).
    \]
    The preceding decomposition establishes this for $j=1$, by the
    definition of the critical norm. Set $s=k+2(m-j)$.
    Using \cref{lem:homogeneous-drift-operator-expansion} and the
    case $\ell=m-j$ of \cref{eq:finite-expansion-regularized-equation},
    we obtain
    \[
        L_\lambda w_j=\chi r^{\gamma_j}\boldsymbol\tau(b_j)+F_j,
        \qquad F_j\in C^{s,\alpha}_{\gamma_j-2},
        \qquad b_j=-\sum_{i=0}^{j-1}
        \mathcal B_{\gamma_i,\lambda,j-i}a_i.
    \]
    The coefficients at the earlier degrees cancel by the induction
    relations. Since each $\mathcal B_{\gamma_i,\lambda,j-i}$ has order
    at most two, $b_j\in C^{s+2,\alpha}(\Sigma,E_\Sigma)$.
    \cref{lem:noncharacteristic-radial-extraction} yields
    \begin{equation}\label{eq:finite-expansion-coefficient-recursion}
        a_j=-\frac1j\sum_{i=0}^{j-1}
        \mathcal B_{\gamma_i,\lambda,j-i}a_i,
        \qquad
        w_{j+1}:=w_j-\chi r^{\gamma_j}\boldsymbol\tau(a_j)
        \in C^{s+2,\alpha}_{\gamma_j-2,\nabla f}.
    \end{equation}
    In particular $w_{j+1}\in C^{s,\alpha}_{\gamma_{j+1}}$,
    which is the ordinary weighted regularity required at the next step.
    The extraction and homogeneous-expansion estimates control every
    coefficient and remainder by $CN$ throughout this finite induction.

    At $j=m$ we have
    $w_{m+1}\in C^{k+2,\alpha}_{\beta,\nabla f}$.
    Thus
    \[
        u=\chi\sum_{j=0}^m r^{\gamma_j}\boldsymbol\tau(a_j)+u^\star,
        \qquad u^\star=U+w_{m+1}
        \in C^{k+2,\alpha}_{\beta,\nabla f}.
    \]
    The asserted coefficient regularities and the estimate follow from
    the bounds just proved and
    \cref{ineq:finite-expansion-resolvent-regularity}.
\end{proof}

\begin{Corollary}[Asymptotic expansion of eigentensors]\label{cor:asymptotic-expansion-of-eigentensor}
    Let $(M,g,f)$ be a complete AC Ricci shrinker and let
    $(\Delta_f-\lambda)u=0$ near infinity, with $u\in L^2_f$ there.
    There is a single sequence
    \[
        a_j\in C^\infty(\Sigma,E_\Sigma),\qquad j=0,1,\ldots,
        \qquad E_\Sigma=T^{p,q}\mathscr C|_\Sigma,
    \]
    such that, for every $N\in\mathbb Z_{\ge0}$,
    \[
        u-\chi\sum_{j=0}^{N}
        r^{-2\lambda-2j}\boldsymbol\tau(a_j)
        \in C^\infty_{-2\lambda-2N-2}.
    \]
    Here one fixed smooth cutoff extension of $u$ is used if $u$ is defined only on an exterior domain, and $C^\infty_\beta$ means that all conically weighted spatial derivatives have the indicated bounds.
    The sequence is independent of $N$, and every $a_j$ is determined recursively by $a_0$.
\end{Corollary}
\begin{proof}
    Interior elliptic regularity gives smoothness on the exterior domain.
    A fixed smooth cutoff extension belongs to $L^2_f$ and has a smooth
    compactly supported source.
    Apply \cref{prop:finite-expansion} at every finite order and every spatial regularity.
    If two such expansions first differ at degree $-2\lambda-2j$, their difference at that degree lies in $C^0_{-2\lambda-2j-2}$. Dividing by $r^{-2\lambda-2j}$ and restricting to radial rays forces the corresponding link coefficient to vanish. Hence the finite expansions agree on every common range and define one sequence $(a_j)_{j\geq0}$. Increasing the spatial regularity shows that each $a_j$ is smooth and that the stated remainder belongs to $C^\infty_{-2\lambda-2N-2}$. The recursion \cref{eq:finite-expansion-coefficient-recursion} determines every coefficient from $a_0$.
\end{proof}

\begin{Proposition}\label{prop:weighted-Schauder-estimates-intermediate}
    Let $(M^n,g,f)$ be a complete AC Ricci shrinker and $E$ be a tensor bundle over $M$.
    Let $k\in\mathbb Z_{\ge0}$, $\alpha\in(0,1)$ and $\lambda,\beta\in\R$ be such that $2\lambda+\beta\in(-2,0]$.
    For any $L^2_f$-solution $u$ to the equation $(\Delta_f-\lambda)u=v\in C^{k,\alpha}_{\beta}(M,E)$, we have
    \[
        \nabla^2u\in C^{k,\alpha}_{\beta}\bigl(M,(T^*M)^{\otimes2}\otimes E\bigr),
        \qquad
        (\nabla_{\nabla f}+\lambda)u\in C^{k,\alpha}_{\beta}(M,E),
    \]
    with the estimate
    \begin{equation*}
        \|\nabla^2u\|_{C^{k,\alpha}_{\beta}}
        +\|(\nabla_{\nabla f}+\lambda)u\|_{C^{k,\alpha}_{\beta}}
        \leq C\left(\|u\|_{L^2_f}
        +\|v\|_{C^{k,\alpha}_{\beta}}\right).
    \end{equation*}

    Moreover, if $2\lambda+\beta<0$, then $u$ has a unique decomposition
    $u=[u]_{-2\lambda}+u^\star$ with a degree-$-2\lambda$ homogeneous
    component and a general component of weight $\beta$, and
    \begin{equation*}
        \|[u]_{-2\lambda}\|_{C^{k,\alpha}_{-2\lambda;-\infty}(M,E)}
        +\|u^\star\|_{C^{k,\alpha}_{\beta}(M,E)}
        \leq C\Big(  \norm{u}_{L^2_f(M,E)}
        +\norm{v}_{C^{k,\alpha}_{\beta}(M,E)}  \Big).
    \end{equation*}
\end{Proposition}
\begin{proof}
    Choose
    \begin{equation*}
        \mu>\max\{0,-\beta/2\}.
    \end{equation*}
    Since the $L^2_f$-spectrum of $\Delta_f$ is nonpositive,
    $\Delta_f-\mu:H^2_f\to L^2_f$ is an isomorphism.  Let $U$ be the unique
    solution of
    \begin{equation*}
        (\Delta_f-\mu)U=v.
    \end{equation*}
    \cref{prop:weighted-Schauder-estimates-BCD} and
    $2\mu+\beta>0$ give
    \begin{equation}\label{eq:intermediate-shifted-solution}
        U\in C^{k+2,\alpha}_{\beta,\nabla f},
        \qquad
        \norm{U}_{C^{k+2,\alpha}_{\beta,\nabla f}}
        \leq C\norm{v}_{C^{k,\alpha}_\beta}.
    \end{equation}
    Put $q:=u-U$.  Then
    \begin{equation}\label{eq:intermediate-q-equation}
        (\Delta_f-\lambda)q=(\lambda-\mu)U.
    \end{equation}

    Because $2\lambda+\beta>-2$, choose $\beta_0$ so that
    \begin{equation}\label{eq:intermediate-beta-zero-choice}
        -2\lambda<\beta_0<\beta+2.
    \end{equation}
    In particular, $\beta<\beta_0$ and
    $U\in C^{k,\alpha}_\beta\subset C^{k,\alpha}_{\beta_0}$.  Applying
    \cref{prop:weighted-Schauder-estimates-BCD} to
    \cref{eq:intermediate-q-equation} at weight $\beta_0$ yields
    \begin{equation}\label{eq:intermediate-weak-bootstrap}
        q\in C^{k+2,\alpha}_{\beta_0,\nabla f},
        \qquad
        q,\nabla q\in C^{k,\alpha}_{\beta_0}.
    \end{equation}
    The second assertion follows from \cref{cor:equivalent-critical-norm} and critical-weight interpolation.

    Put $V_i=\nabla_iq$, $H_{ij}=\nabla_i\nabla_jq$, and
    $E_1=T^*M\otimes E$. Applying
    \cref{eq:exact-weighted-gradient-commutator} first to $q$ and then to
    $V$ gives
    \[
        (\Delta_f-\lambda-1)H+\psi H+\phi V
        =(\lambda-\mu)\nabla^2U,
    \]
    where, in a normal orthonormal frame,
    \begin{equation}\label{eq:intermediate-curvature-coefficients}
    \begin{aligned}
        (\psi H)_{ij}
        &=-2\mathcal R^E_{aj}H_{ia}
          -2\bigl(\mathcal R^{E_1}_{ai}H_{a\cdot}\bigr)_j,\\
        (\phi V)_{ij}
        &=-2(\nabla_i\mathcal R^E_{aj})V_a.
    \end{aligned}
    \end{equation}
    The curvature $\mathcal R^{E_1}$ acts on both the covector slot and
    the $E$ factor. Thus $\psi$ is linear in $\Rm$ and $\phi$ is linear
    in $\nabla\Rm$, with
    $\psi\in C^{k,\alpha}_{-2}$ and $\phi\in C^{k,\alpha}_{-3}$.
    There is no term involving undifferentiated $q$: the weighted divergence
    identity in \cref{lem:commuting-weighted-Laplacian-and-connection} cancels it.
    Estimate \cref{eq:intermediate-shifted-solution} gives
    $\nabla^2U\in C^{k,\alpha}_\beta$, while
    \cref{eq:intermediate-beta-zero-choice,eq:intermediate-weak-bootstrap} give
    \[
        \phi V\in C^{k,\alpha}_{\beta_0-3}\subset C^{k,\alpha}_\beta.
    \]
    Since
    \begin{equation*}
        2(\lambda+1)+\beta=2\lambda+\beta+2>0,
    \end{equation*}
    \cref{prop:weighted-Schauder-estimates-BCD}, now applied to
    the equation for $\nabla^2q$, gives
    \begin{equation}\label{eq:intermediate-Hessian-bound}
        \nabla^2q\in C^{k+2,\alpha}_{\beta,\nabla f},
        \qquad
        \norm{\nabla^2q}_{C^{k,\alpha}_\beta}
        \leq C\left(\norm{u}_{L^2_f}+\norm{v}_{C^{k,\alpha}_\beta}\right).
    \end{equation}
    Together with \cref{eq:intermediate-shifted-solution}, this proves the
    asserted Hessian estimate for $u$.  The original equation then gives
    \begin{equation}\label{eq:intermediate-characteristic-bound}
        (\nabla_{\nabla f}+\lambda)u=\Delta u-v
        \in C^{k,\alpha}_\beta
    \end{equation}
    with the same bound.

    Suppose finally that $\theta:=2\lambda+\beta<0$.  \cref{cor:AC-C0-estimate} gives
    \begin{equation}\label{eq:intermediate-critical-C0-bound}
        \norm{u}_{C^0_{-2\lambda}}
        \leq C\left(\norm{u}_{L^2_f}
        +\norm{v}_{C^{k,\alpha}_\beta}\right),
    \end{equation}
    where the fixed-core boundary term is bounded by interior elliptic
    estimates and the Gaussian Sobolev graph estimate.

    On the canonical end set $y:=r^{2\lambda}u$.  \cref{lem:difference}\cref{lem:difference-radial-derivative} gives
\begin{equation}\label{eq:intermediate-radial-source}
        2\nabla^C_{\mathbf{X}}y
        =2r^{2\lambda}\bigl(
        (\nabla_{\nabla f}+\lambda)u-\Ric(g)* u
        \bigr).
    \end{equation}
    The first term on the right belongs to $C^{k,\alpha}_\theta$ by
    \cref{eq:intermediate-characteristic-bound}. For the curvature term,
    use $\beta+2>-2\lambda$ and
    \cref{eq:intermediate-critical-C0-bound,eq:intermediate-Hessian-bound}.
    Interpolation on rescaled annuli gives
    \[
        \|u\|_{C^{k+2,\alpha}_{\beta+2}}
        \leq C\left(\|u\|_{C^0_{\beta+2}}
        +\|\nabla^2u\|_{C^{k,\alpha}_\beta}\right)
        \leq C\left(\|u\|_{L^2_f}+\|v\|_{C^{k,\alpha}_\beta}\right).
    \]
    Multiplication by $\Ric(g)\in C^\infty_{-2}$ and by $r^{2\lambda}$
    therefore yields
    \[
        r^{2\lambda}\Ric(g)*u\in C^{k,\alpha}_\theta
    \]
    with the same bound.
    Thus the left-hand side of
    \cref{eq:intermediate-radial-source} lies in
    $C^{k,\alpha}_\theta$, with the required norm bound.  Since $\theta<0$,
    \cref{lem:characterization-of-classic-weighted-Holder-space}
    supplies a radially parallel trace $a\in C^{k,\alpha}(\Sigma)$ and a
    remainder $y^\star\in C^{k,\alpha}_\theta$.  Its trace norm is controlled
    by the value of $y$ on one fixed annulus, which is bounded by the same
    interior estimate used in \cref{eq:intermediate-critical-C0-bound}; its
    remainder norm is controlled by \cref{ineq:radial-trace-remainder}.
    Multiplying by $r^{-2\lambda}$ and inserting the fixed cutoff gives
    \begin{equation*}
        u=[u]_{-2\lambda}+u^\star,
        \qquad
        u^\star\in C^{k,\alpha}_\beta,
    \end{equation*}
    with precisely the estimate in the statement.  Uniqueness follows from
    the uniqueness of the radial trace.
\end{proof}

\begin{Remark}\label{rmk:loss-of-boundary-regularity}\label{ex:loss-of-regularity}
    The components in \cref{prop:weighted-Schauder-estimates-intermediate}
    need not have the stronger $C^{k+2,\alpha}$ regularity.
    To see the obstruction, fix $\lambda>0$ and $\beta=-2\lambda-1$.
    In a link coordinate ball choose a nonzero smooth cutoff $\eta$, and set
    \[
        a_N=N^{-k-\alpha-1}\eta(\theta)\sin(N\theta^1),
        \qquad u^{(N)}=\chi\zeta(r/N)r^{-2\lambda}\boldsymbol\tau(a_N),
    \]
    where $a_N$ is extended by zero, $\zeta=0$ on $(0,1]$, and
    $\zeta=1$ on $[2,\infty)$. Then
    \begin{equation}\label{eq:counterexample-link-norms}
        \|a_N\|_{C^{k,\alpha}}\asymp N^{-1},
        \qquad \|a_N\|_{C^{k+2,\alpha}}\asymp N.
    \end{equation}
    The characteristic drift cancels on the homogeneous part, and cone
    scaling on $r\geq N$ gives
    \[
        \|(\Delta_f-\lambda)u^{(N)}\|_{C^{k,\alpha}_\beta}
        \leq C\left(N\|a_N\|_{C^{k,\alpha}}
        +N^{-1}\|a_N\|_{C^{k+2,\alpha}}\right)\leq C.
    \]
    The first term bounds the drift-cutoff contribution and the second
    bounds the Laplacian terms; the AC errors decay faster.
    The homogeneous coefficient of $u^{(N)}$ is $a_N$, so the stronger
    decomposed norm tends to infinity. Since $\lambda>0$ gives a unique
    Gaussian $L^2_f$ solution for every such source, the open mapping
    theorem shows that some sources have solutions outside that stronger
    decomposed domain.
\end{Remark}

\section{Gauge spaces and estimates}
\label{sec:fixed-Banach-estimates}\label{sec:fixed-Banach}

Throughout this section, $(M,g,f)$ is a fixed complete
nontrivial AC Ricci shrinker. We fix $k\geq6$ and $0<\alpha<1$.
We use the cutoff $\chi$ and radial parallel identifications of
\cref{def:standard-AC-setting}, and the conically scaled norms of
\cref{def:weighted-Holder-spaces}. We write
\[
 \mathbf X=\nabla f,\qquad L=\Delta_{f}+\tfrac12,
 \qquad \mathcal P=\delta_{f}\delta^*,
 \qquad \delta^*=-\operatorname{Sym}\nabla.
\]
On the canonical end $\mathbf X=r\partial_r/2$. In particular, the radial
parallel extension satisfies
\[
 \boldsymbol\tau(a\,\mathrm dr+b)=a\,\mathrm dr+r\pi^*b,
 \qquad a\in C^{j,\alpha}(\Sigma),\quad
 b\in C^{j,\alpha}(\Sigma,T^*\Sigma),
\]
where $j\geq0$ and $\pi$ is the projection to $\Sigma$. Homogeneous degrees below
are taken with respect to this identification. Constants may depend on
the fixed background, $k$, and $\alpha$.

\subsection{An auxiliary sum space}
\label{sec:definition-of-spaces-in-slice-theorem}

All coefficient tensors below are sections of the fixed link bundles.

\begin{Definition}\label{def:slice-radial-filtration}
For each integer $m\geq0$, define $E_m$ by the unique expansion
\begin{equation}\label{eq:slice-E-expansion}
 h=\chi\bigl\{\boldsymbol\tau(c_0)
   +r^{-1}\boldsymbol\tau(c_{-1}+2\,\mathrm dr\odot b_{-1})\bigr\}
   +h^\star.
\end{equation}
The coefficients have the types
\[
 c_0,c_{-1}\in C^{m,\alpha}(\Sigma,S^2T^*\Sigma),\qquad
 b_{-1}\in C^{m,\alpha}(\Sigma,T^*\Sigma).
\]
Set
\[
 \|h\|_{E_m}=\|(c_0,c_{-1},b_{-1})\|_{C^{m,\alpha}(\Sigma)}
                  +\|h^\star\|_{C^{m,\alpha}_{-2}}.
\]
The space consists of the expansions for which this norm is finite.
\end{Definition}

The space $E_m$ is Banach because the link and weighted H\"older factors
are complete. Coefficient
uniqueness follows by dividing by successive powers of $r$ and using
the strictly lower remainder weight. The auxiliary spaces
$\mathscr X_k,\mathscr Y_k$ below use the Gaussian inverse to assign
different finite regularities to different coefficients.
For example, $\chi r^{-2}\sin(\log r)\,\mathrm dr^2\in E_m$ for
every $m$.

We use the weighted and compact-link interpolation identities of
\cref{lem:weighted-real-interpolation}.

For a tensor bundle $E$, define the Banach sum
\begin{equation}\label{ic:scalar-sum}
 \mathscr S_k(E)=C^{k+1,\alpha}_0(E)+C^{k-1,\alpha}_{-2}(E),
 \qquad
 \|q\|_{\mathscr S_k}
 =\inf_{q=q_0+q_{-2}}
   \bigl(\|q_0\|_{C^{k+1,\alpha}_0}
               +\|q_{-2}\|_{C^{k-1,\alpha}_{-2}}\bigr).
\end{equation}
The scalar case is denoted simply by $\mathscr S_k$.

\begin{Lemma}\label{ic:sum-multiplier}
The sum in \cref{ic:scalar-sum} is Banach and embeds continuously in
$C^{k-1,\alpha}_0$ and $L^2_f$. There is a continuous inclusion
\[
 C^{k,\alpha}_{-1}(E)\subset\mathscr S_k(E).
\]
Moreover,
\begin{equation}\label{ic:sum-estimates}
 \|q\|_{\mathscr S_k}\leq C\|q\|_{C^{k,\alpha}_{-1}},
 \qquad
 \|bq\|_{\mathscr S_k}
 \leq C\|b\|_{C^{k+1,\alpha}_0}\|q\|_{\mathscr S_k}.
\end{equation}
The second estimate also applies to fixed tensor contractions and
bundle homomorphisms of the indicated regularity.
\end{Lemma}
\begin{proof}
The addition map from the product of the two summands into
distributions has closed kernel. Its quotient, equipped with the
displayed quotient norm, is therefore Banach. Both summands embed in
the asserted local and Gaussian spaces, proving the inclusions.

Choose a uniformly locally finite atlas and partition of unity on the
normalized end annuli. On an annulus of radius $R$, a section of
$C^{k,\alpha}_{-1}$ has normalized norm at most $CR^{-1}N$, where
$N=\|q\|_{C^{k,\alpha}_{-1}}$. Regularize its localized expression
at scale $R^{-1}$. The regularized part has $C^{k+1,\alpha}$ norm at
most $CN$, and the error has $C^{k-1,\alpha}$ norm at most
$CR^{-2}N$. Patching the parts gives a global decomposition in the two
summands with norm at most $CN$. The fixed compact core is handled by
the same construction at a fixed scale. This proves the first
estimate. Multiplication by $b$ is bounded separately on each
summand; adding the two bounds and taking the infimum proves the
second.
\end{proof}

\subsection{Estimates for the shifted Laplacian}
\label{subsec:scalar-terminal-estimates}

The scalar operator $L=\Delta_f+\tfrac12$ has a bounded Gaussian inverse
\[
 G_{1/2}:L^2_f(M)\longrightarrow H^2_f(M).
\]
Indeed, if $p\in H^2_f(M)$ and $Lp=0$, the commutation identity following
\eqref{eq:gauge-operator} gives
$\Delta_f(\mathrm dp)=\mathrm d(Lp)=0$.
By \cref{prop:elliptic-theory-for-weighted-Laplacian}\textup{(i)},
$\mathrm dp$ is parallel. A nonzero parallel gradient splits off a line,
which is excluded for a nontrivial AC shrinker by \cref{sec:AC-shrinkers}.
Thus $\mathrm dp=0$, and $Lp=0$ implies $p=0$.
The inverse and its bound now follow from
\cref{prop:Gaussian-Fredholm-realization}.

\begin{Lemma}\label{ic:scalar-terminal}
For $q\in\mathscr S_k$, let $p=G_{1/2}q$. There is a unique
$a\in C^{k+2,\alpha}(\Sigma)$ such that
\begin{equation}\label{ic:scalar-terminal-expansion}
 \begin{gathered}
 p=\chi ra+p^\star,\qquad
 p^\star\in C^{k+1,\alpha}_0+C^{k,\alpha}_{-1},\\
 \mathrm dp-\mathrm d(\chi ra)\in C^{k,\alpha}_{-1},
 \qquad \nabla^2p\in\mathscr S_k(T^{0,2}M),
 \qquad (\nabla_{\mathbf X}-\tfrac12)p\in\mathscr S_k.
 \end{gathered}
\end{equation}
The sum of these norms, the trace norm, and $\|p\|_{H^2_{f}}$
is at most $C\|q\|_{\mathscr S_k}$; the remainder sum has its
quotient norm.
\end{Lemma}
\begin{proof}
First suppose $q\in C^{k+1,\alpha}_0$ and put
$N=\|q\|_{C^{k+1,\alpha}_0}$. Apply
\cref{prop:weighted-Schauder-estimates-intermediate} with
$\lambda=-1/2$, weight zero, and source order $k+1$.
It gives an initial trace $a\in C^{k+1,\alpha}$,
$p-\chi ra\in C^{k+1,\alpha}_0$, and
$\nabla^2p\in C^{k+1,\alpha}_0$, with norms at most $CN$.
For $R$ sufficiently large, set
\[
 a_R(\theta)=R^{-1}p(R,\theta).
\]
The remainder estimate gives
$\|a-a_R\|_{C^{k+1,\alpha}(\Sigma)}\leq CR^{-1}N$.
On the rescaled annulus $(1/2,4)\times\Sigma$, write
$U_R(s,\theta)=R^{-1}p(Rs,\theta)$, where $s=r/R$.
Its $C^{k+1,\alpha}$ norm is at most $CN$, and its second
covariant derivative for the rescaled connection has
$C^{k+1,\alpha}$ norm at most $CRN$. Ordinary local interpolation
on nested rescaled annuli absorbs the first derivatives and gives
\[
 \|a_R\|_{C^{k+3,\alpha}(\Sigma)}\leq CRN.
\]
Thus
$K(R^{-2},a;C^{k+1,\alpha}(\Sigma),C^{k+3,\alpha}(\Sigma))
\leq CNR^{-1}$.
These scales and the initial trace bound control the whole
$K$-functional. The compact-link case of
\cref{lem:weighted-real-interpolation}\textup{(iii)} at parameter
$1/2$ gives $\|a\|_{C^{k+2,\alpha}}\leq CN$.
Differentiating the already controlled remainder gives
$\mathrm dp-\mathrm d(\chi ra)\in C^{k,\alpha}_{-1}$.

Next suppose $q\in C^{k-1,\alpha}_{-2}$ and put
$N=\|q\|_{C^{k-1,\alpha}_{-2}}$.
The two cases $m=0,1$ of \cref{prop:finite-expansion}, at source
order $k-1$, give the following Gaussian inverse endpoints:
\begin{equation}\label{ic:scalar-inverse-endpoints}
\begin{array}{c|c|c|c}
 \text{source}&\text{degree-one trace}&\text{degree-minus-one trace}
                    &\text{remainder}\\ \hline
 C^{k-1,\alpha}_{-1}&C^{k+1,\alpha}&\text{absent}
                      &C^{k+1,\alpha}_{-1,\nabla f}\\
 C^{k-1,\alpha}_{-3}&C^{k+3,\alpha}&C^{k+1,\alpha}
                      &C^{k+1,\alpha}_{-3,\nabla f}.
\end{array}
\end{equation}
Every norm is bounded by the source norm, by the boundedness of $G_{1/2}$. The critical norm and ordinary interpolation
imply, for either remainder weight $\eta$,
\begin{equation}\label{ic:critical-order-conversion}
 \|R_0\|_{C^{k,\alpha}_{\eta+1}}
 +\|\mathrm dR_0\|_{C^{k,\alpha}_{\eta+1}}
 \leq C\|R_0\|_{C^{k+1,\alpha}_{\eta,\nabla f}}.
\end{equation}
Indeed the two highest derivatives in the critical norm have weight
$\eta+1-k$, and the lower derivatives have stronger weights; apply
\cref{lem:critical-weight-interpolation} when a missing intermediate
derivative is needed.

Let $\mathcal Tq$ be the degree-one trace. The first endpoint defines
it on the present source class, and uniqueness makes it linear.
Choose a fixed smooth radial cutoff $\eta$ equal to one on $[0,1]$
and zero on $[2,\infty)$, and put
\[
 q_{\mathrm{in},R}=\eta(r/R)q,\qquad
 q_{\mathrm{out},R}=q-q_{\mathrm{in},R},
\]
where the inner cutoff is extended by one on the compact core.
The two endpoint norms satisfy
\[
 \|q_{\mathrm{in},R}\|_{C^{k-1,\alpha}_{-3}}\leq CRN,
 \qquad
 \|q_{\mathrm{out},R}\|_{C^{k-1,\alpha}_{-1}}\leq CR^{-1}N.
\]
The trace endpoints therefore give
$K(R^{-2},\mathcal Tq;C^{k+1,\alpha}(\Sigma),C^{k+3,\alpha}(\Sigma))
\leq CNR^{-1}$, and hence
$\|\mathcal Tq\|_{C^{k+2,\alpha}}\leq CN$ by the same interpolation
identity. The bound also holds uniformly for
$\mathcal Tq_{\mathrm{in},R}$, whose original weight-minus-two
source norm is at most $CN$.

The linear map
$\mathcal Sq=\mathrm dG_{1/2}q-\mathrm d(\chi r\mathcal Tq)$
has endpoint outputs $C^{k,\alpha}_0$ and $C^{k,\alpha}_{-2}$,
by \cref{ic:scalar-inverse-endpoints,ic:critical-order-conversion}.
At the second endpoint the additional derivative of the degree-minus-one
term has precisely weight minus two and sufficient regularity.
On the annulus of radius $R$, split the source as above to obtain
\[
 \|\mathcal Sq\|_{C^{k,\alpha}(\text{normalized annulus})}
 \leq C\bigl(R^{-2}RN+R^{-1}N\bigr)=CR^{-1}N.
\]
This proves the desired gradient remainder estimate globally.

For the undifferentiated remainder, the outer solution minus its
degree-one part is bounded by $CR^{-1}N$ on the same annulus.
The inner expansion is
\[
 G_{1/2}q_{\mathrm{in},R}
  =\chi r a_{\mathrm{in},R}
    +\chi r^{-1}b_{\mathrm{in},R}+w_{\mathrm{in},R},
 \qquad
 \|w_{\mathrm{in},R}\|_{C^{k+1,\alpha}_{-3,\nabla f}}\leq CRN,
\]
where $a_{\mathrm{in},R}=\mathcal Tq_{\mathrm{in},R}$.
Thus $w_{\mathrm{in},R}$ contributes at most $CR^{-1}N$ by
\cref{ic:critical-order-conversion}. Since the inner source vanishes
near infinity, the first coefficient recursion of
\cref{prop:finite-expansion} gives
\[
 b_{\mathrm{in},R}=-\mathcal B_{1,-1/2,1}a_{\mathrm{in},R},
\]
where the fixed link operator has order at most two. The uniform
$C^{k+2,\alpha}$ bound just proved for $a_{\mathrm{in},R}$ gives
$\|b_{\mathrm{in},R}\|_{C^{k,\alpha}}\leq CN$.
Its degree-minus-one contribution is therefore at most $CR^{-1}N$,
as required. The fixed core follows from the first endpoint.

For each summand, write
\[
 \nabla^2p=\nabla^2(\chi ra)
       +\nabla\bigl(\mathrm dp-\mathrm d(\chi ra)\bigr).
\]
The first term belongs to $C^{k,\alpha}_{-1}$, and the second to
$C^{k-1,\alpha}_{-2}$. \Cref{ic:sum-multiplier} gives the Hessian
sum estimate. The equation gives
$(\nabla_{\mathbf X}-1/2)p=\Delta p-q$, proving the drift estimate.
Finally, decompose a general $q$ into the two summands, use linearity
of the Gaussian inverse and uniqueness of the degree-one trace, and
take the infimum over decompositions. Dividing the expansion by $r$
proves uniqueness of its coefficient.
\end{proof}

\begin{Lemma}\label{ic:degree-zero-terminal}
Let $E$ be a fixed tensor bundle, and suppose $u\in L^2_{f}(E)$
solves
\[
 Lu=\chi\boldsymbol\tau(F_0)+F^\star,
 \qquad F_0\in C^{k+1,\alpha}(\Sigma,E_\Sigma),\quad
 F^\star\in C^{k,\alpha}_{-1}(E).
\]
There is a unique $a_1\in C^{k+2,\alpha}(\Sigma,E_\Sigma)$ such that
\begin{equation}\label{ic:degree-zero-expansion}
 u=\chi r\boldsymbol\tau(a_1)+2\chi\boldsymbol\tau(F_0)+u^\star,
 \qquad u^\star\in C^{k,\alpha}_{-1}.
\end{equation}
Moreover,
\begin{equation}\label{ic:degree-zero-whole-fields}
 \nabla^2u,\quad
 \nabla u-\nabla^C\bigl(\chi r\boldsymbol\tau(a_1)\bigr),\quad
 (\nabla_{\mathbf X}-\tfrac12)u+\chi\boldsymbol\tau(F_0)
 \in C^{k,\alpha}_{-1}.
\end{equation}
The sum of the trace, remainder, and whole-field norms is bounded by
$C(\|F_0\|_{C^{k+1,\alpha}}+
\|F^\star\|_{C^{k,\alpha}_{-1}}+\|u\|_{L^2_{f}})$.
\end{Lemma}
\begin{proof}
Put $F=\chi\boldsymbol\tau(F_0)$ and define
\[
 U_0=-4(\Delta_{f}-2)^{-1}F,
 \qquad W=U_0-2F.
\]
The positive shift has trivial Gaussian kernel by its energy identity,
so \cref{prop:Gaussian-Fredholm-realization} gives this inverse on
every tensor bundle. If $F_0\in C^{k,\alpha}$,
\cref{prop:weighted-Schauder-estimates-BCD} at weight zero gives
$U_0\in C^{k+2,\alpha}_{0,\nabla f}$ with controlled norm.
Consequently $W$ and $\nabla^2U_0$ belong to $C^{k,\alpha}_0$.
If $F_0\in C^{k+2,\alpha}$, the homogeneous extension satisfies
$\Delta_{f}F\in C^{k,\alpha}_{-2}$ and
\[
 (\Delta_{f}-2)W=-2\Delta_{f}F.
\]
The same proposition at weight minus two gives
$W\in C^{k+2,\alpha}_{-2,\nabla f}$; adding
$2\nabla^2F$ also controls $\nabla^2U_0$ in
$C^{k,\alpha}_{-2}$. Interpolation of these two fixed linear maps,
using \cref{lem:weighted-real-interpolation}\textup{(ii), (iii)}
at parameter $1/2$, proves
\[
 \|W\|_{C^{k,\alpha}_{-1}}
       +\|\nabla^2U_0\|_{C^{k,\alpha}_{-1}}
 \leq C\|F_0\|_{C^{k+1,\alpha}}.
\]
The gradient uses its one-derivative endpoint. Set $V=\nabla U_0$.
The exact identity \cref{eq:exact-weighted-gradient-commutator} gives
\[
 (\Delta_{f}-\tfrac52)V_i
        -2\mathcal R^E_{ai}V_a=-4\nabla_iF.
\]
Here $\nabla F\in C^{k,\alpha}_{-1}$, and the curvature
coefficient has weight minus two. The Gaussian inverse estimate
controls $\|V\|_{L^2_{f}}$. Applying
\cref{prop:weighted-Schauder-estimates-BCD}, with shift $5/2$
and weight minus one, gives
$\|\nabla U_0\|_{C^{k,\alpha}_{-1}}
\leq C\|F_0\|_{C^{k+1,\alpha}}$.
Finally, the defining equation gives the exact identities
\[
 LU_0-F=\tfrac52W,
 \qquad (\nabla_{\mathbf X}-\tfrac12)U_0+F
       =\Delta U_0-\tfrac52W.
\]
Thus all the lift errors have weight minus one at order $k$.

The residual $u-U_0$ has source in $C^{k,\alpha}_{-1}$.
Apply the $m=0$ case of \cref{prop:finite-expansion} at source
order $k$. It gives
\[
 u-U_0=\chi r\boldsymbol\tau(a_1)+R_1,
 \quad a_1\in C^{k+2,\alpha},\quad
 R_1\in C^{k+2,\alpha}_{-1,\nabla f}.
\]
The equivalent critical norm controls $R_1$, its Hessian and its
drift at weight minus one and order $k$. Applying
\cref{lem:critical-weight-interpolation} to its derivatives through
order $k$ also controls its gradient at that weight and order.
The Hessian of the leading degree-one field is in
$C^{k,\alpha}_{-1}$ because $a_1\in C^{k+2,\alpha}$.
The difference between the background and cone connections acting
on that field has weight minus two. Adding the lift estimates proves
\cref{ic:degree-zero-expansion,ic:degree-zero-whole-fields}.
The compact transition region is covered by the same finite-order
estimates, and uniqueness follows by dividing the expansion by $r$.
\end{proof}

\subsection{The fixed gauge graph}

\begin{Definition}\label{ic:graph-definition}
With $\mathcal K$ as in \cref{eq:Gaussian-Killing-space}, define
\begin{equation}\label{ic:target-definition}
 \begin{split}
 \mathscr Y_k=\{v={}&\chi\boldsymbol\tau(c_0)+v^\star:
        c_0\in C^{k+1,\alpha}(\Sigma,T^*\Sigma),\quad
        v^\star\in C^{k,\alpha}_{-1}(T^*M),\quad
        \delta_{f}v\in\mathscr S_k\},\\
 \|v\|_{\mathscr Y_k}={}&\|c_0\|_{C^{k+1,\alpha}}
       +\|v^\star\|_{C^{k,\alpha}_{-1}}
       +\|\delta_{f}v\|_{\mathscr S_k},
 \end{split}
\end{equation}
and
\begin{equation}\label{ic:domain-definition}
 \mathscr X_k=\{Z\in H^2_{f}(T^*M)\cap \mathcal K^\perp:
                                  \mathcal P Z\in\mathscr Y_k\},
 \qquad
 \|Z\|_{\mathscr X_k}=\|Z\|_{H^2_{f}}
                                  +\|\mathcal P Z\|_{\mathscr Y_k}.
\end{equation}
The leading coefficient in \cref{ic:target-definition} is tangential
and unique.
\end{Definition}

\begin{Proposition}\label{ic:graph-isomorphism}
The spaces $\mathscr X_k$ and $\mathscr Y_k$ are Banach, and
\[
 \mathcal P:\mathscr X_k\longrightarrow\mathscr Y_k\cap \mathcal K^\perp
\]
is a bounded isomorphism with inverse the Gaussian gauge inverse.
\end{Proposition}
\begin{proof}
The coefficient and remainder product in \cref{ic:target-definition}
is Banach and embeds in distributions and $L^2_{f}$.
Distributional continuity of $\delta_{f}$ identifies the
limit of a graph-Cauchy sequence, proving completeness of
$\mathscr Y_k$. The confinement estimate and
$\delta^*:H^2_{f}\to H^1_{f}$ give
\[
 \|\mathcal P Z\|_{L^2_{f}}
       \leq C\|\delta^*Z\|_{H^1_{f}}
       \leq C\|Z\|_{H^2_{f}}.
\]
Consequently a Cauchy sequence in \cref{ic:domain-definition} has an
$H^2_{f}\cap \mathcal K^\perp$ limit, and distributional continuity
identifies its image with the limit in $\mathscr Y_k$.
Finally \cref{prop:elliptic-theory-for-gauge-operator} gives the
self-adjoint Gaussian realization and its bounded inverse on
$\mathcal K^\perp$. For $v\in\mathscr Y_k\cap \mathcal K^\perp$, that inverse
satisfies
$\|Z\|_{H^2_{f}}\leq C\|v\|_{L^2_{f}}
\leq C\|v\|_{\mathscr Y_k}$ and belongs to $\mathscr X_k$ by
definition. This proves the stated isomorphism.
\end{proof}

\begin{Proposition}\label{ic:graph-estimates}
Every $Z\in\mathscr X_k$ has unique coefficients and remainder
\begin{equation}\label{ic:gauge-expansion}
 \begin{gathered}
 Z=\chi\{r\boldsymbol\tau(b_1)
                 +\boldsymbol\tau(a_0\,\mathrm dr+b_0)\}+Z^\star,\\
 a_0\in C^{k+2,\alpha}(\Sigma),\qquad
 b_1\in C^{k+2,\alpha}(\Sigma,T^*\Sigma),\\
 b_0\in C^{k+1,\alpha}(\Sigma,T^*\Sigma),\qquad
 Z^\star\in C^{k,\alpha}_{-1}(T^*M).
 \end{gathered}
\end{equation}
Set $p=\delta_fZ$. Then
\begin{equation}\label{ic:gauge-whole-fields}
 \begin{gathered}
 \nabla^2Z,\quad \nabla Z^\star,\quad
 \nabla_{\mathbf X}Z^\star,\quad
 \mathrm d\bigl(p+\tfrac12\chi ra_0\bigr)
       \in C^{k,\alpha}_{-1},\\
 Lp\in\mathscr S_k.
 \end{gathered}
\end{equation}
All coefficient, remainder and whole-field norms are bounded by
$C\|Z\|_{\mathscr X_k}$, and their dependence on $Z$ is linear.
In addition,
\begin{equation}\label{ic:whole-coordinate-regularity}
 \|Z\|_{C^{k+2,\alpha}_1}\leq C\|Z\|_{\mathscr X_k}.
\end{equation}
If $c_0$ is the leading coefficient of $v=\mathcal P Z$ and $p_1$ the
degree-one coefficient of $p$, then
\begin{equation}\label{ic:degree-zero-gauge-action}
 a_0=-2p_1,\qquad b_0=\mathrm d_\Sigma a_0-4c_0.
\end{equation}
\end{Proposition}
\begin{proof}
The confinement estimate gives $p\in L^2_{f}$ with norm bounded
by $C\|Z\|_{H^2_{f}}$. The commutation identities for the
weighted divergence and gauge operator give, distributionally,
\begin{equation}\label{ic:coupled-gauge-system}
 Lp=-\delta_{f}v,\qquad LZ=-\mathrm dp-2v.
\end{equation}
These signs use $\delta^*=-\operatorname{Sym}\nabla$.
The Gaussian realization identifies $p$ with
$-G_{1/2}\delta_{f}v$.
\Cref{ic:scalar-terminal} gives
\[
 p=\chi rp_1+p^\star,\qquad p_1\in C^{k+2,\alpha},\qquad
 p^\star\in C^{k+1,\alpha}_0+C^{k,\alpha}_{-1},\qquad
 \mathrm dp-\mathrm d(\chi rp_1)\in C^{k,\alpha}_{-1}.
\]
The second equation in \cref{ic:coupled-gauge-system} has degree-zero
source
\[
 F_0=-p_1\,\mathrm dr-\mathrm d_\Sigma p_1-2c_0
\]
in the link identification, of regularity $C^{k+1,\alpha}$, and
source remainder in $C^{k,\alpha}_{-1}$. Fixed cutoff errors are
included in the remainder. \Cref{ic:degree-zero-terminal} therefore
gives a degree-one coefficient $c_1\in C^{k+2,\alpha}$, the
degree-zero coefficient $2F_0$, a remainder of weight minus one,
and the three whole-field estimates of that lemma.

The degree-one radial component of $c_1$ vanishes. Indeed the gradient
estimate just obtained shows that $\delta Z$ is bounded,
whereas a nonzero radial degree-one term would give a degree-two term
in $Z(\mathbf X)$. This is incompatible with
$p=\delta Z-Z(\mathbf X)$ and the scalar expansion above. Hence
$c_1=b_1$ is tangential, $a_0=-2p_1$, and
$b_0=-2\mathrm d_\Sigma p_1-4c_0$, proving both the expansion and
\cref{ic:degree-zero-gauge-action}. The scalar gradient estimate gives
$\mathrm d(p+\tfrac12\chi ra_0)\in C^{k,\alpha}_{-1}$, while the
first coupled equation gives $Lp\in\mathscr S_k$.

To obtain the estimates directly for the remainder, put
$Z_1=\chi r\boldsymbol\tau(b_1)$ and
$Z_0=\chi\boldsymbol\tau(a_0\,\mathrm dr+b_0)$.
The coefficient regularities give $\nabla Z_0\in C^{k,\alpha}_{-1}$,
and the background--cone connection difference gives
$(\nabla-\nabla^C)Z_1\in C^{k,\alpha}_{-2}$.
Subtracting these fields from the gradient cancellation estimate in
\cref{ic:degree-zero-terminal} proves
$\nabla Z^\star\in C^{k,\alpha}_{-1}$.
Outside the cutoff region,
$\nabla^C_{\mathbf X}Z_1=\tfrac12Z_1$ and
$\nabla^C_{\mathbf X}Z_0=0$. Thus
\[
 \nabla_{\mathbf X}Z^\star
  =\bigl((\nabla_{\mathbf X}-\tfrac12)Z+\tfrac12Z_0\bigr)
    +\tfrac12Z^\star
    -(\nabla_{\mathbf X}-\tfrac12)Z_1-\nabla_{\mathbf X}Z_0
    \in C^{k,\alpha}_{-1}.
\]
Indeed the first term is the drift estimate of
\cref{ic:degree-zero-terminal}, and the final two terms are controlled
by the connection difference, with fixed compact cutoff terms included.
The whole Hessian estimate for $Z$ is already part of that lemma.
Uniqueness of the successive traces makes all the coefficient and
remainder maps linear, and every estimate used is bounded in the
graph norm.

Finally, on an annulus of radius $R$, normalize $Z$ by its weight one with respect to $g$. Its $C^{k,\alpha}$ norm is uniformly bounded by the
expansion, and its normalized second covariant derivative is
uniformly bounded by $R\|\nabla^2Z\|_{C^{k,\alpha}}$. Ordinary
interpolation for uniformly controlled connections on nested annuli
gives a uniform $C^{k+2,\alpha}$ bound. Scaling back gives \cref{ic:whole-coordinate-regularity}. On the compact core, elliptic regularity for
$\mathcal P Z=v$ gives $C^{k+2,\alpha}$ control. This completes the proof.
\end{proof}

\section{Weighted slice theorem and rigidity}\label{sec:slice-theorem}

For the local addition, slice construction, and finite analytic estimates,
$(M^n,g,f)$ is a fixed complete
nontrivial AC Ricci shrinker in the canonical setting. We retain the
mass normalization of \cref{eq:normalization-of-potential} and the
spaces, estimates, and conventions of \cref{sec:fixed-Banach} for this
background. Write $P(h,Z)$ for the nonlinear gauge residual associated
with the linear operator $\mathcal P$. All finite weighted norms are
computed using $g$ and $f$.
In a link chart $\theta=(\theta^1,\ldots,\theta^{n-1})$, a subscript
$\theta^i$ labels the corresponding tangential coordinate component.

All estimates in this section use conically scaled norms. For a
section, $O_{j+\alpha}(r^\beta)$\phantomsection\label{idx:conic-O} denotes a bound in $C^{j,\alpha}_\beta$.
In component formulas it denotes the corresponding bound for the
fixed-link coefficient family. These bounds include mixed radial
and angular derivatives on normalized annuli. We write $O_\infty(r^\beta)$
when these bounds hold for every integer $j\geq0$.

\subsection{The product-adapted local addition}
\label{sec:weighted-divergence-free-gauge}

\begin{Lemma}
\label{lem:slice-product-affine-connection}
There is a smooth affine connection $\nabla^{\mathrm a}$ on $M$
agreeing on a far canonical end with
$\mathrm d_{\mathbb R}\oplus\nabla^{g_\Sigma}$.
Its exponential is a local addition, and on that end
\begin{equation}\label{eq:slice-product-exponential}
 \operatorname{Exp}^{\mathrm a}_{(r,\theta)}
 (v^r\partial_r+V^\Sigma)
 =(r+v^r,\exp^{g_\Sigma}_\theta V^\Sigma).
\end{equation}
\end{Lemma}
\begin{proof}
Interpolate between a smooth connection on the compact core and the
product connection across a fixed compact annulus. An affine
exponential is a local addition near the zero section; on the end
the product geodesic equation gives the displayed formula.
\end{proof}

For a small $Z\in\mathscr X_k$, define
\begin{equation}\label{eq:slice-adapted-diffeomorphism}
 F_Z(x)=\operatorname{Exp}^{\mathrm a}_x
                 (Z^{\sharp_{g}}(x)).
\end{equation}
Although the tangential part of $Z^{\sharp_g}$ can have $g$-norm of order $r$,
the associated link vector has a uniformly small norm. Thus this
formula is defined on the whole end, while a fixed neighborhood of
the zero section suffices on the core.

\subsection{Radial consequences of the gauge graph}

Throughout this subsection $k\geq6$ and $0<\alpha<1$ are fixed.
On the end put $t=\log r$.

\begin{Lemma}\label{ic:radial-graph}
Let $Z\in\mathscr X_k$, and let $a_0,b_0,b_1$ be the traces in
\cref{ic:graph-estimates}. Write
$\boldsymbol\tau_r^{-1}Z=u\,\mathrm dr+w$, where, for each fixed $r$,
\[
 u(r,\cdot)\in C^{k+2,\alpha}(\Sigma),\qquad
 w(r,\cdot)\in C^{k+2,\alpha}(\Sigma,T^*\Sigma).
\]
Then
\begin{align}
 r\,\mathrm du-\boldsymbol\tau(\mathrm d_\Sigma a_0)
       &\in C^{k,\alpha}_{-1},\label{ic:radial-gradient}\\
 r\nabla^2u,\quad r\partial_tu
       &\in\mathscr S_k,\label{ic:radial-second}\\
 \partial_tw-w+b_0&\in C^{k,\alpha}_{-1},\qquad
 \partial_t\partial_{\theta^i}(w/r)\in\mathscr S_k.
       \label{ic:angular-second}
\end{align}
The last expression is written in a fixed link chart. The corresponding
norms, including a fixed compact extension, are bounded by
$C\|Z\|_{\mathscr X_k}$. These assertions also hold for the components of $Z^{\sharp_g}$
in the rescaled frame $\partial_r,r^{-1}\partial_{\theta^i}$,
with the link indices raised by $g_\Sigma$ in their leading terms.
\end{Lemma}
\begin{proof}
Set $p=\delta_{f}Z$. Since
$\mathbf X=\frac r2\partial_r$, the exact identity
$ru=2(\delta Z-p)$ gives, where $\chi=1$,
\begin{equation}\label{ic:radial-first-identity}
 r\,\mathrm du-\boldsymbol\tau(\mathrm d_\Sigma a_0)
 =(a_0-u)\,\mathrm dr+2\,\mathrm d\delta Z-2\,\mathrm d(p+\tfrac12\chi ra_0).
\end{equation}
On the fixed transition annulus the right-hand side has the additional
term $\mathrm d((\chi-1)ra_0)$, which has a controlled compactly
supported extension. Here $a_0-u\in C^{k,\alpha}_{-1}$, and
$\mathrm d\delta Z$ is a contraction of the whole tensor
$\nabla^2Z$. Thus \cref{ic:graph-estimates} proves
\cref{ic:radial-gradient}. In particular,
$\partial_ru\in C^{k,\alpha}_{-2}$ and
$\partial_{\theta^i}u-\partial_{\theta^i}a_0=O_{k+\alpha}(r^{-1})$.
Twice differentiating the undifferentiated identity gives
\begin{equation}\label{ic:radial-Hessian-identity}
 r\nabla^2u=2\nabla^2\delta Z
 -2\nabla^2p
 -2\operatorname{Sym}(\mathrm dr\otimes\mathrm du)
 -u\nabla^2r.
\end{equation}
The first term is a contraction of $\nabla(\nabla^2Z)$ and belongs to
$C^{k-1,\alpha}_{-2}$. The second belongs to $\mathscr S_k$ by
\cref{ic:scalar-terminal}. The whole-field bounds
$Z\in C^{k+2,\alpha}_1$ and
$\nabla Z\in C^{k+1,\alpha}_0$ show that the final two terms
belong to $C^{k+1,\alpha}_0$. This proves the Hessian assertion.

We check the remaining radial drift directly. Write
$Z=u\,\mathrm dr+r w_{\theta^i}\,\mathrm d\theta^i$ in a link chart.
The cone connection and symmetrized-gradient formulas are given in
\cite[Eq.~(3.2) and Lemma~3.14]{kroencke2022optimal}; their
symmetrized gradient has the opposite sign to our $\delta_C^*$.
With our positive-divergence convention, they give
\[
 \delta_CZ=u_r+\frac{n-1}{r} u+\frac1r\bar\delta w,
 \qquad
 (\Delta_C Z)_r=u_{rr}+\frac{n-1}{r} u_r
       +\frac1{r^2}(\bar\Delta u-(n-1)u-2\bar\delta w).
\]
The radial Ricci endomorphism of a cone vanishes. Put $f_C=r^2/4$.
The identity
$2\mathcal P_C=-\Delta_{f_C,C}-\mathrm d\delta_{f_C,C}
-\frac12-\Ric(g_C)$ has radial row
\[
 (2\mathcal P_CZ)_r=-(\Delta_CZ)_r-\partial_r(\delta_CZ)+r\partial_ru.
\]
Indeed the two drift contributions are each
$\frac r2\partial_ru$, and the zeroth-order terms
$\frac12u$ and $-\frac12u$ cancel. Multiplication by $r$ and
$\partial_t=r\partial_r$ therefore yield
\begin{equation}\label{ic:cone-radial-row}
\begin{split}
 r(2\mathcal P_CZ)_r={}&r^{-1}\bigl[-2\partial_t^2u
 +(4-2n)\partial_tu-\bar\Delta u+2(n-1)u\\
 &\hspace{27mm}+2\bar\delta w
       -\bar\delta(\partial_tw-w)\bigr]+r\partial_tu.
\end{split}
\end{equation}
Here $\mathcal P_C=\delta_{f_C,C}\delta_C^*$.

The canonical expansion
$g=g_C+2\Ric(g_C)+O(r^{-4})$ implies that the difference
between the radial rows of $2\mathcal P$ and $2\mathcal P_C$, after multiplication
by $r$, is a sum of contractions of
\begin{equation}\label{ic:radial-background-error}
 O(r^{-1})\nabla^2Z,
 \qquad O(r^{-2})\nabla Z,
 \qquad O(r^{-3})Z.
\end{equation}
For the second-order and ordinary first-order terms this follows from
$g^{-1}-g_C^{-1}=O(r^{-2})$ and the connection difference
$O(r^{-3})$. In the drift row, the two lower connection indices are
radial. The leading tensor $2\Ric(g_C)$ has zero radial row and zero
$rr$ connection correction, so that this connection correction is
$O(r^{-5})$; multiplication by the drift and by $r$ gives the last
bound in \cref{ic:radial-background-error}. All differentiated
coefficient bounds follow from the same canonical expansion.
Each term in \cref{ic:radial-background-error} belongs to
$\mathscr S_k$.

Write $v=\mathcal PZ=\chi\boldsymbol\tau(c_0)+v^\star$.
Ordinary divergence gives a raw term
$r^{-1}\bar\delta c_0\in C^{k,\alpha}_{-1}$ and a term in
$C^{k-1,\alpha}_{-2}$, together with faster background corrections.
Thus $\delta v\in\mathscr S_k$ by
\cref{ic:sum-multiplier}. Since $\delta_{f}v$ belongs
to the same space by definition,
\[
 rv_r=2(\delta v-\delta_{f}v)\in\mathscr S_k.
\]
All bracketed terms in \cref{ic:cone-radial-row}, including the
factor $r^{-1}$, are also in this space. Indeed,
$r^{-1}\partial_t^2u$ is a component of $r\nabla^2u$
plus controlled first derivatives and connection corrections.
The same observation treats $r^{-1}\bar\Delta u$.
For a direct use of \cref{ic:gauge-whole-fields}, put
$w^\star=w-rb_1-b_0$ on the far end. Radial cone parallel transport
and $\nabla_{\mathbf X}Z^\star\in C^{k,\alpha}_{-1}$ imply
$\partial_tw^\star\in C^{k,\alpha}_{-1}$: replacing
$\nabla_{\mathbf X}$ by $\nabla^C_{\mathbf X}$ adds an
$O(r^{-2})Z^\star\in C^{k,\alpha}_{-3}$ term. Hence
$\partial_tw-w+b_0=\partial_tw^\star-w^\star
\in C^{k,\alpha}_{-1}$.
For the angular terms, use
\[
 w=rb_1+b_0+O_{k+\alpha}(r^{-1}),\qquad
 \partial_tw-w=-b_0+O_{k+\alpha}(r^{-1}).
\]
Their contributions consist of $2\bar\delta b_1$ in
$C^{k+1,\alpha}_0$, raw terms
$r^{-1}\bar\delta b_0\in C^{k,\alpha}_{-1}$, and terms in
$C^{k-1,\alpha}_{-2}$. The remaining radial zeroth and first
derivatives have the same bounds. Rearranging
\cref{ic:cone-radial-row} proves $r\partial_tu\in\mathscr S_k$.

Consequently,
\[
 \partial_t\partial_{\theta^i}(w/r)
 =-r^{-1}\partial_{\theta^i} b_0
   +r^{-1}\partial_{\theta^i}(\partial_tw-w+b_0)\in\mathscr S_k.
\]
The two terms belong respectively to $C^{k,\alpha}_{-1}$ and
$C^{k-1,\alpha}_{-2}$.

Finally, in canonical coordinates the inverse metric satisfies
$g^{rr}-1=O(r^{-4})$ and $g^{r\theta^i}=O(r^{-5})$.
In the rescaled frame $\partial_r,r^{-1}\partial_{\theta^i}$, its tangential block differs from the cone inverse
by $O(r^{-2})$.
Thus raising $Z$ changes its radial component by a section of
$C^{k+2,\alpha}_{-3}$ and its angular coordinate component
$w^{\sharp_{g_\Sigma}}/r$ by a section of
$C^{k+2,\alpha}_{-2}$. These errors retain the asserted bounds
after each of the displayed differentiations and radial factors.
All maps used in this proof are bounded linear maps of the graph
parameter; the norm estimates follow as well.
\end{proof}

\subsection{Pullback estimates in the complete target norm}

For the local addition of \cref{lem:slice-product-affine-connection}
write $F_Z=(R,\Theta)$ on the end and put $r_F=R-r$.
Define
\[
\begin{gathered}
 \psi(\theta)=\exp^{g_\Sigma}_\theta
                  b_1^{\sharp_{g_\Sigma}}(\theta),\\
 \phi(\theta)=\mathrm D_v\exp^{g_\Sigma}_\theta
       \big|_{b_1^{\sharp_{g_\Sigma}}(\theta)}
                  b_0^{\sharp_{g_\Sigma}}(\theta).
\end{gathered}
\]
Here $\phi$ is a section of $\psi^*T\Sigma$.
In fixed link charts, write $\partial_i=\partial_{\theta^i}$ for
angular coordinate derivatives. Target angular derivatives are written
using $\Theta$.

\begin{Lemma}\label{ic:map-estimates}
For $Z$ in a sufficiently small ball in $\mathscr X_k$, the following
statements hold.
\begin{enumerate}[label=\textnormal{(\roman{*})},ref=\roman{*}]
\item The map and its first derivatives satisfy
\begin{align}
 R&=r+a_0+O_{k+\alpha}(r^{-1}),&
 \Theta&=\psi+r^{-1}\phi+O_{k+\alpha}(r^{-2}),\label{ic:map-expansion}\\
 \partial_{\theta^i}R-\partial_{\theta^i}a_0&=O_{k+\alpha}(r^{-1}),&
 \partial_{\theta^i}\Theta-\partial_{\theta^i}\psi&=O_{k+\alpha}(r^{-1}),\label{ic:map-first}\\
 r\partial_t\Theta&=-\phi+O_{k+\alpha}(r^{-1}),&
 \partial_rr_F&\in C^{k,\alpha}_{-2}.\label{ic:map-drift}
\end{align}
The coordinate expansions for $\Theta$ are understood in link charts.
\item The normalized differential of $F_Z$ belongs to
$C^{k+1,\alpha}_0$, and its whole covariant Hessian relative
to the source and target background connections belongs to
$C^{k,\alpha}_{-1}$.
\item The radial and mixed derivatives satisfy
\[
 r\nabla^2r_F,\quad r\partial_tr_F,
 \quad\partial_t\partial_{\theta^i}\Theta\in\mathscr S_k.
\]
\item The differences from the identity and remainders in these statements,
together with $a_0,b_0,b_1,\phi$ and $\psi-\mathrm{id}$ in link charts,
are bounded by $C\|Z\|_{\mathscr X_k}$.
Their dependence on $Z$ is smooth between the indicated fixed Banach
spaces, with bounded derivatives on smaller balls.
\end{enumerate}
\end{Lemma}
\begin{proof}
On the far canonical end the radial component of $Z^{\sharp_g}$ is $r_F$, and its
angular coordinate component is the link vector $\beta$ in
$\Theta=\exp^{g_\Sigma}_\theta\beta$. By
\cref{ic:graph-estimates,ic:radial-graph} and the inverse-metric
estimates at the end of the preceding proof,
\[
 \beta=b_1^{\sharp_{g_\Sigma}}
        +r^{-1}b_0^{\sharp_{g_\Sigma}}+O_{k+\alpha}(r^{-2}),\qquad
 \beta\in C^{k+2,\alpha}_0,
 \qquad r_F=a_0+O_{k+\alpha}(r^{-1}).
\]
Taylor expansion of the fixed smooth link exponential gives
\cref{ic:map-expansion}. The radial gradient estimate gives the first
assertion of \cref{ic:map-first}. The estimate
$\nabla Z^\star\in C^{k,\alpha}_{-1}$, together with
$b_0\in C^{k+1,\alpha}$ and the connection difference
$\nabla-\nabla^C=O(r^{-3})$, gives
\[
 \partial_i\beta
  -\partial_i b_1^{\sharp_{g_\Sigma}}=O_{k+\alpha}(r^{-1}).
\]
The pointwise chain rule proves the other assertion.
Similarly,
$r\partial_t\beta=-b_0^{\sharp_{g_\Sigma}}+O_{k+\alpha}(r^{-1})$.
Writing $E(\theta,v)=\exp^{g_\Sigma}_\theta(v)$ in fixed link charts,
we obtain
\[
\begin{split}
 r\partial_t\Theta+\phi
 ={}&E_v(\theta,\beta)
          \bigl(r\partial_t\beta+b_0^{\sharp_{g_\Sigma}}\bigr)\\
 &+\bigl(E_v(\theta,b_1^{\sharp_{g_\Sigma}})
             -E_v(\theta,\beta)\bigr)b_0^{\sharp_{g_\Sigma}}
 =O_{k+\alpha}(r^{-1}).
\end{split}
\]
This proves the angular drift estimate. The radial assertions follow
from \cref{ic:radial-graph}.

For the differential, use the source frame
$e_0=\partial_r$, $e_i=r^{-1}\partial_i$ and the target frame
$E_0=\partial_R$, $E_j=R^{-1}\partial_{\Theta^j}$.
The matrix $\mathrm dF_Z(e_b)=J_b^\mu E_\mu$ is
\[
 J=\begin{pmatrix}
  \partial_rR&r^{-1}\partial_iR\\
  R\partial_r\Theta^j&(R/r)\partial_i\Theta^j
 \end{pmatrix}.
\]
To estimate its full regularity, let
$D_\rho(s,\theta)=(\rho s,\theta)$ and
$g_\rho=\rho^{-2}D_\rho^*g$ on nested normalized annuli. The whole-field
estimate \cref{ic:whole-coordinate-regularity} gives
\[
\begin{gathered}
 F_{Z,\rho}:=D_\rho^{-1}\circ F_Z\circ D_\rho
 =\bigl(s+\rho^{-1}r_F(\rho s,\theta),
              E(\theta,\beta(\rho s,\theta))\bigr),\\
 \|F_{Z,\rho}-\operatorname{Id}\|_{C^{k+2,\alpha}}
       \leq C\|Z\|_{\mathscr X_k}.
\end{gathered}
\]
The constant is uniform in $\rho$, and smallness of $Z$ ensures that
the images lie in the larger normalized annuli. Differentiating this
whole map gives $J-I\in C^{k+1,\alpha}_0$ with the asserted bound.

Define the source and target connection coefficients by
$\nabla_{e_a}e_b=\Gamma^c_{ab}e_c$ and
$\nabla_{E_\nu}E_\lambda=\Gamma^\mu_{\nu\lambda}E_\mu$,
respectively. The full covariant Hessian is
\[
 (\nabla\mathrm dF_Z)^\mu_{ab}
 =e_a(J_b^\mu)-\Gamma^c_{ab}J_c^\mu
   +\Gamma^\mu_{\nu\lambda}(F_Z)J_a^\nu J_b^\lambda.
\]
The normalized metrics $g_\rho$ have uniform smooth bounds, and
$R/r$ is bounded above and below. Thus the differentiated matrix and
both connection terms belong to $C^{k,\alpha}_{-1}$. Their sum
vanishes at $Z=0$; the preceding annular estimate and the smooth
coefficient formulas give its norm bound by
$C\|Z\|_{\mathscr X_k}$. This proves the differential and Hessian
assertions using the whole-field regularity.

For the mixed derivative, the coordinate chain rule gives
\[
 \partial_t\partial_i\Theta
 =E_v(\theta,\beta)\,
       \partial_t\partial_i\beta
  +E_{iv}(\theta,\beta)\,\partial_t\beta
  +E_{vv}(\theta,\beta)
       [\partial_i\beta,\partial_t\beta].
\]
Here $E_{iv}=\partial_iE_v$, and
$\partial_t\partial_i\beta\in\mathscr S_k$ by
\cref{ic:angular-second} and its raised-index version.
The coefficient $E_v(\theta,\beta)$ belongs to
$C^{k+1,\alpha}_0$, so \cref{ic:sum-multiplier} controls the first
term in $\mathscr S_k$. Since $\beta\in C^{k+2,\alpha}_0$, the
remaining terms belong to $C^{k+1,\alpha}_0$. This proves the mixed
sum-space estimate. On the far canonical end $r_F=(Z^{\sharp_g})^r$ is linear in
$Z$, and \cref{ic:radial-graph} gives bounded linear maps
$Z\mapsto r\nabla^2r_F$ and $Z\mapsto r\partial_tr_F$ into
$\mathscr S_k$, with fixed compact extensions.

Finally, the graph trace maps $Z\mapsto(a_0,b_0,b_1)$ and the maps
$Z\mapsto r_F,\beta$ are bounded linear maps in the norms just used.
Taylor's formula with integral
remainder for $E$ expresses each asymptotic remainder as products of
the displayed trace and remainder factors with derivatives of $E$.
Parameter derivatives of these formulas have the same
weighted product bounds. The normalized differential and Hessian
formulas likewise have bounded parameter derivatives in their stated
norms. For the sum-space entry, put
$q_i(Z)=\partial_t\partial_i\beta_Z$ and, for a variation
$W\in\mathscr X_k$, write $\beta_W$ for the corresponding angular
component. The only term requiring the sum norm has variation
\[
 \mathrm D_Z\bigl[E_v(\theta,\beta_Z)q_i(Z)\bigr][W]
 =E_{vv}(\theta,\beta_Z)[\beta_W,q_i(Z)]
   +E_v(\theta,\beta_Z)q_i(W).
\]
Both terms are controlled by \cref{ic:sum-multiplier}; variations of
the other mixed terms remain in $C^{k+1,\alpha}_0$. Every higher
parameter derivative still contains at most one factor of the form
$q_i(Z)$ or $q_i(W)$, multiplied by controlled
$C^{k+1,\alpha}_0$ coefficients. This proves smooth dependence and
bounded derivatives in $\mathscr S_k$, and hence all the asserted
parameter estimates. On the compact core the ordinary local addition
and the local $C^{k+2,\alpha}$ bound in
\cref{ic:whole-coordinate-regularity} give the same conclusions.
\end{proof}

\begin{Proposition}\label{ic:pullback-closure}
Let $h\in E_{k+4}$ and $Z\in\mathscr X_k$ be sufficiently small, and
put $H=F_Z^*(g+h)-g$. Then the following hold.
\begin{enumerate}[label=\textnormal{(\roman{*})},ref=\roman{*}]
\item The whole tensor satisfies
\begin{equation}\label{ic:pullback-field}
 H\in C^{k+1,\alpha}_0.
\end{equation}
\item Write the input metric coefficients as
$c_0^h,c_{-1}^h,b_{-1}^h$ and put $\gamma=g_\Sigma+c_0^h$.
The tangential components satisfy
\[
 H_{\theta^i\theta^j}
 =r^2(c_0^H)_{\theta^i\theta^j}
  +r(c_{-1}^H)_{\theta^i\theta^j}+O_{k-1+\alpha}(1),
\]
where
\begin{equation}\label{ic:tangential-minus-one}
\begin{aligned}
 c_0^H&=\psi^*\gamma-g_\Sigma,\\
 c_{-1}^H&=2a_0\psi^*\gamma+\psi^*c_{-1}^h
                      +\mathrm D_{\psi,\phi}\gamma.
\end{aligned}
\end{equation}
Here $\mathrm D_{\psi,\phi}\gamma$ denotes the first variation of
$\psi^*\gamma$ in the direction $\phi$, and
\[
 c_0^H\in C^{k+1,\alpha}(\Sigma,S^2T^*\Sigma),\qquad
 c_{-1}^H\in C^{k,\alpha}(\Sigma,S^2T^*\Sigma).
\]
The mixed and radial components satisfy
\begin{equation}\label{ic:pullback-radial}
 H_{r\theta^i}=\Omega_{\theta^i}+O_{k+\alpha}(r^{-1}),\qquad
 H_{rr}\in C^{k,\alpha}_{-2},\qquad
 r^2H_{rr}\in\mathscr S_k,
\end{equation}
where
\begin{equation}\label{ic:mixed-trace}
 \Omega=\mathrm d_\Sigma a_0+\psi^*b_{-1}^h
                   -\gamma_\psi(\phi,\mathrm d\psi)
         \in C^{k+1,\alpha}(\Sigma,T^*\Sigma).
\end{equation}
In particular, $H\in E_{k-1}$.
\item The weighted divergence satisfies
\begin{equation}\label{ic:residual-closure}
 P(h,Z):=\delta_fH
 =-\tfrac12\chi\boldsymbol\tau(\Omega)+P^\star,
 \quad P^\star\in C^{k,\alpha}_{-1},\quad
 \delta_fP(h,Z)\in\mathscr S_k.
\end{equation}
\end{enumerate}
All the indicated norms are bounded by
$C(\|h\|_{E_{k+4}}+\|Z\|_{\mathscr X_k})$ on a fixed small ball.
\end{Proposition}
\begin{proof}
In fixed link charts, omit $\theta$ from angular component indices.
In compositions we use
$i,j$ for source angular indices and $a,b,\ell$ for target angular indices.
Put $F=F_Z$, $\tilde g=g+h$, and
$N=\|h\|_{E_{k+4}}+\|Z\|_{\mathscr X_k}$.
All constants below are uniform on a fixed small ball. We first work
on the far end, where $\chi=1$. The cutoff terms and the compact
core are included at the end of the proof.

\noindent\textup{(i)}
By \cref{ic:map-estimates}, the normalized components of
$\mathrm dF$ belong to $C^{k+1,\alpha}_0$ with uniformly bounded norm.
Since $R/r$ is bounded above and below, composition with $F$
preserves the weights of the metric input.
The coordinate pullback formula bounds $F^*h$ in
$C^{k+1,\alpha}_0$ by $C\|h\|_{E_{k+4}}$.
For the background contribution, use the normalized map $F_{Z,\rho}$
of \cref{ic:map-estimates}. In a fixed chart, let $G$ be the matrix
of $g_\rho$ and $J_\rho=\mathrm dF_{Z,\rho}$. The map estimate,
the chain rule, and the H\"older product inequality give
\[
 \|J_\rho-I\|_{C^{k+1,\alpha}}
 +\|G(F_{Z,\rho})-G\|_{C^{k+1,\alpha}}
                  \leq C\|Z\|_{\mathscr X_k}.
\]
Here the coefficient estimate uses the uniform smooth bounds for
$g_\rho$ and the $C^{k+2,\alpha}$ bound for
$F_{Z,\rho}-\operatorname{Id}$. The exact matrix identity is
\[
\begin{split}
 F_{Z,\rho}^*G-G={}&G(F_{Z,\rho})-G
                   +(J_\rho-I)^{\mathsf T}G(F_{Z,\rho})\\
 &+G(F_{Z,\rho})(J_\rho-I)
       +(J_\rho-I)^{\mathsf T}G(F_{Z,\rho})(J_\rho-I).
\end{split}
\]
Multiplication in $C^{k+1,\alpha}$ and scaling back therefore give
\[
 \|F^*g-g\|_{C^{k+1,\alpha}_0}
 \leq C\bigl(\|Z\|_{\mathscr X_k}+\|Z\|_{\mathscr X_k}^2\bigr)
 \leq C\|Z\|_{\mathscr X_k}.
\]
The same matrix calculation applies on the fixed compact core.
Together these estimates prove the bound for $H$.

\noindent\textup{(ii)}
The canonical background expansion and the definition of $E_{k+4}$
give the exact component representation
\[
 \tilde g_{rr}=1+r^{-2}A,\qquad
 \tilde g_{ri}=B_{i},\qquad
 \tilde g_{ij}=r^2\Gamma_{ij},
\]
where the fixed-link coefficient families $A,B,\Gamma$ have
bounded $C^{k+4,\alpha}_0$ norms and
\[
 B=b_{-1}^h+O_{k+4+\alpha}(r^{-1}),\qquad
 \Gamma=\gamma+r^{-1}c_{-1}^h+O_{k+4+\alpha}(r^{-2}).
\]
Component remainders in these formulas are measured as families on
the fixed link. A tangential coordinate covector has $g$-norm comparable to
$r^{-1}$ times its fixed-link norm.
All comparisons of $\Theta$ and $\psi$ below are made in the fixed
link charts used in \cref{ic:map-estimates}.

Set $v=\partial_rR=1+\partial_rr_F$. The mixed block of the
pullback is exactly
\[
\begin{split}
 H_{ri}={}&
 (1+R^{-2}A(F))v\,\partial_iR
 +B(F)\bigl(v\,\partial_i\Theta+\partial_iR\,\partial_r\Theta\bigr)\\
 &+R^2\Gamma(F)(\partial_r\Theta,\partial_i\Theta)
                  -g_{ri}.
\end{split}
\]
Use the estimates
\[
\begin{gathered}
 \partial_iR=\partial_i a_0+O_{k+\alpha}(r^{-1}),\qquad
 \partial_i\Theta=\partial_i\psi+O_{k+\alpha}(r^{-1}),\\
 r^2\partial_r\Theta=r\partial_t\Theta=-\phi+O_{k+\alpha}(r^{-1}),\qquad
 v=1+O_{k+\alpha}(r^{-2}),\qquad R/r=1+O_{k+\alpha}(r^{-1}).
\end{gathered}
\]
The first term of the mixed-block formula contributes
$\partial_i a_0+O_{k+\alpha}(r^{-1})$.
The second contributes
$b_{-1}^h|_\psi(\partial_i\psi)+O_{k+\alpha}(r^{-1})$;
its term containing $\partial_r\Theta$ has weight at most minus
two. The third contributes
$-\gamma_\psi(\phi,\partial_i\psi)+O_{k+\alpha}(r^{-1})$.
Finally $g_{ri}=O(r^{-3})$ by the canonical expansion.
This proves \cref{ic:mixed-trace} and
$H_{ri}-\Omega_{i}=O_{k+\alpha}(r^{-1})$.
The regularities $a_0,\psi\in C^{k+2,\alpha}$ and
$\phi\in C^{k+1,\alpha}$ show that every term of $\Omega$ belongs
to $C^{k+1,\alpha}$, with norm bounded by $CN$.

For the tangential block, the terms containing a derivative of $R$
have normalized weight at most minus two. Hence
\[
 r^{-2}H_{ij}
 =(R/r)^2\Gamma(F)
       (\partial_i\Theta,\partial_j\Theta)
       -(g_\Sigma)_{ij}+O_{k-1+\alpha}(r^{-2}).
\]
Taylor expansion of $R=r+a_0+O_{k+\alpha}(r^{-1})$ and
$\Theta=\psi+r^{-1}\phi+O_{k+\alpha}(r^{-2})$ gives the coefficients
in \cref{ic:tangential-minus-one}. In coordinates, the first variation is
\[
\begin{split}
 (\mathrm D_{\psi,\phi}\gamma)_{ij}
 ={}&\phi^\ell(\partial_\ell\gamma_{ab})(\psi)
                     \,\partial_i\psi^a\,\partial_j\psi^b\\
 &+\gamma_{ab}(\psi)\,\partial_i \phi^a\,\partial_j\psi^b
  +\gamma_{ab}(\psi)\,\partial_i\psi^a\,\partial_j \phi^b.
\end{split}
\]
These three terms are the coefficient of the first power of $s$ in
$\gamma_{ab}(\psi+s\phi)\,\partial_i(\psi^a+s\phi^a)
\partial_j(\psi^b+s\phi^b)$.
It uses one derivative of $\phi$, so
$c_0^H\in C^{k+1,\alpha}$ and $c_{-1}^H\in C^{k,\alpha}$.
One angular derivative of the map remainder
$O_{k+\alpha}(r^{-2})$ has order $k-1$. Consequently
\[
 r^{-2}H_{ij}
  =(c_0^H)_{ij}+r^{-1}(c_{-1}^H)_{ij}+O_{k-1+\alpha}(r^{-2}),
\]
with all coefficient and remainder norms bounded by $CN$.

Put $\tilde A=r^2(g_{rr}-1)$. The exact radial formula is
\[
 H_{rr}=(v^2-1)+R^{-2}A(F)v^2+2B(F)(v\,\partial_r\Theta)
       +R^2\Gamma(F)(\partial_r\Theta,\partial_r\Theta)-r^{-2}\tilde A.
\]
Since $\partial_rr_F\in C^{k,\alpha}_{-2}$,
$v^2-1=(2+\partial_rr_F)\partial_rr_F$ has weight minus two.
Also $\partial_r\Theta=r^{-2}(-\phi+O_{k+\alpha}(r^{-1}))$.
Every other term in this formula has weight at most minus two and
order $k$. Thus $H_{rr}\in C^{k,\alpha}_{-2}$.
Multiplication by $r^2$ yields
\begin{equation}\label{ic:radial-pullback-exact}
\begin{split}
 r^2H_{rr}={}&(2+\partial_rr_F)r\partial_tr_F
 +(r/R)^2A(F)v^2\\
 &+2B(F)(v\,r\partial_t\Theta)
 +(R/r)^2\Gamma(F)(r\partial_t\Theta,r\partial_t\Theta)
                   -\tilde A.
\end{split}
\end{equation}
Together with the mixed and tangential expansions, the radial bound
gives the following expansion:
\[
 H=\chi\bigl\{\boldsymbol\tau(c_0^H)
       +r^{-1}\boldsymbol\tau(c_{-1}^H
                            +2\,\mathrm dr\odot\Omega)\bigr\}
          +H^\star,\qquad
 H^\star\in C^{k-1,\alpha}_{-2}.
\]
Thus $H\in E_{k-1}$.

The normalized differential estimate in \cref{ic:map-estimates}
gives $\partial_rr_F\in C^{k+1,\alpha}_0$, in addition to its
weight-minus-two bound at order $k$. The factors
\[
 r/R,\quad R/r,\quad v,\quad A(F),\quad B(F),\quad\Gamma(F)
                  \in C^{k+1,\alpha}_0
\]
are therefore controlled multipliers of $\mathscr S_k$.
Here the composed coefficient estimates use the $E_{k+4}$ input
and the normalized $C^{k+2,\alpha}$ map bounds.
Since $r\partial_tr_F\in\mathscr S_k$ by
\cref{ic:radial-graph}, the first term of
\cref{ic:radial-pullback-exact} lies in $\mathscr S_k$.

Set $e=r\partial_t\Theta+\phi\in C^{k,\alpha}_{-1}$ in the fixed
link charts. Inserting $r\partial_t\Theta=-\phi+e$ into the mixed
and quadratic angular terms separates them into three types.
Terms involving only $\phi\in C^{k+1,\alpha}$ belong to the high
summand $C^{k+1,\alpha}_0$. Terms linear in $e$ belong to
$C^{k,\alpha}_{-1}\hookrightarrow\mathscr S_k$ by the
embedding and multiplier estimate of \cref{ic:sum-multiplier}.
Terms quadratic in $e$ belong to
$C^{k,\alpha}_{-2}\subset C^{k-1,\alpha}_{-2}$.

The remaining radial coefficient terms belong to the high summand.
To record their quantitative bound, write
$A-\tilde A=r^2h_{rr}$ and $\rho=(r/R)^2v^2$. Then
\[
 \rho A(F)-\tilde A
 =\rho(A-\tilde A)(F)+(\rho-1)\tilde A(F)
                      +\tilde A(F)-\tilde A.
\]
The first term has high norm bounded by $C\|h\|_{E_{k+4}}$.
The other two have high norm bounded by $C\|Z\|_{\mathscr X_k}$:
$\rho-1$ has that bound, and the last difference is the composition
difference of a fixed smooth background coefficient.
Thus \cref{ic:radial-pullback-exact} proves
\[
 \|r^2H_{rr}\|_{\mathscr S_k}\leq CN.
\]
The same grouping gives $\|H_{rr}\|_{C^{k,\alpha}_{-2}}\leq CN$
in the unscaled radial formula. The difference $A-\tilde A$ is
estimated directly in $C^{k+4,\alpha}_0$. The preceding coefficient
expansions consequently also give $\|H\|_{E_{k-1}}\leq CN$.

\noindent\textup{(iii)}
The mixed coordinate remainder has weight minus one; its contribution
to $H(\partial_r,\cdot)$ has weight minus two with respect to $g$. With the radial estimate this gives
\[
 H(\partial_r,\cdot)
   =\chi r^{-1}\boldsymbol\tau(\Omega)
                      +O_{k+\alpha}(r^{-2}).
\]
Since $\mathbf X=r\partial_r/2$ on the canonical end,
\[
 H(\mathbf X,\cdot)
  =\tfrac12\chi\boldsymbol\tau(\Omega)
                      +O_{k+\alpha}(r^{-1}).
\]
By (i), $\delta H\in C^{k,\alpha}_{-1}$. Subtracting this radial
contraction proves
\[
 P=\delta_fH
 =-\tfrac12\chi\boldsymbol\tau(\Omega)+P^\star,
 \qquad \|P^\star\|_{C^{k,\alpha}_{-1}}\leq CN.
\]

For the fixed-link one-form family
$H_{\cdot r}=H_{ir}\,\mathrm d\theta^i$, the cone
radial divergence formula is
\begin{equation}\label{ic:radial-divergence-H}
 r(\delta_CH)_r=\partial_tH_{rr}+(n-1)H_{rr}
 +r^{-1}\bar\delta(H_{\cdot r})
 -r^{-2}(g_\Sigma)^{ij}H_{ij}.
\end{equation}
Each term has a definite position in the scalar sum:
\[
\begin{aligned}
 \partial_tH_{rr}&\in C^{k-1,\alpha}_{-2},&
 H_{rr}&\in C^{k,\alpha}_{-2},\\
 r^{-1}\bar\delta(H_{\cdot r}-\Omega)
      &\in C^{k-1,\alpha}_{-2},&
 r^{-1}\bar\delta\Omega&\in C^{k,\alpha}_{-1},\\
 r^{-2}(g_\Sigma)^{ij}H_{ij}
      &\in C^{k+1,\alpha}_0.
\end{aligned}
\]
The differentiated mixed remainder loses one order and gains the
displayed factor $r^{-1}$; the raw trace derivative has order $k$
because $\Omega\in C^{k+1,\alpha}$. Its inclusion in
$\mathscr S_k$ is precisely \cref{ic:sum-multiplier}.
The last term uses the whole-field high-order estimate for $H$.

The inverse-metric and connection differences give contractions
of the form
\[
 r(\delta-\delta_C)H
   =rO(r^{-2})\nabla^CH+rO(r^{-3})H.
\]
The first term belongs to $C^{k,\alpha}_{-2}$, hence to the
lower summand, and the second belongs to
$C^{k+1,\alpha}_{-2}$, hence to the high summand. Here
$\nabla^CH=\nabla H+O(r^{-3})H$ has the required
$C^{k,\alpha}_{-1}$ bound. It follows that
\[
 r(\delta H)_r\in\mathscr S_k,\qquad
 rP_r=r(\delta H)_r-\tfrac12r^2H_{rr}\in\mathscr S_k,
\]
with norms bounded by $CN$.

Ordinary divergence of the one-form expansion gives
\[
 \delta P=-\tfrac12\chi r^{-1}\bar\delta\Omega
                  +\delta P^\star+E_\delta.
\]
The first term belongs to $C^{k,\alpha}_{-1}$ and
$\delta P^\star\in C^{k-1,\alpha}_{-2}$.
The background part of $E_\delta$ is a sum of
$O(r^{-2})\nabla^C\boldsymbol\tau(\Omega)$ and
$O(r^{-3})\boldsymbol\tau(\Omega)$; these have orders $k$
and $k+1$ and weight at most minus three. Thus they also belong
to $\mathscr S_k$, with norm at most $CN$.
Finally,
\[
 \delta_fP=\delta P-\tfrac12rP_r
   =\delta P-\tfrac12r(\delta H)_r+\tfrac14r^2H_{rr}
                 \in\mathscr S_k,
 \qquad \|\delta_fP\|_{\mathscr S_k}\leq CN.
\]

All cutoff derivatives are supported in the fixed transition
annulus and have the displayed finite-order bounds. On the compact
core, the local addition and the $C^{k+2,\alpha}$ control of $Z$
give $H\in C^{k+1,\alpha}$, $P\in C^{k,\alpha}$, and
$\delta_fP\in C^{k-1,\alpha}$ with norms bounded by $CN$.
These are the required local orders for $E_{k-1}$ and
$\mathscr Y_k$, since the compact part of the lower scalar
summand has order $k-1$. This completes all the estimates.
\end{proof}

\subsection{The fixed Banach slice}

\begin{Proposition}\label{ic:forward-C1}
On a sufficiently small open neighborhood of $(0,0)$,
\[
 P:E_{k+4}\times\mathscr X_k
          \longrightarrow\mathscr Y_k\cap\mathcal K^\perp
\]
is $C^1$ with locally Lipschitz differential. In the product norm,
on a smaller convex neighborhood,
\begin{align}
 \|P(x+e)-P(x)-\mathrm DP(x)e\|_{\mathscr Y_k}
      &\leq C\|e\|^2,\label{ic:Frechet-remainder}\\
 \|\mathrm DP(x)-\mathrm DP(y)\|
      &\leq C\|x-y\|.\label{ic:derivative-Lipschitz}
\end{align}
Moreover $\mathrm D_ZP(0,0)=-2\mathcal P$.
\end{Proposition}
\begin{proof}
Write $x=(h,Z)$ and $\xi=(\ell,W)$, with
$\|\xi\|=\|\ell\|_{E_{k+4}}+\|W\|_{\mathscr X_k}$.
We work on a sufficiently small convex ball, with uniform constants.

\textnormal{Step 1.} We first record the composition estimate used below.
On nested normalized coordinate charts, the map
$(a,F)\mapsto a\circ F$ is $C^2$ into $C^{j,\alpha}$ for
$a\in C^{j+3,\alpha}$ and maps $F$ in a bounded
$C^{j,\alpha}$ neighborhood, where $j\geq1$.
For variations $(a_1,U_1)$ and $(a_2,U_2)$, its second differential is
\[
 \mathrm d^2a(F)[U_1,U_2]
       +\mathrm da_1(F)U_2+\mathrm da_2(F)U_1.
\]
The H\"older product and chain rules bound this expression in
$C^{j,\alpha}$. The additional derivative gives continuity in that
same norm: composition of $\mathrm d^2a$ at two nearby maps is
controlled using $\mathrm d^3a\in C^{j,\alpha}$.
The first-order integral formula then identifies the displayed
expression as the continuous second Banach differential. Composing
with $Z\mapsto F_Z$ adds the usual term
$\mathrm da(F_Z)\mathrm D^2F_Z[W_1,W_2]$, controlled by
\cref{ic:map-estimates}.

In the formulas of \cref{ic:pullback-closure}, two parameter
variations use at most $k+3$ metric derivatives: $H$ is measured
at order $k+1$, $\nabla H$ at order $k$, and second spatial
derivatives at order $k-1$. The high radial and link coefficients
are measured at most at order $k+1$. The remaining derivative
provided by $h\in E_{k+4}$ gives continuity of the second
differentials in the stated H\"older norms. The whole maps have
uniform $C^{k+2,\alpha}$ bounds by
\cref{ic:whole-coordinate-regularity,ic:map-estimates}.
Comparable radii preserve all weights, so the estimates are uniform
over the end charts. The compact core uses finitely many ordinary
charts.

\textnormal{Step 2.} Apply this calculus and \cref{ic:map-estimates}
to the exact pullback formulas in \cref{ic:pullback-closure}.
They give $C^2$ maps, with bounded first and second differentials,
into the following spaces:
\[
\begin{gathered}
 H\in C^{k+1,\alpha}_0,\qquad
 \nabla H\in C^{k,\alpha}_{-1},\qquad
 H_{rr}\in C^{k,\alpha}_{-2},\\
 \Omega\in C^{k+1,\alpha}(\Sigma,T^*\Sigma),\qquad
 H_{\cdot r}-\Omega\in C^{k,\alpha}_{-1}.
\end{gathered}
\]
The gradient assertion follows by bounded covariant differentiation.
The last norm is the fixed-link family norm. The weighted remainder
claims follow from the factorizations in that proof: each term
contains a decaying map remainder, metric remainder, or radial
factor. These factors remain decaying after parameter variation.
For example, differentiating $b(\Theta)-b(\psi)$ and regrouping
retains a difference of the maps or their variations in every term;
the same holds after two variations. At target order $k$, the
coefficient differences use at most $k+3$ derivatives, with the
same extra derivative for continuity. Thus the weighted assertions
follow in their full $C^{k,\alpha}$ norms.

For the radial sum-space entry, keep the source-coordinate factor
\[
 q(Z)=r\partial_tr_F\in\mathscr S_k,
 \qquad r\partial_t\Theta=-\phi+e,
 \qquad e\in C^{k,\alpha}_{-1}.
\]
On the far end $q$ is bounded linear in $Z$, while $\phi$ and $e$
depend smoothly on $Z$ in $C^{k+1,\alpha}$ and
$C^{k,\alpha}_{-1}$, respectively. By
\cref{ic:sum-multiplier}, multiplication is a bounded bilinear map
\[
 C^{k+1,\alpha}_0\times\mathscr S_k\longrightarrow\mathscr S_k.
\]
Indeed, estimate the two summands of any decomposition and take
the infimum in the quotient norm. The first term of
\cref{ic:radial-pullback-exact} is
$(2+\partial_rr_F)q(Z)$, so it is $C^2$ into $\mathscr S_k$.
The term $(r/R)^2A(F)v^2-\tilde A$ is $C^2$ into the high
summand. In the angular terms, the factors involving only $\phi$
belong to $C^{k+1,\alpha}_0$; those linear in $e$ belong to
$C^{k,\alpha}_{-1}\subset\mathscr S_k$; and the term quadratic
in $e$ belongs to $C^{k,\alpha}_{-2}$, hence to the lower
summand. All accompanying coefficients are controlled
$C^{k+1,\alpha}_0$ multipliers. Bounded product rules therefore
prove that $r^2H_{rr}$ is $C^2$ into $\mathscr S_k$, with bounded
second differential. Compactly supported terms have the required
order $k-1$ in the lower summand.

\textnormal{Step 3.} The complete target follows by the fixed linear
operations in \cref{ic:pullback-closure}. On the far end,
\[
 P^\star=\delta H-\tfrac12rH_{rr}\,\mathrm dr
             -\tfrac12\boldsymbol\tau_r(H_{\cdot r}-\Omega).
\]
Contraction of $\nabla H$, multiplication of $H_{rr}$ by $r$,
and radial identification of the mixed remainder are bounded into
$C^{k,\alpha}_{-1}$. Hence $P^\star$ is $C^2$ in that norm.
Equation \eqref{ic:radial-divergence-H}, with the fixed background
correction, gives $r(\delta H)_r$ by bounded linear maps from the
fields in Step 2 to $\mathscr S_k$: the differentiated mixed
remainder lies in the lower summand, the trace derivative lies in
$C^{k,\alpha}_{-1}$, and the tangential trace of $H$ lies in the
high summand. Similarly,
\[
\begin{aligned}
 \delta P&=-\tfrac12\chi r^{-1}\bar\delta\Omega
                  +\delta P^\star+E_\delta(\Omega),\\
 \delta_fP&=\delta P-\tfrac12r(\delta H)_r+\tfrac14r^2H_{rr},
\end{aligned}
\]
where $E_\delta$ is the fixed background and cutoff correction
already estimated in \cref{ic:pullback-closure}.
Every operation is bounded into $\mathscr S_k$. On the compact
core, $H\mapsto\delta_f\delta_fH$ is bounded from
$C^{k+1,\alpha}$ to $C^{k-1,\alpha}$.

It follows that $(-\Omega/2,P^\star,\delta_fP)$ is $C^2$ in
the product defining \cref{ic:target-definition}. The compatibility
condition defining the graph of $\delta_f$ is closed in this
product, by distributional continuity, as in
\cref{ic:graph-isomorphism}. Its first and second differentials
therefore belong to the same closed subspace. Consequently $P$
is $C^2$ into $\mathscr Y_k$, and
\[
 \|\mathrm D^2P(x)[\xi,\eta]\|_{\mathscr Y_k}
       \leq C\|\xi\|\|\eta\|.
\]
Integrating this bound along line segments proves
\cref{ic:Frechet-remainder,ic:derivative-Lipschitz}.

\textnormal{Step 4.} If $\eta\in\mathcal K$, the Gaussian energy
identity gives $\delta^*\eta=0$. Integration by parts with a
cutoff $\chi_R$, equal to one on a ball of radius $R$ and with
$|\mathrm d\chi_R|\leq C/R$, gives
\[
 \langle\chi_RP,\eta\rangle_{L^2_f}
 =\langle\chi_RH,\delta^*\eta\rangle_{L^2_f}
       -\langle H,\mathrm d\chi_R\odot\eta\rangle_{L^2_f}.
\]
Here $H,P\in L^2_f$ by their polynomial bounds and
$\eta\in H^2_f$. Cauchy--Schwarz makes the cutoff error tend
to zero, proving $P(h,Z)\perp\mathcal K$. This is a fixed
closed subspace, so it also contains the derivatives above.
Finally, the differential of the local addition at the zero
section is the identity, and therefore
\[
 \mathrm D_ZP(0,0)W
 =\delta_f\mathcal L_{W^{\sharp_g}}g
 =-2\delta_f\delta^*W=-2\mathcal PW.
\]
\end{proof}

\begin{Theorem}[Fixed-order weighted slice]\label{ic:finite-slice}
Let $(M,g,f)$ be a complete nontrivial AC Ricci shrinker.\phantomsection\label{idx:fixed-Banach-slice}
Fix $k\geq6$ and $0<\alpha<1$. There are $\epsilon_k,C_k>0$
and a neighborhood of zero in $\mathscr X_k$ such that every
$h\in E_{k+4}$ with $\|h\|_{E_{k+4}}<\epsilon_k$ has a unique
parameter $Z(h)$ in this neighborhood satisfying
\[
 \delta_{f}\bigl(F_{Z(h)}^*(g+h)-g\bigr)=0.
\]
The branch $h\mapsto Z(h)$ is $C^1$ and
\[
 \|Z(h)\|_{\mathscr X_k}
 +\|F_{Z(h)}^*(g+h)-g\|_{E_{k-1}}
       \leq C_k\|h\|_{E_{k+4}}.
\]
The map $F_{Z(h)}$ is a global $C^{k+2,\alpha}_{\rm loc}$
diffeomorphism isotopic to the identity.
\end{Theorem}
\begin{proof}
The graph isomorphism of \cref{ic:graph-isomorphism} and
\cref{ic:forward-C1} meet the
hypotheses of the $C^1$ Banach implicit function theorem.
It gives the branch, uniqueness in a small graph ball, and
$\|Z(h)\|_{\mathscr X_k}\leq C_k\|h\|_{E_{k+4}}$.
\Cref{ic:pullback-closure} gives the stated $E_{k-1}$ estimate.

For small $Z$, \cref{ic:map-estimates} makes $F_Z$ a local
diffeomorphism everywhere; on the core use a fixed finite atlas,
and on the end use its normalized differential and the small
$C^1$ perturbation $\psi$ of the link identity. Its radial component
is $r+O(1)$, so $F_Z$ is proper. The same statements hold uniformly
for $F_{sZ}$, $0\leq s\leq1$, giving a proper local-diffeomorphism
homotopy from the identity. A proper local diffeomorphism is a
covering map; this homotopy induces the identity map on the
fundamental group, so the covering has one sheet on each connected
component. Hence $F_Z$ is a global diffeomorphism, and the homotopy
is the asserted isotopy. The local differentiability of the map
and its inverse follows from the whole-field
$C^{k+2,\alpha}_{\rm loc}$ estimate and the finite-dimensional
inverse function theorem.
\end{proof}

\Needspace{4\baselineskip}
\begin{Corollary}\label{ic:sliced-radial-improvement}
Let $m\geq2$ and $h\in E_m$ satisfy $\delta_fh=0$.
In \eqref{eq:slice-E-expansion}, one has $b_{-1}=0$, and on the end
\[
 h_{rr}=-2r^{-2}\tr_{g_\Sigma}c_0
        -2r^{-3}\tr_{g_\Sigma}c_{-1}
        +O_{m-1+\alpha}(r^{-4}).
\]
The remainder norm is bounded by $C\|h\|_{E_m}$.
\end{Corollary}
\begin{proof}
The exact equation is $h(\mathbf X,\cdot)=\delta h$.
The right side belongs to $C^{m-1,\alpha}_{-1}(M,T^*M)$,
with norms taken using the background metric $g$, whereas the
degree-zero term on the left side is
$\frac12\boldsymbol\tau(b_{-1})$. Hence $b_{-1}=0$.
In the component formulas below, the $O_{j+\alpha}$ bounds use fixed-link
components with conically scaled derivatives. The radial cone divergence is
\[
 (\delta_Ch)_r=\partial_rh_{rr}+\frac{n-1}{r}h_{rr}
        +r^{-2}\bar\delta(h_{\cdot r})
        -r^{-3}\tr_{g_\Sigma}(h_{\theta^i\theta^j}).
\]
Here $h_{rr}=O_{m+\alpha}(r^{-2})$,
$h_{r\theta^i}=O_{m+\alpha}(r^{-1})$, and
\[
 h_{\theta^i\theta^j}
 =r^2(c_0)_{\theta^i\theta^j}+r(c_{-1})_{\theta^i\theta^j}
    +O_{m+\alpha}(1).
\]
The canonical inverse-metric and connection errors contribute
$O_{m-1+\alpha}(r^{-3})$ to the divergence. Consequently
\[
 (\delta h)_r=-r^{-1}\tr_{g_\Sigma}c_0
      -r^{-2}\tr_{g_\Sigma}c_{-1}
      +O_{m-1+\alpha}(r^{-3}).
\]
Since $\mathbf X=(r/2)\partial_r$, multiplication by $2/r$
gives the expansion. Every error norm is bounded by $C\|h\|_{E_m}$.
\end{proof}

\subsection{Application: Rigidity of AC Ricci shrinkers}
\label{sec:AC-rigidity}

Let $(M^n,\hg,\hf)$ be a complete nontrivial AC Ricci shrinker,
and choose a minimum point $\hat p$ of $\hf$. We use the normalization
\cref{eq:normalization-of-potential} and the fixed background Gaussian norms.
Throughout this subsection, hats denote the fixed background geometry.
The spaces $E_m$, $\mathscr X_k$, and $\mathscr Y_k$ are formed
using $(\hat g,\hat f)$.
Set
\begin{equation}\label{fik:stability-operator}
 \begin{aligned}
 \hat L&:=\frac12\hat\Delta_{\hat f}+\Rm_{\hat g},\\
 \operatorname{ISD}(\hg,\hf)
 &:=\{h\in H^2_{\hat f}(S^2T^*M):
       \hat Lh=0,\ \hat\delta_{\hat f}h=0\}.
 \end{aligned}
\end{equation}
This is the infinitesimal solitonic deformation space of
Li--Zhang \cite[Definition~3.18]{li2023rigidity}; their weighted
Lichnerowicz operator is $2\hat L$. We abbreviate it to $\operatorname{ISD}$.
For a metric--potential pair, define the soliton defect by
\begin{equation}\label{fik:defect}
 \Phi(g,f):=\frac12g-\Ric_g-\nabla_g^2f.
\end{equation}

\begin{Theorem}[Pointed Gromov--Hausdorff rigidity of AC shrinkers]
\label{thm:AC-rigidity}
If $\operatorname{ISD}=\{0\}$, there is a pointed Gromov--Hausdorff
neighborhood of $(M,\hg,\hat p)$ such that every complete
$n$-dimensional Ricci shrinker $(N,g,f,p)$ in this neighborhood, based at
a minimum point of $f$, has $(N,g)$ isometric to $(M,\hg)$.
\end{Theorem}

We first obtain global comparison maps and convergence in each fixed $E_m$.
The slice theorem then reduces the rigidity problem to an estimate for
$\Phi$ on weighted divergence free perturbations.

\begin{Lemma}\label{fik:lem:comparison}
Let $(M_i^n,g_i,f_i,p_i)$ be a sequence of complete Ricci shrinkers, based at minimum points of $f_i$, converging pointed smoothly to a complete nontrivial AC shrinker $(M^n,\hg,\hf,\hat p)$ with compact link $\Sigma$. For all sufficiently large $i$, $M_i$ is AC and there are global diffeomorphisms $\Psi_i:M\to M_i$ such that
\begin{equation}\label{fik:eq:E-convergence}
 \Psi_i^*g_i-\hg\in\bigcap_{m\ge2}E_m,\qquad
 \norm{\Psi_i^*g_i-\hg}_{E_m}\longrightarrow0
 \quad\text{for every fixed }m.
\end{equation}
All potentials have normalization \cref{eq:normalization-of-potential}. The maps identify canonical end coordinates. In these coordinates,
\begin{equation}\label{fik:eq:potential-before-slice}
 \Psi_i^*f_i=\boldsymbol{\mu}(g_i)+\frac14r^2+O_\infty(r^{-2}),
 \qquad \hf=\boldsymbol{\mu}(\hg)+\frac14r^2+O_\infty(r^{-2}),
\end{equation}
with bounds uniform in $i$. On compact sets the pulled-back metrics and potentials converge smoothly.
\end{Lemma}

\begin{proof}
By \cref{ineq:quadratic-growth-of-potential},
\begin{equation}\label{fik:eq:potential-growth}
 \frac14(d_{g_i}(p_i,x)-5n)_+^2
 \le f_i(x)-\boldsymbol{\mu}(g_i)\le\frac14(d_{g_i}(p_i,x)+\sqrt{2n})^2.
\end{equation}
Choose a large regular value $a$ of $\hf-\boldsymbol{\mu}(\hg)$. The sublevel $\{f_i-\boldsymbol{\mu}(g_i)\le a+1\}$ is contained in the ball of radius $2\sqrt{a+1}+5n$ about $p_i$. A pointed smooth comparison on a larger ball therefore contains this entire sublevel set. Smooth convergence and an adjustment in a collar of $\{\hf-\boldsymbol{\mu}(\hg)=a\}$ identify
\[
 \{\hf-\boldsymbol{\mu}(\hg)\le a\}\cong\{f_i-\boldsymbol{\mu}(g_i)\le a\},\qquad
 \Sigma_a:=\{\hf-\boldsymbol{\mu}(\hg)=a\}\cong\Sigma_{i,a}:=\{f_i-\boldsymbol{\mu}(g_i)=a\}.
\]
We use shrinker pseudolocality \cite[Theorem 1.6]{liwang2024heatII}. Curvature and volume control on an initial ball of radius $r_0$ give a curvature bound $(\epsilon r_0)^{-2}$ on a smaller fixed spatial ball for
$-1\le t<\min\{-1+(\epsilon r_0)^2,0\}$, with $\epsilon$ depending only on the dimension and the chosen volume ratio.

Fix a positive volume ratio below the Euclidean one and choose $r_0$ with $(\epsilon r_0)^2>2$. Since the reference end is AC, taking $a$ sufficiently large makes every ball of radius $r_0$ centered on $\Sigma_a$ satisfy the hypotheses, with strict margins. Enlarge the compact comparison region to contain these balls. The same bounds then hold uniformly around $\Sigma_{i,a}$.

Let $\Psi_{i,t}$ be the family defined in \cref{sec:Ricci-flow} for
$(M_i,g_i,f_i)$, so that $\Psi_{i,-1}=\mathrm{id}$ and
$\partial_t\Psi_{i,t}=(-t)^{-1}\nabla f_i\circ\Psi_{i,t}$. The associated
Ricci flow is
\begin{equation}\label{fik:eq:selfsimilar}
 g_i(t)=(-t)\Psi_{i,t}^*g_i,\qquad t<0.
\end{equation}
Pseudolocality gives uniform curvature bounds on a fixed tube about $\Sigma_{i,a}$ for $g_i(t)$, $-1\le t<0$. On a smaller tube the metrics are uniformly equivalent to their initial metrics. Local derivative estimates, using smooth initial convergence for times close to $-1$, give all curvature derivative bounds on the level set. Rescaling \eqref{fik:eq:selfsimilar} yields
\begin{equation}\label{fik:eq:trajectory-curvature}
 \abs{\nabla^j\Rm_{g_i}}_{g_i}(\Psi_{i,t}(x))
 \le C_j(-t)^{1+j/2},
 \quad x\in\Sigma_{i,a},\quad -1\le t<0.
\end{equation}
The constants are independent of $i$. The uniform initial geometry of the tube and its bounded-curvature evolution also give a uniform positive local injectivity radius there for $g_i(t)$.

The scalar bound $C_0$ in \eqref{fik:eq:trajectory-curvature} can be fixed before increasing $a$; choose $a>2C_0+1$. Along a trajectory,
\begin{equation}\label{fik:eq:potential-trajectory}
 \frac{\mathrm d}{\mathrm d t}f_i(\Psi_{i,t}(x))
 =\frac{f_i(\Psi_{i,t}(x))-\boldsymbol{\mu}(g_i)
 -\mathrm R_{g_i}(\Psi_{i,t}(x))}{-t}.
\end{equation}
Since $\mathrm R_{g_i}\ge0$, integration gives
\begin{equation}\label{fik:eq:trajectory-potential-bound}
 \frac{a-C_0/2}{-t}\le f_i(\Psi_{i,t}(x))-\boldsymbol{\mu}(g_i)
 \le\frac{a}{-t}.
\end{equation}
These trajectories escape every compact set and have nonvanishing gradient.

Consider
\[
 \Sigma_{i,a}\times[-1,0)\longrightarrow\{f_i-\boldsymbol{\mu}(g_i)\ge a\},
 \qquad (x,t)\longmapsto\Psi_{i,t}(x).
\]
Strict increase of $f_i$ on the trajectories and uniqueness of the flow give injectivity. The map is a local diffeomorphism onto its image. For surjectivity, join any point with $f_i-\boldsymbol{\mu}(g_i)>a$ to the core and take the part of this path after its last crossing of $\Sigma_{i,a}$. The set of reached path parameters is open. It is closed as well: if $\Psi_{i,t_j}(x_j)$ remains in a compact set,
\eqref{fik:eq:trajectory-potential-bound} keeps $t_j$ away from zero, and
compactness of $\Sigma_{i,a}$ gives a convergent subsequence. Thus every point of the exterior is reached.

The exterior is therefore diffeomorphic to $\Sigma_{i,a}\times[-1,0)$
and has no critical points of $f_i$. Extending the core identification
along these trajectories gives a global diffeomorphism. This is also the geometric mechanism in \cite[Section~10]{li2020heat}.

Equations \eqref{fik:eq:trajectory-curvature}--\eqref{fik:eq:trajectory-potential-bound} and \eqref{fik:eq:potential-growth} imply
\begin{equation}\label{fik:eq:uniform-AC-bounds}
 \abs{\nabla^j\Rm_{g_i}}\le C_j(1+d_{g_i}(p_i,\cdot))^{-2-j},
 \qquad \inj_{g_i}\ge c(1+d_{g_i}(p_i,\cdot)).
\end{equation}
For the injectivity bound use the estimate on the fixed tube for $g_i(t)$ and \eqref{fik:eq:selfsimilar}; the scale factor $(-t)^{-1/2}$ is comparable to distance along the trajectory. Uniform compact-core geometry completes these estimates on the whole manifold.

Quadratic curvature decay gives a smooth asymptotic cone and canonical end coordinates; see \cite[Section 2.2]{kotschwar2015rigidity} and \cite[Theorem 3.8 and Section 3.3]{conlon2024classification}. In the shrinking normalization,
\begin{equation}\label{fik:eq:uniform-expansion}
 G_i=g_{C,i}+2\Ric(g_{C,i})+e_i,\qquad
 \abs{\nabla^{C,j}e_i}\le C_jr^{-4-j},
\end{equation}
where $G_i$ is the metric in the canonical chart and
$g_{C,i}=\mathrm d r^2+r^2g_{\Sigma,i}$. The constants can be chosen uniformly by \eqref{fik:eq:uniform-AC-bounds}. The same expansion holds for $\hg$. 

We explain the dependence of these coordinates on the sequence. Identify $\Sigma_{i,a}$ with $\Sigma_a$ as above and use the gradient flow to label trajectories. Their canonical radii are determined by
\begin{align}
 A_i(x)&=a-\int_{-1}^0\mathrm R_{g_i}(\Psi_{i,\tau}(x))\,\mathrm d\tau,\label{fik:eq:radius-seed}\\
 r(\Psi_{i,t}(x))&=\frac{2\sqrt{A_i(x)}}{\sqrt{-t}}.\label{fik:eq:radius-flow}
\end{align}
The function $A_i$ is positive. These formulas give $\nabla f_i(r)=r/2$ and
\[
 f_i(\Psi_{i,t}(x))-\boldsymbol{\mu}(g_i)-\frac14r^2
 =\frac1{-t}\int_t^0\mathrm R_{g_i}(\Psi_{i,\tau}(x))\,\mathrm d\tau
 =O(-t).
\]

All angular derivatives in these formulas are uniformly controlled. Indeed, a variation field satisfies
\[
 \nabla_{\partial_t}V
 =\frac1{-t}\left(\frac12V-\Ric_{g_i}^{\sharp}V\right).
\]
The equation for $(-t)^{1/2}V$ has a uniformly bounded coefficient by
\eqref{fik:eq:trajectory-curvature}. Differentiating the variation equation
and using the higher bounds there gives the estimates at every order.
The integrand in \eqref{fik:eq:radius-seed} is consequently $O(-\tau)$
in each fixed $C^m$ norm on the initial link. Smooth convergence on compact
time intervals in $[-1,0)$ and the uniform tail bound give $A_i\to A$
smoothly. The canonical coordinate changes therefore converge on every
compact annulus.

On the initial link the rescaled tangential metrics satisfy
\[
 \partial_t\bigl((-t)\Psi_{i,t}^*g_i\bigr)
 =-2\Psi_{i,t}^*\Ric_{g_i}.
\]
The right-hand side is uniformly bounded with all fixed-order angular
derivatives. The rescaled metrics are uniformly equivalent, since the
coefficient in the rescaled variation equation is integrable on $[-1,0)$.
Their limits as $t\uparrow0$ are positive definite. Compact-time convergence
and the uniform tail estimate show that, in the compatible canonical charts,
\begin{equation}\label{fik:eq:cone-convergence}
 g_{\Sigma,i}\longrightarrow g_\Sigma\quad\text{smoothly}.
\end{equation}
The uniform expansion \eqref{fik:eq:uniform-expansion} holds in these charts. The Taylor argument in \cref{prop:asymptotic-expansion-of-shrinker-metric} gives this expansion with uniform constants from the preceding bounds.

The end identifications converge on a fixed compact collar to the reference identification. A collar adjustment extends them to global maps $\Psi_i$ agreeing with the canonical identifications on the far end and converging smoothly on compact sets.

Subtracting \eqref{fik:eq:uniform-expansion} for the two metrics gives
\begin{equation}\label{fik:eq:difference-expansion}
 \Psi_i^*g_i-\hg=r^2(g_{\Sigma,i}-g_\Sigma)
 +2\bigl(\Ric(g_{C,i})-\Ric(g_C)\bigr)+(e_i-e).
\end{equation}
The first term is the degree-zero tangential coefficient in \cref{eq:slice-E-expansion}, and the degree-minus-one coefficients vanish. The Ricci difference has conic degree $-2$ and tends to zero in each weighted norm by \eqref{fik:eq:cone-convergence}. For each fixed $m$,
\begin{equation}\label{fik:eq:tail}
 \norm{e_i-e}_{C^{m,\alpha}_{-2}(\{r\ge R\})}\le C_mR^{-2},
\end{equation}
uniformly in $i$. The H\"older estimate follows from one additional derivative on enlarged unit slabs. On every compact region these remainders converge smoothly. First choose $R$ large in \eqref{fik:eq:tail} and then let $i\to\infty$ on the remaining compact region. This proves \eqref{fik:eq:E-convergence}. The differentiated potential formula proves \eqref{fik:eq:potential-before-slice}.
\end{proof}

The background has no nonzero parallel one-forms. Indeed, if $\omega$ is
parallel, then $\Ric_{\hg}(\omega^\sharp,\cdot)=0$, and the soliton
equation gives
\[
 \mathrm d\inn{\omega}{\mathbf X}=\frac12\omega.
\]
Thus a nonzero $\omega$ is exact and its complete parallel field splits
off a line. Curvature is constant along each such line, whereas AC
curvature decay forces it to tend to zero. The shrinker would therefore
be flat, hence Gaussian, contrary to nontriviality.

The Poincar\'e estimate \cref{lem:Poincare-inequality} gives
\begin{equation}\label{fik:eq:multiplier}
 \norm{|\mathbf{X}|T}_{L^2_{\hat f}}\le C\norm{T}_{H^1_{\hat f}}.
\end{equation}
By \cref{lem:Gaussian-Rellich}, $H^1_{\hat f}\hookrightarrow L^2_{\hat f}$
is compact. The absence of parallel one-forms therefore implies
\begin{equation}\label{fik:eq:oneform-gap}
 \norm{\omega}_{L^2_{\hat f}}\le C\norm{\hat\nabla\omega}_{L^2_{\hat f}}.
\end{equation}
Indeed, a sequence with unit $L^2_{\hat f}$ norm and derivatives tending
to zero would converge strongly to a nonzero parallel one-form.

\begin{Lemma}\label{fik:lem:linear}
Assume $\operatorname{ISD}=\{0\}$. Suppose $h,u$ are locally $C^2$ on
$M$, $h\in H^2_{\hf}$, $\hat\nabla u\in H^1_{\hf}$, and
$\hat\delta_{\hf}h=0$. With $v=\frac12\tr_{\hg}h-u$, the first-variation
formula gives
\[
 \mathrm D\Phi_{(\hg,\hf)}(h,u)=\hat Lh+\hat\nabla^2v.
\]
Then
\begin{equation}\label{fik:eq:linear-estimate}
 \norm{h}_{H^2_{\hf}}+\norm{\hat\nabla u}_{H^1_{\hf}}
 \le C\norm{\mathrm D\Phi_{(\hg,\hf)}(h,u)}_{L^2_{\hf}}.
\end{equation}
\end{Lemma}

\begin{proof}
The first-variation formula follows from \cite[Lemma~4.4]{li2023rigidity}.
It involves at most two classical derivatives and therefore applies
at the stated local regularity. The commutation identities below hold
distributionally by local approximation; the Gaussian integration by
parts uses density in weighted Sobolev spaces.
The weighted Lichnerowicz operator preserves $\ker\hat\delta_{\hf}$.
By $\operatorname{ISD}=\{0\}$ and
\cite[Theorem~3.19]{li2023rigidity},
\begin{equation}\label{fik:eq:h-coercive}
 \norm{h}_{H^2_{\hf}}\le C\norm{\hat Lh}_{L^2_{\hf}}.
\end{equation}
Approximating $\mathrm d v$ in $H^1_{\hf}$ by compactly supported
one-forms and using $\hat\delta_{\hf}\hat Lh=0$ distributionally gives
\[
 \inn{\hat Lh}{\hat\nabla^2v}_{L^2_{\hf}}=0.
\]
Consequently,
\[
 \norm{\hat Lh}_{L^2_{\hf}}^2+\norm{\hat\nabla^2v}_{L^2_{\hf}}^2
 =\norm{\mathrm D\Phi_{(\hg,\hf)}(h,u)}_{L^2_{\hf}}^2.
\]
Apply \eqref{fik:eq:oneform-gap} to $\mathrm d v$, and use
$u=\frac12\tr h-v$ and \eqref{fik:eq:h-coercive}. This proves the claim.
\end{proof}

\begin{Proposition}[Rigidity estimate on the slice]\label{fik:prop:rigidity}
Assume $\operatorname{ISD}=\{0\}$. There exist $\eta>0$ and $C$ with the
following property. If $h,u$ are locally $C^2$ on $M$,
$\hat\delta_{\hf}h=0$, $h\in H^2_{\hf}$, $\hat\nabla u\in H^1_{\hf}$, and
$\norm{h}_{C^2(\hg)}\le\eta$, then $\hg+h$ is positive definite and
\begin{equation}\label{fik:eq:rigidity}
 \norm{h}_{H^2_{\hf}}+\norm{\hat\nabla u}_{H^1_{\hf}}
 \le C\norm{\Phi(\hg+h,\hf+u)}_{L^2_{\hf}}.
\end{equation}
In particular, an exact shrinker pair under these hypotheses has $h=0$ and $u$ constant.
\end{Proposition}

\begin{proof}
The first variation on the slice is
\begin{equation}\label{fik:eq:linearization}
 \mathrm D\Phi_{(\hg,\hf)}(h,u)=\hat Lh+\hat\nabla^2(\frac12\tr h-u).
\end{equation}
We estimate the remainder directly. For $g=\hg+h$, let $A_h=\nabla^g-\hat\nabla$. Then
\begin{equation}\label{fik:eq:connection}
 (A_h)_{ij}^{\ell}=\frac12g^{\ell q}
 (\hat\nabla_i h_{jq}+\hat\nabla_j h_{iq}-\hat\nabla_q h_{ij}).
\end{equation}
Denote its linearization by $A_h^{(1)}$, obtained by replacing $g^{\ell q}$ by $\hg^{\ell q}$. When $\norm{h}_{C^0}\le1/2$,
\[
 \abs{A_h}\le C\abs{\hat\nabla h},\qquad
 \abs{A_h-A_h^{(1)}}\le C\abs h\abs{\hat\nabla h},
\]
and
\[
 \abs{\hat\nabla(A_h-A_h^{(1)})}\le C\bigl(\abs h\abs{\hat\nabla^2h}+\abs{\hat\nabla h}^2\bigr).
\]
The Ricci difference is a contraction of $\hat\nabla A_h$ and $A_h*A_h$. Its remainder after subtracting the linear term is bounded by
$C(\abs h\abs{\hat\nabla^2h}+\abs{\hat\nabla h}^2)$.
For the potential, the exact identity is
\[
 \nabla_g^2(\hf+u)=\hat\nabla^2\hf+\hat\nabla^2u-A_h\cdot\mathrm d\hf-A_h\cdot\mathrm d u.
\]
Therefore $\mathcal R:=\Phi(\hg+h,\hf+u)-\mathrm D\Phi_{(\hg,\hf)}(h,u)$ satisfies
\begin{equation}\label{fik:eq:pointwise-remainder}
 \abs{\mathcal R}\le C\left(
 \abs h\abs{\hat\nabla^2h}+\abs{\hat\nabla h}^2
 +\abs h\abs{\hat\nabla h}\abs{\mathbf{X}}+\abs{\hat\nabla h}\abs{\hat\nabla u}\right).
\end{equation}
Applying \eqref{fik:eq:multiplier} to $T=\hat\nabla h$ gives
\begin{equation}\label{fik:eq:remainder}
 \norm{\mathcal R}_{L^2_{\hf}}\le C\norm h_{C^1}
 \left(\norm h_{H^2_{\hf}}+\norm{\hat\nabla u}_{H^1_{\hf}}\right).
\end{equation}
Combine \eqref{fik:eq:linear-estimate} and \eqref{fik:eq:remainder}, and choose $\eta$ so that the last term can be absorbed. This proves \eqref{fik:eq:rigidity}.
\end{proof}

\begin{Remark}\label{rmk:AC-rigidity-Li-Zhang}
Li--Zhang \cite[Theorem~4.11]{li2023rigidity}, with $k=2$, give the
corresponding estimate on a non-flat Ricci shrinker with bounded curvature
and $\operatorname{ISD}=\{0\}$:
\[
 \norm h_{H^2_{\hat f}}+\norm{\hat\nabla u}_{H^1_{\hat f}}
 \le C\left(\norm{\Phi(\hg+h,\hf+u)}_{L^2_{\hat f}}
       +\norm{\hat\delta_{\hat f}h}_{H^1_{\hat f}}\right),
\]
provided $\norm h_{C^2}+\norm u_{C^2}$ is sufficiently small.
In \cref{fik:prop:rigidity}, the exact slice removes the divergence term,
and only the metric perturbation is required to be small; the potential
difference need only satisfy $\hat\nabla u\in H^1_{\hat f}$. This is useful
here because the slice diffeomorphism can produce a linearly growing
potential difference; see \cref{fik:potential-growth-slice}.
\end{Remark}

\begin{proof}[Proof of \cref{thm:AC-rigidity}]
Suppose the conclusion fails. There is a sequence of complete
$n$-dimensional Ricci shrinkers $(M_i,g_i,f_i,p_i)$, based at minimum
points of $f_i$, converging pointed Gromov--Hausdorff to
$(M,\hg,\hat p)$, with no $(M_i,g_i)$ isometric to $(M,\hg)$.
The Euclidean tangent spaces of the limit and
\cite[Proposition 5.8]{liliwang2021structure} give a uniform entropy lower
bound. In the normalization of \cite{liliwang2021structure}, the
potentials are $\widetilde f_i:=f_i-\boldsymbol\mu(g_i)$ and satisfy
$\mathrm R_{g_i}+|\mathrm d\widetilde f_i|_{g_i}^2=\widetilde f_i$.
By \cite[Theorem~1.1]{liliwang2021structure}, a subsequence converges
smoothly on compact subsets in this normalization. Moreover,
\cite[Theorem~1.2(c)]{liliwang2021structure} gives convergence of the
total weighted measures. Since
\[
    (4\pi)^{-n/2}\int_{M_i}e^{-\widetilde f_i}\,\mathrm dV_{g_i}
    =e^{\boldsymbol\mu(g_i)},
\]
the constants $\boldsymbol\mu(g_i)$ converge, and the mass-normalized
potentials $f_i$ converge smoothly on compact subsets to a
mass-normalized potential. The limiting potential equals $\hf$:
its difference from $\hf$ has parallel
gradient, which vanishes by the absence of parallel one-forms proved
above, and normalization fixes the constant.

Apply \cref{fik:lem:comparison} to obtain global diffeomorphisms $\Psi_i$
with $h_i:=\Psi_i^*g_i-\hg\to0$ in $E_{10}$. Fix $k=6$ in
\cref{ic:finite-slice}; for large $i$ it gives
$Z_i\in\mathscr X_6$ and $F_i=F_{Z_i}$. Set
\begin{equation}\label{fik:final-pair}
 F_i^*\Psi_i^*g_i=\hg+k_i,\qquad
 F_i^*\Psi_i^*f_i=\hf+u_i.
\end{equation}
Then $k_i\in E_5$, $\hat\delta_{\hat f}k_i=0$, and
$\|k_i\|_{C^2}\to0$. By \cref{ic:map-estimates}, write on the end
\[
 R_i:=r\circ F_i=r+a_i(\theta)+q_i,\qquad q_i=O_{6+\alpha}(r^{-1}),
\]
where $a_i\in C^{8,\alpha}(\Sigma)$.
The canonical potential expansions \cref{fik:eq:potential-before-slice}
have the form
\[
\begin{gathered}
 \Psi_i^*f_i=\boldsymbol\mu(g_i)+\frac14r^2+e_i,
 \qquad \hf=\boldsymbol\mu(\hg)+\frac14r^2+e,\\
 e_i,e=O_\infty(r^{-2}).
\end{gathered}
\]
Substitution gives the exact identity
\[
\begin{split}
 u_i={}&\frac12ra_i+\boldsymbol\mu(g_i)-\boldsymbol\mu(\hg)\\
       &+\frac12rq_i+\frac14(a_i+q_i)^2+e_i\circ F_i-e.
\end{split}
\]
The terms $rq_i$ and $(a_i+q_i)^2$ are $O_{6+\alpha}(1)$ by the
weighted product estimate. The radii $R_i$ and $r$ are comparable,
and the normalized maps $F_i$ have bounded $C^{8,\alpha}$ norms.
The chain rule therefore gives $e_i\circ F_i=O_{6+\alpha}(r^{-2})$.
Together with the constant difference of the entropies, these bounds give
\begin{equation}\label{fik:potential-growth-slice}
 u_i=\frac12ra_i(\theta)+O_{6+\alpha}(1),\qquad
 |\hat\nabla^ju_i|\le C_i(1+r)^{1-j}\quad(0\le j\le2).
\end{equation}
Indeed, the first expansion belongs to $C^{6,\alpha}_1$ on the end,
so its first two covariant derivatives have the displayed bounds.
Local $C^{8,\alpha}$ regularity of $F_i$ supplies the same conclusion
on the compact core. Consequently $k_i\in H^2_{\hat f}$ and
$\hat\nabla u_i\in H^1_{\hat f}$. The maps $F_i$ are locally
$C^{8,\alpha}$, so simultaneous pullback preserves the classical
soliton equation and normalized mass. Thus
$\Phi(\hg+k_i,\hf+u_i)=0$. The finite-regularity estimate in
\cref{fik:prop:rigidity} gives $k_i=0$.

The map $\Psi_i\circ F_i$ is a $C^{8,\alpha}$ local isometry between
smooth metrics. In local coordinates its second derivatives satisfy
\[
 \partial_j\partial_\ell F^a
 =\hat\Gamma^q_{j\ell}\partial_qF^a
  -\Gamma^a_{bc}(F)\partial_jF^b\partial_\ell F^c.
\]
Successive differentiation gives smoothness. It is globally invertible
by \cref{ic:finite-slice}, hence a smooth isometry, a contradiction.
\end{proof}

\begin{Remark}\label{rmk:Gaussian-rigidity}
The Gaussian shrinker is rigid in the space of Ricci shrinkers equipped
with the pointed Gromov--Hausdorff topology. This was proved by
Yokota \cite{yokota2009reduced,yokota2012addendum} and independently by
Li--Wang \cite{li2020heat} using entropy-gap arguments.
\end{Remark}

We now apply the theorem to the four-dimensional FIK
shrinker \cite[Theorem~1.5]{feldman2003rotationally} on
$\mathcal O(-1)\to\mathbb{CP}^1$.

\Needspace{4\baselineskip}
\begin{Corollary}[Pointed Gromov--Hausdorff rigidity of FIK]\label{fik:rigidity}
Let $(M^4,\hg,\hf)$ be the FIK shrinker and
$\hat p\in\arg\min\hf$. There is a pointed Gromov--Hausdorff neighborhood
of $(M,\hg,\hat p)$ such that every complete four-dimensional
Ricci shrinker $(N,g,f,p)$ in this neighborhood, based at a minimum point
of $f$, has $(N,g)$ isometric to $(M,\hg)$.
\end{Corollary}

\begin{proof}
Naff--Ozuch \cite[Theorem~1]{naff2025linear} prove strict negativity of
the entropy Hessian on tensors orthogonal to the Ricci tensor and the
infinitesimal diffeomorphism directions. On this subspace the Hessian is
represented by $\hat L$; their weighted Lichnerowicz operator is $2\hat L$.
If $h\in\operatorname{ISD}$, self-adjointness and
$\hat L\Ric_{\hg}=\frac12\Ric_{\hg}$ imply $h\perp\Ric_{\hg}$.
Weighted divergence-freeness gives orthogonality to infinitesimal
diffeomorphisms. Strict negativity then forces $h=0$.
Thus $\operatorname{ISD}=\{0\}$, and \cref{thm:AC-rigidity} applies.
\end{proof}

\clearpage
\section*{Index of notation and terminology}
\phantomsection
\addcontentsline{toc}{section}{Index of notation and terminology}
\begingroup
\fontsize{10}{11.3}\selectfont
\setlength{\columnsep}{2em}
\setlength{\parindent}{0pt}
\setlength{\parskip}{0pt}
\newcommand{\indexpage}[1]{\hyperref[#1]{\textcolor{linkblue}{\pageref*{#1}}}}
\newcommand{\notationentry}[2]{\par\hangindent=1em\hangafter=1 #1,~#2\par}
\begin{multicols}{2}
\raggedcolumns
\raggedright
\notationentry{AC rigidity}{\indexpage{thm:AC-rigidity}}
\notationentry{Adapted connection $\nabla^{\mathrm a}$}{\indexpage{lem:slice-product-affine-connection}}
\notationentry{Adapted diffeomorphism $F_Z$}{\indexpage{eq:slice-adapted-diffeomorphism}}
\notationentry{Adapted exponential $\operatorname{Exp}^{\mathrm a}$}{\indexpage{eq:slice-product-exponential}}
\notationentry{Angular operator $\mathcal A_{\Sigma}$}{\indexpage{lem:homogeneous-drift-operator-expansion}}
\notationentry{Associated Ricci flow $g(t)=(-t)\Psi_t^*g$}{\indexpage{idx:associated-Ricci-flow}}
\notationentry{Asymptotically conical (AC)}{\indexpage{sec:AC-shrinkers}}
\notationentry{$C^{k,\alpha}_\beta$ (weighted H\"older space)}{\indexpage{def:weighted-Holder-spaces}}
\notationentry{$C^{k,\alpha}_{\beta,\nabla f}$ (drift derivative space)}{\indexpage{def:cone-holder-space-general-component}}
\notationentry{$C^{k,\alpha}_{\mathrm{cyl}}$ (uniform cylindrical norm)}{\indexpage{eq:uniform-cylindrical-Holder-norm}}
\notationentry{$C^{k,\alpha}_{\Lambda;\beta}$ (cone H\"older space)}{\indexpage{def:cone-holder-space}}
\notationentry{$C^{k,\alpha}_{\Lambda;\beta,\nabla f}$ (cone drift space)}{\indexpage{def:cone-holder-space}}
\notationentry{$C^{k,\alpha}_{\Lambda;-\infty}$ (pure homogeneous sums)}{\indexpage{def:cone-holder-space}}
\notationentry{Canonical end chart}{\indexpage{thm:canonical-asymptotic-chart}}
\notationentry{Cone dilations $D_\lambda$}{\indexpage{idx:cone-metric}}
\notationentry{Cone H\"older spaces}{\indexpage{def:cone-holder-space}}
\notationentry{Cone metric $g_C$}{\indexpage{idx:cone-metric}}
\notationentry{$\mathcal D$ (induced connection on $E_\Sigma$)}{\indexpage{lem:homogeneous-drift-operator-expansion}}
\notationentry{Degree of a tensor}{\indexpage{def:degree}}
\notationentry{$\delta_f$, $\delta^*$}{\indexpage{def:weighted-conventions}}
\notationentry{$\mathrm d$ (exterior derivative)}{\indexpage{idx:conventions}}
\notationentry{$E_\Sigma$ (cone tensor bundle on the link)}{\indexpage{idx:radial-parallel-identification}}
\notationentry{Entropy $\boldsymbol\mu(g)$}{\indexpage{eq:normalization-of-potential-second}}
\notationentry{FIK rigidity}{\indexpage{fik:rigidity}}
\notationentry{Gauge at infinity}{\indexpage{thm:canonical-asymptotic-chart}}
\notationentry{Gauge operator $\mathcal P=\delta_f\delta^*$}{\indexpage{eq:gauge-operator}}
\notationentry{Gaussian weight $\mathrm dV_f=(4\pi)^{-n/2}e^{-f}\mathrm dV_g$}{\indexpage{def:weighted-conventions}}
\notationentry{$H^k_f$ (Gaussian Sobolev space)}{\indexpage{idx:weighted-sobolev}}
\notationentry{Infinitesimal solitonic deformations $\operatorname{ISD}$}{\indexpage{fik:stability-operator}}
\notationentry{Intrinsic link Laplacian $\bar\Delta$}{\indexpage{lem:homogeneous-drift-operator-expansion}}
\notationentry{Killing one-forms $\mathcal K$}{\indexpage{eq:Gaussian-Killing-space}}
\notationentry{$L^2_f$ (Gaussian $L^2$ space)}{\indexpage{idx:weighted-sobolev}}
\notationentry{$\hat L$ (stability operator)}{\indexpage{fik:stability-operator}}
\notationentry{Nontrivial AC shrinker}{\indexpage{idx:nontrivial-AC}}
\notationentry{$O_{j+\alpha}(r^\beta)$ (finite-order conic bounds)}{\indexpage{idx:conic-O}}
\notationentry{$O_\infty(r^\beta)$ (all-orders conic bounds)}{\indexpage{idx:conic-O}}
\notationentry{$\Phi(g,f)$ (soliton defect)}{\indexpage{fik:defect}}
\notationentry{$\Psi_t$ (soliton diffeomorphisms)}{\indexpage{idx:associated-Ricci-flow}}
\notationentry{$\mathrm R$ (scalar curvature)}{\indexpage{eq:normalization-of-potential-second}}
\notationentry{Ricci shrinker}{\indexpage{eq:Ricci-shrinker}}
\notationentry{Soliton defect}{\indexpage{fik:defect}}
\notationentry{Standard AC setting}{\indexpage{def:standard-AC-setting}}
\notationentry{$\boldsymbol\tau$ (radial parallel extension)}{\indexpage{eq:parallel-transport-along-radial-direction}}
\notationentry{$\boldsymbol\tau_r$ (radial fiber identification)}{\indexpage{idx:radial-parallel-identification}}
\notationentry{$\mathscr T_\beta$ (weighted cylindrical representative)}{\indexpage{eq:compactified-weighted-representative}}
\notationentry{$u_\gamma$ (link coefficient)}{\indexpage{def:cone-holder-space}}
\notationentry{$[u]_\gamma$ (homogeneous component)}{\indexpage{def:cone-holder-space}}
\notationentry{$u^\star$ (general component)}{\indexpage{def:cone-holder-space}}
\notationentry{Weight $\boldsymbol\rho$}{\indexpage{eq:positive-potential-weight}}
\notationentry{Weighted Sobolev spaces}{\indexpage{idx:weighted-sobolev}}
\notationentry{$\mathbf X$ (canonical radial field)}{\indexpage{idx:canonical-radial-field}}
\notationentry{$E_m$ (metric perturbation space)}{\indexpage{def:slice-radial-filtration}}
\notationentry{$\mathscr S_k$ (scalar sum)}{\indexpage{ic:scalar-sum}}
\notationentry{$\mathscr X_k,\mathscr Y_k$ (gauge graph pair)}{\indexpage{ic:graph-definition}}
\notationentry{$G_{1/2}$ (scalar Gaussian inverse)}{\indexpage{subsec:scalar-terminal-estimates}}
\notationentry{Terminal scalar trace}{\indexpage{ic:scalar-terminal}}
\notationentry{Degree-zero source lift}{\indexpage{ic:degree-zero-terminal}}
\notationentry{Whole-field gauge estimates}{\indexpage{ic:graph-estimates}}
\notationentry{Radial graph estimates}{\indexpage{ic:radial-graph}}
\notationentry{Forward gauge map $P(h,Z)$}{\indexpage{ic:residual-closure}}
\notationentry{Fixed Banach slice}{\indexpage{idx:fixed-Banach-slice}}
\notationentry{Sliced radial improvement}{\indexpage{ic:sliced-radial-improvement}}
\notationentry{$L_\lambda=\Delta_f-\lambda$}{\indexpage{ic:Llambda}}
\end{multicols}
\endgroup
\clearpage

\phantomsection
\addcontentsline{toc}{section}{References}
\bibliographystyle{amsalpha-linked}
\begingroup
\fontsize{10.5}{12.5}\selectfont
\bibliography{references}
\endgroup
\par\vspace{10pt}
\noindent\begin{minipage}{\textwidth}
Yu Li, Institute of Geometry and Physics, University of Science and Technology
of China, No.~96 Jinzhai Road, Hefei, Anhui Province, 230026, China;
Hefei National Laboratory, No.~5099 West Wangjiang Road, Hefei, Anhui Province,
230088, China; E-mail: \href{mailto:yuli21@ustc.edu.cn}{yuli21@ustc.edu.cn}.

\medskip
Longhuan Yan, School of Mathematical Sciences, University of Science and
Technology of China, No.~96 Jinzhai Road, Hefei, Anhui Province, 230026, China;
E-mail: \href{mailto:yanlonghuan@mail.ustc.edu.cn}{yanlonghuan@mail.ustc.edu.cn}.

\medskip
Wenjia Zhang, School of Mathematical Sciences, University of Science and
Technology of China, No.~96 Jinzhai Road, Hefei, Anhui Province, 230026, China;
E-mail: \href{mailto:wj12345678@mail.ustc.edu.cn}{wj12345678@mail.ustc.edu.cn}.
\end{minipage}\par

\end{document}